\documentclass[11pt,leqno]{amsart}

\usepackage{amsmath}
\usepackage{amssymb}
\usepackage{a4wide}
\usepackage{bbm}
\usepackage{graphics}
\usepackage{epsfig}
\usepackage{color}
\usepackage{mathrsfs}
\usepackage[sans]{dsfont}

\newtheorem{theorem}{Theorem}[section]
\newtheorem{lemma}[theorem]{Lemma}
\newtheorem{proposition}[theorem]{Proposition}
\newtheorem{definition}[theorem]{Definition}
\newtheorem{remark}[theorem]{Remark}
\newtheorem{corollary}[theorem]{Corollary}
\newtheorem{example}[theorem]{Example}

\newcounter{hypo}
\newcounter{hypoa}

\newcounter{hypoaa}
\newcounter{hypobb}

\newenvironment{hypaa}{  \begin{enumerate} \setcounter{enumi}{\value{hypoaa}} \item}{\stepcounter{hypoaa} \end{enumerate}}

\def\C{{\mathbb C}}

\def\N{{\mathbb N}} 
\def\R{{\mathbb R}}

\def\S{{\mathbb S}}

\def\CA{\mathcal {A}}
\def\CB{\mathcal {B}}

\def\CE{\mathcal {E}}

\def\CI{{\mathcal I}}

\def\CO{\mathcal {O}}

\def\CQ{\mathcal {Q}}
\def\CR{\mathcal {R}}
\def\CS{\mathcal {S}}

\def\SC{\mathscr {C}}
\def\SE{\mathscr {E}}
\def\SF{\mathscr {F}}
\def\SG{\mathscr {G}}

\def\SM{\mathscr {M}}

\def\one{\mathds{1}}
\def\re{\mathop{\rm Re}\nolimits}
 \def\im{\mathop{\rm Im}\nolimits}

\def\I{\mathop{\rm I}\nolimits}
\def\Op{\mathop{\rm Op}\nolimits}

\newcommand{\tr}{\operatorname{tr}}

\newcommand{\chsupp}{\operatorname{ch\; supp}}
\newcommand{\Res}{\operatorname{Res}}

\newcommand{\supp}{\operatorname{supp}}
\newcommand{\spe}{\operatorname{sp}}

\def\arg{\mathop{\rm arg}\nolimits}
\def\dist{\mathop{\rm dist}\nolimits}

\def\<{\langle}
\def\>{\rangle}

\def\res{\mathop{\rm Res}\nolimits}

\def\ds{\displaystyle}

\newcommand{\fract}[2]{\genfrac{}{}{0pt}{}{\scriptstyle #1}{\scriptstyle #2}}

\makeatletter
 \@addtoreset{equation}{section}
 \makeatother
 
\title{Applications of resonance theory without analyticity assumption}

\author[J.-F. Bony]{Jean-Fran\c{c}ois Bony}
\address{Jean-Fran\c{c}ois Bony, IMB, UMR CNRS 5251, Universit\'e de Bordeaux, 33405 Talence, France}
\email{bony@math.u-bordeaux.fr}
\author[L. Michel]{Laurent Michel}
\address{Laurent Michel, IMB, UMR CNRS 5251, Universit\'e de Bordeaux, 33405 Talence, France}
\email{lamichel@math.u-bordeaux.fr}
\author[T. Ramond]{Thierry Ramond}
\address{Thierry Ramond, Laboratoire de Math\'ematiques d'Orsay, Univ. Paris-Sud, CNRS, Universit\'e Paris-Saclay, 91405 Orsay, France}
\email{thierry.ramond@math.u-psud.fr}


\begin{document}

\begin{abstract}
We prove that the results in scattering theory that involve resonances are still valid for non-analytic potentials, even if the notion of resonance is not defined in this setting. More precisely, we show that if the potential of a semiclassical Schr\"{o}dinger operator is supposed to be smooth and to decrease at infinity, the usual formulas relating scattering quantities and resonances still hold. The main ingredient for the proofs is a resolvent estimate of a new type, relating the resolvent of an operator with the resolvent of its cut-off counterpart.
\end{abstract}

\maketitle

\section{Introduction} \label{s1}

In this paper, we consider semiclassical Schr\"{o}dinger operators $P$ on $L^{2} ( \R^{n} )$, $n \geq 1$,
\begin{equation}\label{a1}
P = - h^{2} \Delta + V ( x ),
\end{equation}
where the potential $V$ is a real-valued smooth function. In this setting, the resonances of $P$ near the real axis are usually defined through the analytic distortion method due to Aguilar and Combes \cite{AgCo71_01} and Hunziker \cite{Hu86_01}. For this reason, the potential is supposed to be analytic outside of a compact set and to vanish at infinity. More precisely, $V \in C^{\infty} ( \R^{n} ; \R )$ is assumed to extend holomorphically in the sector 
\begin{equation} \label{b60}
\CS = \big\{ x \in \C^{n} ; \ \vert \re x \vert > C \text{ and } \vert \im x \vert \leq \delta \vert x \vert \big\} ,
\end{equation}
for some $C , \delta > 0$, and $V ( x ) \to 0$ as $x \to \infty$ in $\CS$. Under this assumption, one can define the distorted operator $P_{\theta}$ of angle $\theta > 0$ small enough. Its spectrum is discrete in $\CE_{\theta} = \{ z \in \C ; \ - 2 \theta < \arg z \leq 0 \}$ and the resonances of $P$ are the eigenvalues of $P_{\theta}$ in $\CE_{\theta}$. The resonances do not depend on $\theta$, and $\res ( P )$ denote their set. Methods close to that of analytic distortions have been developed by Helffer and Sj\"{o}strand \cite{HeSj86_01}, Sj\"{o}strand and Zworski \cite{SjZw91_01}, Lahmar-Benbernou and Martinez \cite{LaMa02_01} or Sj\"{o}strand \cite{Sj07_01}. They all require the analyticity of $V$ at infinity (see \eqref{b60}), or even globally, that is in the set
\begin{equation} \label{b61}
\CS_{\rm g} = \big\{ x \in \C^{n} ; \ \vert \im x \vert \leq \delta \< x \> \big\} ,
\end{equation}
for some $\delta > 0$. These different methods give the same resonances when they can be applied simultaneously (see Helffer and Martinez \cite{HeMa87_01}). A general presentation of resonance theory can be found in the books of Sj\"{o}strand \cite{Sj07_01} or Dyatlov and Zworski \cite{DyZw16_01}.

Nevertheless, when the potential is not analytic, various approaches allow to define the resonances. For exponentially decreasing potentials, that is $\vert V ( x ) \vert \lesssim e^{- \delta \vert x \vert}$ with $\delta > 0$, the cut-off resolvent
\begin{equation*}
z \longmapsto ( P - z )^{- 1} : L^{2}_{\rm comp} ( \R^{n} ) \to L^{2}_{\rm loc} ( \R^{n} ) ,
\end{equation*}
has a meromorphic extension from the upper complex half-plane to a neighborhood of the real axis. The resonances are then the poles of this extension. Such an idea goes back to Dolph, McLeod and Thoe \cite{DoMcTh66_01}. In the perturbation regime also, it is possible to define a time-dependent notion of resonances. This setting is devoted to the study of (non semiclassical) selfadjoint operators of the form $Q = Q_{0} + \kappa W$ where $Q_{0}$ has an embedded eigenvalue $\lambda_{0}$ with normalized eigenvector $\psi_{0}$ in its essential spectrum and $\vert \kappa \vert \ll 1$. Under the Fermi golden rule condition, one can show that the quantum evolution $\< \psi_{0} , e^{- i t Q} \psi_{0} \>$ behaves like $e^{- i t \lambda}$ for times $t$ controlled by $\kappa^{- 1}$ where $\lambda \approx \lambda_{0}$ is the resonance. This has been proved for instance by Orth \cite{Or90_01}, Soffer and Weinstein \cite{SoWe98_01} or Jensen and Nenciu \cite{JeNe06_01}. In the semiclassical regime, this idea has been followed by G{\'e}rard and Sigal \cite{GeSi92_01}. They define a quasiresonant state associated to a quasiresonance $z$ to be any reasonable function $u$ such that
\begin{equation*}
\left\{ \begin{aligned}
&( P - z ) u = \CO ( h^{\infty} ) \text{ in } L^{2}_{\rm loc} ( \R^{n} ) , \\
&u \text{ is an outgoing function.}
\end{aligned} \right.
\end{equation*}
Finally, Cancelier, Martinez and the third author \cite{CaMaRa05_01} and Martinez, Sj\"{o}strand and the third author \cite{MaRaSj09_01} have generalized the method of analytic distortions for potentials which are not analytic but only $C^{\infty}$. For that they construct a family of potentials $V_{C}$, $C \to + \infty$, which approximate $V$ and satisfy \eqref{b60} with the same constant $C$. The resonances of $P_{C} = - h^{2} \Delta + V_{C} ( x )$ are given by an analytic distortion and the resonances of $P$ are defined as the accumulation points of $\res ( P_{C} )$ as $C \to + \infty$. Summing up, for nonanalytic and non exponentially decreasing potentials, there exist notions of resonances, but they are not as indisputable as in the analytic setting.

In this paper, we will follow a different path. Instead of trying to define the resonances for nonanalytic semiclassical Schr\"{o}dinger operators, we want to generalize the applications of resonance theory to slowly decaying $C^{\infty}$ potentials. In other words, we want to replace automatically the analyticity assumption by a $C^{\infty}$ assumption in the results of scattering theory using the resonances. The precise symbol assumption \ref{h1} on $V$ as well as the general notations used throughout the paper are given in the next section.

The first application that we consider is the resonance expansion of the quantum propagator. This type of result says typically that, for $\chi \in C^{\infty}_{0} ( \R^{n} )$,
\begin{equation} \tag{A} \label{b62}
\chi e^{- i t P / h} \chi = \sum_{z \in \res ( P )} e^{- i t z / h} \chi \Pi_{z} \chi + \text{remainder} ,
\end{equation}
where $\Pi_{z}$ is the generalized spectral projector associated to the resonance $z$. Such a formula corresponds to the Dirichlet formula in the setting of eigenvalues. The resonance expansion of the wave group was first obtained by Lax and Phillips \cite{LaPh67_01} in the exterior of a star-shaped obstacle. It was then generalized to various non-trapping settings (see for instance Va{\u\i}nberg \cite{Va89_01} and the references of the second edition of \cite{LaPh67_01}). Concerning trapping regimes, there are two types of results: the first ones are valid without assumption on the trapping, while the latter, more precise, treat specific captures. Among the general results, one can cite Tang and Zworski \cite{TaZw00_01}, Stefanov \cite{St01_01} or Burq and Zworski \cite{BuZw01_01}. On the other hand, precise resonance expansions have been proved by Nakamura, Stefanov and Zworski \cite{NaStZw03_01} in the ``well in the island'' situation, by Fujii\'e, Zerzeri and two authors \cite{BoFuRaZe11_01} at barrier-top, by Dyatlov \cite{Dy12_01} for the wave equation in general relativity, \ldots Note that the cut-off function $\chi$ is compactly supported and so does not ``see'' the analyticity of $P$ at infinity in \eqref{b62}. Then, it is natural to hope that this formula still holds true in the $C^{\infty}$ setting. This is done in Section \ref{s2} under \ref{h1} only. In the formula obtained, the resonances $z$ are those of the Schr\"{o}dinger operator with potential $V$ truncated at infinity. In particular, we extend the resonance expansions known for potentials analytic at infinity.

The second application of resonance theory concerns the Breit--Wigner formula for the derivative of the spectral shift function. The precise definition of the spectral shift function (SSF) associated to the pair $( P , P_{0} )$, denoted  $s_{P , P_{0}} ( \lambda )$, can be found in \eqref{a73}. The Breit--Wigner formula says typically
\begin{equation} \tag{B} \label{b63}
s_{P , P_{0}}^{\prime} ( \lambda ) = \sum_{z \in \Res ( P )} \frac{\vert \im z \vert}{\pi \vert \lambda - z \vert^{2}} + \text{remainder} .
\end{equation}
Such a result has been proved by Melrose \cite{Me88_01} in the exterior of obstacles and by G{\'e}rard, Martinez and Robert \cite{GeMaRo89_01} in the well in the island situation. General formulas (that is without assumption on the trapping) have been successively proved by Petkov and Zworski \cite{PeZw99_01,PeZw00_01}, Sj\"{o}strand and the first author \cite{BoSj01_01}, Bruneau and Petkov \cite{BrPe03_01}, \ldots In Section \ref{s3}, we generalize \eqref{b63} and the results previously cited to the $C^{\infty}$ setting, more precisely under assumption \ref{h1} with $\rho > n$. As before, the resonances appearing in the formula we obtain are those of the Schr\"{o}dinger operator with potential $V$ truncated at infinity. Nevertheless, the new remainder term can no longer be $\CO ( h^{\infty} )$ since infinity gives a non-negligible contribution to the derivative of the SSF. It was not the case for the cut-off quantum evolution. As an application, we study the transitional regime of $s_{P , P_{0}}^{\prime} ( \lambda )$ at the energy level of a homoclinic trajectory between strong trapping (huge peaks) and non-trapping (smooth behavior).

As a last application of resonance theory, we study the scattering amplitude. The precise definition of this function, denoted $S_{P , P_{0}} ( \lambda , \theta , \omega )$, can be found at the beginning of Section \ref{s4}. For potentials analytic at infinity, it is known that
\begin{equation} \tag{C} \label{b65}
\lambda \longmapsto S_{P , P_{0}} ( \lambda , \theta , \omega ) \text{ has a meromorphic extension},
\end{equation}
from $\R$ to $\CE_{\theta}$ with poles at the resonances. It has first been proved by Lax and Phillips \cite{LaPh67_01} for obstacle scattering and generalized to different situations (see e.g. Agmon \cite{Ag86_01}). In the semiclassical limit, the residue of the scattering amplitude has been computed in the well in the island situation by Lahmar-Benbernou and Martinez \cite{Be99_01,LaMa99_01} and at barrier-top by Fujii\'e, Zerzeri and two authors \cite{BoFuRaZe11_01}. For $C^{\infty}$ potentials, we prove in Section \ref{s4} that \eqref{b65} still holds true, modulo a function which is $\CO ( h^{\infty} )$ on the real axis. Compared to the previous applications of resonance theory, the values at infinity of the potential is crucial in the behavior of the scattering amplitude and the results obtained for potentials analytic at infinity can not be directly transposed to the $C^{\infty}$ setting. Nevertheless, we show in Theorem \ref{b23} that the dependence of the scattering matrix with respect to a change of potential at infinity can be followed using natural Fourier integral operators.

To prove the results summarized previously, we use an intermediary operator analytic at infinity and equal to $P$ in an appropriate region. Concretely, let $\chi \in C^{\infty}_{0} ( \R^{n} )$ and
\begin{equation*}
Q = - h^{2} \Delta + W ( x ) ,
\end{equation*}
where $W = V$ in a sufficiently large domain containing the support of $\chi$ and the trapped set. Then, the cut-off resolvents of $P$ and $Q$ are almost the same. More precisely,
\begin{equation} \tag{D} \label{b64}
\chi ( P - \lambda \pm i 0 )^{- 1} \chi = \chi ( Q - \lambda \pm i 0 )^{- 1} \chi + \text{remainder} ,
\end{equation}
for $\lambda \in I \Subset ] 0 , + \infty [$. This estimate is obtained without assumption on the trapping in Section \ref{s5} for $W ( x ) = V ( x ) g_{0} ( x / R )$ with $R \gg 1$ and $g_{0}$ a plateau function. The proof is based on the constructions of Isozaki and Kitada and the general estimates on the resolvent of Burq. When the cut-off resolvent of $P$ is polynomially bounded, \eqref{b64} is proved in Section \ref{s6} under the assumption that $V$ and $W$ coincide near the trapped set. This proof only uses $C^{\infty}$ microlocal analysis.

Equation \eqref{b64} allows to extend the formulas \eqref{b62}, \eqref{b63} and \eqref{b65} to the $C^{\infty}$ setting the following way. For a Schr\"{o}dinger operator $T$, let $\SM ( T )$ be a quantum quantity which has an asymptotic expansion in terms of the resonances of $T$ when its potential is analytic at infinity, say
\begin{equation*}
\SM ( T ) = \SF \big( \res ( T ) \big) .
\end{equation*}
In the cases we consider, there exists also a representation formula which allows to write
\begin{equation*}
\SM ( T ) = \SG \big( \chi ( T - \lambda \pm i 0 )^{- 1} \chi \big) .
\end{equation*}
Note that $W$, the potential of $Q$, can be chosen compactly supported and then analytic at infinity. Thus, the previous equations combined with \eqref{b64} give
\begin{align*}
\SM ( P ) &= \SG \big( \chi ( P - \lambda \pm i 0 )^{- 1} \chi \big) \\
&= \SG \big( \chi ( Q - \lambda \pm i 0 )^{- 1} \chi \big) + \text{remainder} \\
&= \SM ( Q ) + \text{remainder} \\
&= \SF \big( \res ( Q ) \big) + \text{remainder} ,
\end{align*}
showing that $\SM$ has an asymptotic expansion in terms of resonances for $C^{\infty}$ potentials.

\section{General setting and notations}

In this short section, we collect some definitions and hypotheses made throughout the paper. Instead of assuming that $V$ is analytic at infinity (see \eqref{b60}), we suppose that
\begin{hypaa} \label{h1}
$V \in C^{\infty} ( \R^{n} ; \R )$ and there exists $\rho > 0$ such that, for all $\alpha \in \N^{n}$,
\begin{equation*}
\vert \partial_{x}^{\alpha} V ( x ) \vert \lesssim \< x \>^{- \rho - \vert \alpha \vert} .
\end{equation*}
\end{hypaa}
Thanks to the Cauchy formula, this assumption holds true for any potential $V$ which is $C^{\infty} ( \R^{n} ; \R )$, analytic at infinity in the sense of \eqref{b60} and satisfies $\vert V ( x ) \vert \lesssim \< x \>^{- \rho}$ for $x \in \CS$.

We only use the ordinary notations and some basic results of semiclassical microlocal analysis. For a clear presentation of this theory, we send back the reader for example to the textbooks of Dimassi and Sj\"ostrand \cite{DiSj99_01}, Martinez \cite{Ma02_02} and Zworski \cite{Zw12_01}. In particular, $S ( m )$ denotes the set of symbols controlled by the weight function $m$, the semiclassical pseudodifferential operator of symbol $a \in S ( m )$ is defined by
\begin{equation*}
( \Op ( a ) u ) ( x ) = \frac{1}{( 2 \pi h)^{n}} \iint e^{i ( x - y ) \cdot \xi / h} a \Big( \frac{x + y}{2} , \xi , h \Big) u( y ) \, d y \, d \xi ,
\end{equation*}
and $\Psi ( m ) = \Op ( S ( m ) )$ is the set of pseudodifferential operators of symbols in $S ( m )$.

For $V$ satisfying \ref{h1}, we let $p ( x , \xi ) = \xi^{2} + V ( x ) \in S ( \< \xi \>^{2} )$ denote the symbol of $P$, i.e. $P = \Op ( p )$. Its associated Hamiltonian vector field is
\begin{equation*}
H_{p} = \partial_{\xi} p \cdot \partial_{x} - \partial_{x} p \cdot \partial_{\xi} = 2 \xi \cdot \partial_{x} - \nabla V ( x ) \cdot \partial_{\xi} .
\end{equation*}
Integral curves $t \longmapsto \exp ( t H_{p} )( x , \xi )$ of $H_{p}$ are called Hamiltonian trajectories, and $p$ is constant along such curves. The trapped set at energy $E$ for $P$ is defined by
\begin{equation} \label{a92}
K_{p} ( E ) = \big\{ ( x , \xi ) \in p^{- 1} ( E ) ; \ t \longmapsto \exp ( t H_{p} ) ( x , \xi ) \text{ is bounded} \big\} .
\end{equation}
For $E > 0$, $K ( E )$ is compact and stable by the Hamiltonian flow.

\section{Resonance expansion of the propagator} \label{s2}

The most expected application of resonance theory is the expansion of the truncated quantum propagator in terms of the resonances. Roughly speaking, for $\chi \in C^{\infty}_{0} ( \R^{n} )$ and $\varphi \in C^{\infty}_{0} ( ] 0 , + \infty [ )$, such a result writes 
\begin{equation} \label{b68}
\chi e^{- i t P / h} \varphi ( P ) \chi = \sum_{\fract{z \in \Res ( P )}{z \text{ close to } \supp \varphi}} e^{- i t z / h} \chi \Pi_{z} \chi + \text{remainder} ,
\end{equation}
for $t \geq 0$ in an appropriate interval. In other words, \eqref{b68} quantifies the exponential decay of the local energy for the quantum evolution $e^{- i t P / h}$ which is unitary. Then, the resonances (resp. resonant states) can be seen as quasi-eigenvalues (resp. metastable states). We send back the reader to the introduction and the corollaries of the present section for references concerning resonance expansion of the propagator.

In this part, we obtain formulas like \eqref{b68} for Schr\"{o}dinger operators $P$ satisfying only \ref{h1} without the analyticity hypothesis. For that, we first consider another Schr\"{o}dinger operator
\begin{equation*}
Q = - h^{2} \Delta + W ( x ) ,
\end{equation*}
satisfying also \ref{h1} such that $V$ and $W$ coincide in a bounded region, and try to show that
\begin{equation} \label{a4}
\chi e^{- i t P / h} \varphi ( P ) \chi = \chi e^{- i t Q / h} \varphi ( Q ) \chi + \CO ( h^{\infty} ) ,
\end{equation}
for $t$ real in some interval. Then, choosing $W$ such that a resonance expansion like \eqref{b68} is already known for $Q$ (typically $W$ is analytic at infinity) and applying \eqref{a4}, we obtain a resonance expansion for $P$.

\begin{figure}%[!h]
\begin{center}
\begin{picture}(0,0)%
\includegraphics{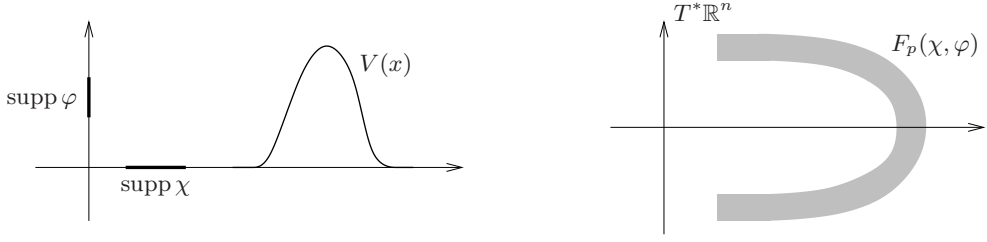}%
\end{picture}%
\setlength{\unitlength}{1105sp}%
\begingroup\makeatletter\ifx\SetFigFont\undefined%
\gdef\SetFigFont#1#2#3#4#5{%
  \reset@font\fontsize{#1}{#2pt}%
  \fontfamily{#3}\fontseries{#4}\fontshape{#5}%
  \selectfont}%
\fi\endgroup%
\begin{picture}(21344,4969)(-1221,-683)
\put(6076,3014){\makebox(0,0)[lb]{\smash{{\SetFigFont{9}{10.8}{\rmdefault}{\mddefault}{\updefault}$V ( x  )$}}}}
\put(-1049,2264){\makebox(0,0)[b]{\smash{{\SetFigFont{9}{10.8}{\rmdefault}{\mddefault}{\updefault}$\supp \varphi$}}}}
\put(1501,314){\makebox(0,0)[b]{\smash{{\SetFigFont{9}{10.8}{\rmdefault}{\mddefault}{\updefault}$\supp \chi$}}}}
\put(18001,3464){\makebox(0,0)[lb]{\smash{{\SetFigFont{9}{10.8}{\rmdefault}{\mddefault}{\updefault}$F_{p} ( \chi , \varphi )$}}}}
\put(13126,4139){\makebox(0,0)[lb]{\smash{{\SetFigFont{9}{10.8}{\rmdefault}{\mddefault}{\updefault}$T^{*} \R^{n}$}}}}
\end{picture}%
\end{center}
\caption{An example of potential $V$ and the associated set $F_{p} ( \chi , \varphi )$.} \label{f2}
\end{figure}

Following this general strategy, we want to prove \eqref{a4} for two operators $P = - h^{2} \Delta + V ( x )$ and $Q = - h^{2} \Delta + W ( x )$ satisfying \ref{h1} and two cut-off functions $\chi \in C^{\infty}_{0} ( \R^{n} )$ and $\varphi \in C^{\infty}_{0} ( ] 0 , + \infty [ )$. For simplicity, we suppose that $\pi_{x} ( K_{p} ( \supp \varphi ) ) \Subset \supp \chi$. It is natural to assume that $V$ and $W$ coincide on the support of $\chi$, but this hypothesis may not be enough, even for finite time. Indeed, there may be Hamiltonian curves of energy in $\supp \varphi$, starting above $\supp \chi$ and coming back above $\supp \chi$ after a finite time. Thus, we define
\begin{equation*}
F_{p} ( \chi , \varphi ) = \big\{ ( x , \xi ) \in T^{*} \R^{n} ; \ \exp ( t_{\pm} H_{p} ) ( x , \xi ) \in \supp ( \chi \varphi ( p ) ) \text{ for some } \pm t_{\pm} \geq 0 \big\} ,
\end{equation*}
the set of these curves (see Figure \ref{f2}). It is a compact subset of $T^{*} \R^{n}$ containing $\supp ( \chi \varphi ( p ) )$. This set can be seen as the convex hull of $\supp ( \chi \varphi ( p ) )$ for the Hamiltonian curves of $p$, but may be not connected. We assume that
\begin{hypaa} \label{h2}
$\pi_{x} ( K_{p} ( \supp \varphi ) \cup K_{q} ( \supp \varphi ) ) \Subset \supp \chi$, $F_{p} ( \chi , \varphi ) = F_{q} ( \chi , \varphi )$ and $p = q$ near $F_{p} ( \chi , \varphi )$.
\end{hypaa}
This assumption implies that $K_{p} ( \supp \varphi ) = K_{q} ( \supp \varphi )$ and $V = W$ near $\supp \chi$. Since we work with Schr\"{o}dinger operators, the last part of \ref{h2} is equivalent to $V = W$ near $\pi_{x} ( F_{p} ( \chi , \varphi ) )$, the base space projection of $F_{p} ( \chi , \varphi )$. Our first result is

\begin{proposition}[Comparison of evolutions for short time]\sl \label{a2}
Let $P , Q$ satisfying \ref{h1}, $\chi \in C^{\infty}_{0} ( \R^{n} )$ and $\varphi \in C^{\infty}_{0} ( ] 0 , + \infty [ )$ be such that \ref{h2} holds true. For all $C > 0$, we have
\begin{equation*}
\chi e^{- i t P / h} \varphi ( P ) \chi = \chi e^{- i t Q / h} \varphi ( Q ) \chi + \CO ( h^{\infty} ) ,
\end{equation*}
uniformly for $\vert t \vert \leq h^{- C}$.
\end{proposition}

The proof of Proposition \ref{a2} as well as that of Remark \ref{a47} and that of Proposition \ref{a38} are postponed to Section \ref{s6}. In order to compare the evolutions for all time, we will now assume that the weighted resolvent of $P$ is polynomially bounded. More precisely, for $\varphi \in C^{\infty}_{0} ( ] 0 , + \infty [ )$, we assume that
\begin{hypaa} \label{h4}
There exist $C > 0$ and $s > 1 / 2$ such that
\begin{equation*}
\big\Vert \< x \>^{- s} ( P - \lambda \pm i 0 )^{- 1} \< x \>^{- s} \big\Vert \lesssim h^{- C} ,
\end{equation*}
uniformly for $\lambda$ near the support of $\varphi$.
\end{hypaa}
From Proposition \ref{a13}, if the assumption \ref{h4} is satisfied for some $s > 1 / 2$, it holds true for all $s > 1 / 2$ with the same constant $C$. Moreover, it is equivalent to assume \ref{h4} for $P$ or for $Q$. More precisely,

\begin{remark}\sl \label{a47}
Let $P , Q$ satisfying \ref{h1}, $\chi \in C^{\infty}_{0} ( \R^{n} )$ and $\varphi \in C^{\infty}_{0} ( ] 0 , + \infty [ )$ be such that \ref{h2} holds true. Then, $P$ satisfies \ref{h4} if and only if $Q$ satisfies \ref{h4}. Moreover, for all $s > 1 / 2$, we have in that case
\begin{equation*}
\big\Vert \< x \>^{- s} ( Q - \lambda \pm i 0  )^{- 1} \< x \>^{- s} \big\Vert \lesssim \big\Vert \< x \>^{- s} ( P - \lambda \pm i 0  )^{- 1} \< x \>^{- s} \big\Vert \lesssim \big\Vert \< x \>^{- s} ( Q - \lambda \pm i 0  )^{- 1} \< x \>^{- s} \big\Vert ,
\end{equation*}
uniformly for $\lambda$ near the support of $\varphi$.
\end{remark}

Remark \ref{a47} is essentially due to Datchev and Vasy \cite{DaVa12_02}. But it is not clear how to verify their geometric assumptions in our setting and to adapt directly their result. This is why we give a self-contained proof of Remark \ref{a47} in Section \ref{s6} (note that \eqref{a50} and \eqref{a52} are similar to \cite[Lemma 3.1]{DaVa12_02}).

\begin{proposition}[Comparison of evolutions under a polynomial estimate]\sl \label{a38}
Let $P , Q$ satisfy \ref{h1}, $\chi \in C^{\infty}_{0} ( \R^{n} )$ and $\varphi \in C^{\infty}_{0} ( ] 0 , + \infty [ )$ be such that \ref{h2} and \ref{h4} hold true. Then,
\begin{equation*}
\chi e^{- i t P / h} \varphi ( P ) \chi = \chi e^{- i t Q / h} \varphi ( Q ) \chi + \CO ( h^{\infty} ) ,
\end{equation*}
uniformly for $t \in \R$.
\end{proposition}

With Propositions \ref{a2} and \ref{a38} in mind, one may ask if \eqref{a4} holds true for all time $t \in \R$ under the assumptions \ref{h1} and \ref{h2} only. It happens that such a statement is false.

\begin{remark}\sl \label{a58}
There exist $\chi \in C^{\infty}_{0} ( \R^{n} )$, $\varphi \in C^{\infty}_{0} ( ] 0 , + \infty [ )$ and operators $P , Q$ satisfying \ref{h1} and \ref{h2} such that \eqref{a4} does not hold uniformly for $t \in \R$. Of course, \ref{h4} is not satisfied in that case.
\end{remark}

\begin{figure}%[!h]
\begin{center}
\begin{picture}(0,0)%
\includegraphics{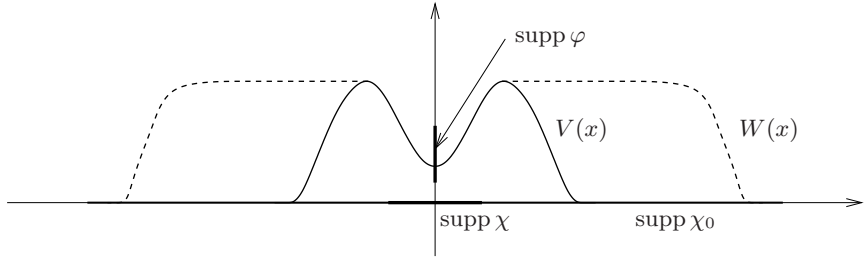}%
\end{picture}%
\setlength{\unitlength}{1105sp}%
\begingroup\makeatletter\ifx\SetFigFont\undefined%
\gdef\SetFigFont#1#2#3#4#5{%
  \reset@font\fontsize{#1}{#2pt}%
  \fontfamily{#3}\fontseries{#4}\fontshape{#5}%
  \selectfont}%
\fi\endgroup%
\begin{picture}(19244,5744)(-9621,-383)
\put(6826,2339){\makebox(0,0)[lb]{\smash{{\SetFigFont{9}{10.8}{\rmdefault}{\mddefault}{\updefault}$W ( x  )$}}}}
\put(2701,2339){\makebox(0,0)[lb]{\smash{{\SetFigFont{9}{10.8}{\rmdefault}{\mddefault}{\updefault}$V ( x  )$}}}}
\put(1801,4439){\makebox(0,0)[lb]{\smash{{\SetFigFont{9}{10.8}{\rmdefault}{\mddefault}{\updefault}$\supp \varphi$}}}}
\put(901,314){\makebox(0,0)[b]{\smash{{\SetFigFont{9}{10.8}{\rmdefault}{\mddefault}{\updefault}$\supp \chi$}}}}
\put(5401,314){\makebox(0,0)[b]{\smash{{\SetFigFont{9}{10.8}{\rmdefault}{\mddefault}{\updefault}$\supp \chi_{0}$}}}}
\end{picture}%
\end{center}
\caption{The setting in the proof of Remark \ref{a58}.} \label{f3}
\end{figure}

\begin{proof}[Proof of Remark \ref{a58}]
We first consider a smooth compactly supported potential $V$ as in Figure \ref{f3}. Using Theorem \ref{a15} below and \cite[Proposition D.1]{BoFuRaZe18_01} which allows to compare the truncated resolvent of $P = - h^{2} \Delta + V$ with the resolvent of the distorted operator $P_{\theta}$, we have $\partial_{z} \chi ( P - z )^{- 1} \chi = \partial_{z} \chi ( P_{\theta} - z )^{- 1} \chi = \chi ( P_{\theta} - z )^{- 2} \chi =\CO ( e^{3 C_{0} / h} )$, for $z = \lambda \pm i 0$, $\lambda \in \supp \varphi$. Then, an integration by parts in the Stone formula (see \eqref{a7}) yields
\begin{equation} \label{a56}
\chi e^{- i t P / h} \varphi ( P ) \chi = - \frac{1}{2 \pi} \frac{h}{t} \int_{\R} e^{- i t \lambda / h} \partial_{\lambda} \big( \varphi ( \lambda ) F_{P} ( \lambda ) \big) \, d \lambda = \CO \Big( h \frac{e^{3 C_{0} / h}}{t} \Big) = \CO ( h^{\infty} ) ,
\end{equation}
uniformly for $t \geq e^{4 C_{0} / h}$.

On the other hand, we choose another potential $W \in C^{\infty}_{0} ( \R^{n} )$ as in Figure \ref{f3} and we set $Q = - h^{2} \Delta + W$. In particular, $W$ coincides with $V$ near the support of $\chi$. We also ask that $W$ satisfies the geometric assumptions of Fujii{\'e}, Lahmar-Benbernou and Martinez \cite{FuLaMa11_01} (see Helffer and Sj{\"o}strand \cite[Chapitre 10]{HeSj86_01} in the analytic case). In particular, the minimum in the well is non-degenerate and we denote by $S_{0}$ the Agmon distance from this minimum to the sea. Under these assumptions, Theorem 2.3 of \cite{FuLaMa11_01} gives that the imaginary part of the first resonance $\rho ( h )$ (that is the resonance closest to energy of the bottom of the well) satisfies
\begin{equation*}
\im \rho \sim - f_{0} h^{\frac{1 - n_{\Gamma}}{2}} e^{- 2 S_{0} / h} ,
\end{equation*}
for some constant $f_{0} > 0$ and $0 \leq n_{\Gamma} \leq n - 1$. Eventually, we assume that the island is large enough so that $2 S_{0} - 1 > 4 C_{0}$. Let $u$ be a resonant state for $Q$ associated to the resonance $\rho$ and normalized in the island. From Theorem 4.3 of G{\'e}rard and Sigal \cite{GeSi92_01} and since $u$ is exponentially small outside the well, we have
\begin{equation} \label{a57}
\chi e^{- i t Q / h} \varphi ( Q ) \chi u = \chi e^{- i t Q / h} \chi_{0} u + \CO ( h^{\infty} ) = e^{- i t \rho / h} \chi u + \CO ( h^{\infty} ) ,
\end{equation}
uniformly for $0 \leq t \leq e^{( 2 S_{0} - 1 ) / h}$. The result of G{\'e}rard and Sigal requires the analyticity of the potential near the whole $\R^{n}$ since it relies on some estimates of Helffer and Sj{\"o}strand \cite{HeSj86_01}. But we can use \cite{FuLaMa11_01} instead to extend this result to potentials analytic at infinity.

In the present geometric setting, we have $F_{p} ( \chi , \varphi ) = F_{q} ( \chi , \varphi ) = \supp ( \chi \varphi ( p ) )$ and \ref{h2} is satisfied. Moreover, \eqref{a56} and \eqref{a57} imply
\begin{equation}
\big\Vert \chi e^{- i t P / h} \varphi ( P ) \chi u \big\Vert = \CO ( h^{\infty} ) \qquad \text{and} \qquad \big\Vert \chi e^{- i t Q / h} \varphi ( Q ) \chi u \big\Vert = 1 + \CO ( h^{\infty} ) ,
\end{equation}
uniformly for $e^{4 C_{0} / h} \leq t \leq e^{( 2 S_{0} - 1 ) / h}$. This prove that \eqref{a4} does not hold for such times.
\end{proof}

From the previous discussion, we have to be more specific about the operator $Q$ in order to have a general result and get rid of \ref{h4}. Then, let $g_{0} \in C^{\infty}_{0} ( \R^{n} ; [ 0 , 1 ] )$ be such that $g_{0} = 1$ near $B ( 0 , 1 )$. We assume that
\begin{hypaa} \label{h3}
$Q = - h^{2} \Delta + W ( x )$ with $\ds W ( x ) = g_{0} \Big( \frac {x}{R} \Big) V ( x )$,
\end{hypaa}
where $R > 1$ is a large constant which will be fixed in the sequel. Under this hypothesis, our main result is the following.

\begin{theorem}[General comparison of evolutions]\sl \label{a3}
Assume \ref{h1} and \ref{h3}. Let $\chi \in C^{\infty}_{0} ( \R^{n} )$ and $\varphi \in C^{\infty}_{0} ( ] 0 , + \infty [ )$. For $R > 0$ large enough, we have
\begin{equation*}
\chi e^{- i t P / h} \varphi ( P ) \chi = \chi e^{- i t Q / h} \varphi ( Q ) \chi + \CO ( h^{\infty} ) ,
\end{equation*}
uniformly for $t \in \R$.
\end{theorem}

The $R$'s for which this result holds true may depend on $\chi , \varphi$ and $P$. Note also that \ref{h2} is satisfied for $R$ large enough. The proof of Theorem \ref{a3} relies on the estimate of the difference of the resolvents given in Theorem \ref{a6} below and on a polynomial bound on the integral of the resolvent obtained by Stefanov.

\begin{proof}
By the Stone formula, we have
\begin{equation} \label{a7}
\chi e^{- i t P / h} \varphi ( P ) \chi = \frac{1}{2 \pi i} \int_{\R} e^{- i t \lambda / h} \varphi ( \lambda ) F_{P} ( \lambda ) \, d \lambda ,
\end{equation}
where
\begin{align*}
F_{P} ( \lambda ) &= \chi \lim_{\varepsilon \to 0} \big( (P - \lambda - i \varepsilon )^{-1} - (P - \lambda + i \varepsilon )^{-1} \big) \chi  \\
&= \chi ( P - \lambda - i 0 )^{-1} \chi - \chi ( P - \lambda + i 0 )^{-1} \chi ,
\end{align*}
thanks to the limiting absorption principle. On the other hand, Theorem \ref{a6} gives for $R$ large enough
\begin{equation*}
\chi ( P - \lambda \pm i 0 )^{-1} \chi = \chi ( Q - \lambda \pm i 0 )^{-1} \chi + \CO ( h^{\infty} ) \big\Vert \< x \>^{- 1} ( Q - \lambda - i 0 )^{-1} \< x \>^{- 1} \big\Vert ,
\end{equation*}
uniformly for $\lambda \in \supp \varphi$. Thus, \eqref{a7} implies
\begin{equation} \label{a8}
\chi e^{- i t P / h} \varphi ( P ) \chi = \chi e^{- i t Q / h} \varphi ( Q ) \chi + \CO ( h^{\infty} ) \int_{\supp \varphi} \big\Vert \< x \>^{- 1} ( Q - \lambda - i 0 )^{-1} \< x \>^{- 1} \big\Vert \, d \lambda .
\end{equation}
Since $Q$ is a compactly supported perturbation of $- h^{2} \Delta$, one can use Proposition 3 of Stefanov \cite{St01_01} which gives that
\begin{equation} \label{a9}
\int_{\supp \varphi} \big\Vert \< x \>^{- 1} ( Q - \lambda - i 0 )^{-1} \< x \>^{- 1} \big\Vert \, d \lambda \lesssim h^{- \frac{5 n}{2} - 3} .
\end{equation}
More precisely, this result is originally stated with cut-off functions instead of weights $\< x \>^{- 1}$ and in the high frequency regime. But, the first point can be overcome thanks to Proposition \ref{a13}. On the other hand, the argument of Stefanov can be directly adapted to our semiclassical regime and provides the upper bound \eqref{a9} in that case (taking the $\varepsilon$ there equal to $1 / 4$). Eventually, Theorem \ref{a3} follows from \eqref{a8} and \eqref{a9}.
\end{proof}

\begin{remark}\sl \label{a72}
In \ref{h3}, we have chosen to stick the potential $V$ to $0$, but we could have sticked it to another potential. More precisely, one can replace $W ( x )$ in \ref{h3} by
\begin{equation*}
 W ( x ) = g_{0} \Big( \frac {x}{R} \Big) V ( x ) + ( 1 - g_{0} ) \Big( \frac {x}{R} \Big) \widetilde{V} ( x ) ,
\end{equation*}
where $\widetilde{V}$ satisfies \ref{h1}. The choice $\widetilde{V} = 0$ we have made is the most natural for the applications to the resonance theory.
\end{remark}

We have chosen to state Theorem \ref{a3} for Schr\"{o}dinger operators $P = - h^{2} \Delta + V ( x )$ but this result may be generalized to more general operators. For instance, one could consider elliptic selfadjoint differential operators of order two close to $- h^{2} \Delta$ at infinity, smooth compact obstacles with connected complement and even manifolds with different types of ends. The important ingredients used in the proof are a polynomial bound of the resolvent truncated at infinity, and that what goes to infinity does not come back.

Theorem \ref{a3} allows to extend straightly all the resonance expansions already obtained for operators with compactly supported potential to the setting of operators with $C^{\infty}$ potential. In the rest of this section, we give some examples of such applications. We start with a general result showing that the propagator associated to any operator satisfying \ref{h1} has a resonance expansion in long time. For that, we use a theorem of Burq and Zworski \cite{BuZw01_01}. Instead, we could have considered Tang and Zworski \cite{TaZw00_01}, Stefanov \cite{St01_01}, \ldots

\begin{corollary}[General resonance expansion]\sl \label{a11}
Assume \ref{h1} and let $\chi \in C^{\infty}_{0} ( \R^{n} )$, $\varphi \in C^{\infty}_{0} ( ] 0 , + \infty [ )$ with $[ a , b ] = \chsupp \varphi$, $0 < \delta < a / 4$ and $M > 1$ be large enough. Then, there exist $L , R > 1$, $\delta < c ( h ) < 2 \delta$ and $d ( h ) = \CO ( h )$ such that
\begin{equation} \label{a59}
\chi e^{- i t P / h} \chi \varphi ( P ) = \sum_{z \in \Omega ( h ) \cap \res ( Q )} \chi \CR \big( e^{- i t \bullet / h} ( \bullet - Q )^{- 1} , z \big) \chi \varphi ( Q ) + \CO ( h^{\infty} ),
\end{equation}
uniformly for $t > h^{- L}$ where $Q$ is as in \ref{h3}, $\res ( Q )$ is the set of its resonances,
\begin{equation*}
\Omega ( h ) = \big] a - c ( h ) , b + c ( h ) \big[ - i \big[ 0 , h^{M} ( 1 + d ( h ) ) \big[ ,
\end{equation*}
and $\CR ( f ( \bullet ) , z )$ denotes the residue of a meromorphic family of operator $f$ at $z$.
\end{corollary}

The constant $R$ can be chosen arbitrarily large and independent of $M$. The small functions $c ( h ) $ and $d ( h )$ guaranty that $\partial \Omega ( h ) \setminus \R$, the ``remainder'' part of the boundary of $\Omega$, is away from the resonances of $Q$. Finally, $\chsupp \varphi$ denotes the convex hull of the support of $\varphi$.

\begin{proof}
Let $\widetilde{\varphi} \in C^{\infty}_{0} ( \R ; [ 0 , 1 ] )$ be such that $\varphi \prec \widetilde{\varphi} \prec \one_{] 0 , + \infty [}$. From the pseudodifferential calculus and Theorem \ref{a3}, we have
\begin{align}
\chi e^{- i t P / h} \chi \varphi ( P ) &= \chi e^{- i t P / h} \widetilde{\varphi} ( P ) \chi \varphi ( P ) + \CO ( h^{\infty} )  \nonumber \\
&= \chi e^{- i t Q / h} \widetilde{\varphi} ( Q ) \chi \varphi ( P ) + \CO ( h^{\infty} )  \nonumber \\
&= \chi e^{- i t Q / h} \widetilde{\varphi} ( Q ) \chi \varphi ( Q ) + \CO ( h^{\infty} )  \nonumber \\
&= \chi e^{- i t Q / h} \chi \varphi ( Q ) + \CO ( h^{\infty} ) , \label{a10}
\end{align}
uniformly for $t \in \R$. In the previous equalities, $R > 1$ is chosen large enough such that $V = W$ near the support of $\chi$ and such that Theorem \ref{a3} holds true. Then, Corollary \ref{a11} is a direct consequence of \eqref{a10} and of Theorem 1 of Burq and Zworski \cite{BuZw01_01}. Here, we use that the Schr\"{o}dinger operator with compactly supported potential $Q$ enters into the setting of \cite{BuZw01_01}.
\end{proof}

Resonance expansion may be seen as the generalization of the Dirichlet formula for eigenvalues to the theory of resonances. For $P$ with compact resolvent, this formula gives
\begin{equation*}
e^{- i t P / h} = \sum_{\lambda \in \spe ( P )} e^{- i t \lambda / h} \Pi_{\lambda} ,
\end{equation*}
where $\Pi_{\lambda}$ is the spectral projector associated to the eigenvalue $\lambda$ of $P$. However, Theorem \ref{a3} does not have a counterpart for eigenvalues. Indeed, assume that
\begin{equation*}
e^{- i t P / h} \varphi ( P ) = e^{- i t Q / h} \varphi ( Q ) + \CO ( h^{\infty} ) ,
\end{equation*}
holds true with $\supp \varphi \cap( \spe_{\rm ess} ( P ) \cup \spe_{\rm ess} ( Q ) ) =\emptyset$. Then, the eigenvalues of $P$ are exactly those of $Q$ in $\varphi^{- 1} ( 1 )$ for $h$ small enough, which is not reasonable in general. Indeed, assume that $\lambda_{0} \in \varphi^{- 1} ( 1 )$ is an eigenvalue of $P$ but not of $Q$. We can write
\begin{equation*}
\Pi_{\lambda_{0}} ( P ) = \sum_{\lambda \in \spe ( Q )} \varphi ( \lambda ) e^{- i t ( \lambda - \lambda_{0} ) / h} \Pi_{\lambda} ( Q  ) - \sum_{\lambda_{0} \neq \lambda \in \spe ( P )} \varphi ( \lambda ) e^{- i t ( \lambda - \lambda_{0} ) / h} \Pi_{\lambda} ( P ) + r ( t ) ,
\end{equation*}
with $\Vert r ( t ) \Vert \leq 1 / 2$. Integrating with respect to $t \in [ 0 , T ]$ and taking the norm lead to $T \leq C + T / 2$ for some constant $C > 0$ which provides a contradiction for $T$ large. This argument can not be made for resonances because the leading part of the asymptotic (i.e. the sum in \eqref{a59}) is already $\CO ( h^{\infty} )$ at the time $T$ needed to get the contradiction. In other words, the difference between the resonances of $P$ and $Q$ is small compared to the imaginary part of these resonances (under \ref{h3} with $R$ large enough). Then, Theorem \ref{a3} can be seen as ``the stability at infinity of the metastability''.

\begin{figure}%[!h]
\begin{center}
\begin{picture}(0,0)%
\includegraphics{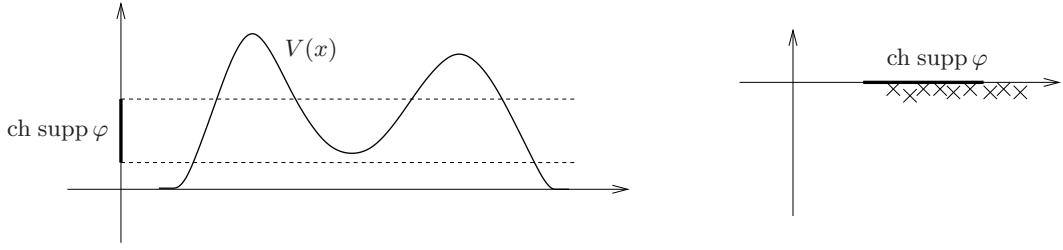}%
\end{picture}%
\setlength{\unitlength}{1105sp}%
\begingroup\makeatletter\ifx\SetFigFont\undefined%
\gdef\SetFigFont#1#2#3#4#5{%
  \reset@font\fontsize{#1}{#2pt}%
  \fontfamily{#3}\fontseries{#4}\fontshape{#5}%
  \selectfont}%
\fi\endgroup%
\begin{picture}(23662,5444)(-2564,-383)
\put(3676,3839){\makebox(0,0)[lb]{\smash{{\SetFigFont{9}{10.8}{\rmdefault}{\mddefault}{\updefault}$V ( x  )$}}}}
\put(-2549,2039){\makebox(0,0)[lb]{\smash{{\SetFigFont{9}{10.8}{\rmdefault}{\mddefault}{\updefault}$\chsupp \varphi$}}}}
\put(17176,3614){\makebox(0,0)[lb]{\smash{{\SetFigFont{9}{10.8}{\rmdefault}{\mddefault}{\updefault}$\chsupp \varphi$}}}}
\end{picture}%
\end{center}
\caption{The geometric setting and the resonances of $Q$ in Example \ref{a61}.} \label{f4}
\end{figure}

\begin{example}\rm \label{a61}
We now apply Theorem \ref{a3} to the well in the island setting. We consider $P$ satisfying \ref{h1} and $\varphi \in C^{\infty}_{0} ( ] 0 , + \infty [ )$. Noting $I = \chsupp \varphi$, we assume that there exist two closed sets, $\Sigma^{e} ( I )$, $\Sigma^{i} ( I )$, such that
\begin{equation*}
p ^{- 1} ( I ) = \Sigma^{e} ( I ) \cup \Sigma^{i} ( I ) , \qquad \Sigma^{e} ( I ) \cap \Sigma^{i} ( I ) = \emptyset, \qquad \Sigma^{i} ( I ) \Subset T^{*} \R^{n} ,
\end{equation*}
and
\begin{equation*}
\rho \in \Sigma^{e} ( I ) \quad \Longrightarrow \quad \exp ( t H_{p}) ( \rho ) \to \infty \text{ as } t \to \pm \infty .
\end{equation*}
In particular, $K_{p} ( I ) = \Sigma^{i} ( I )$. The geometric setting as well as the distribution of resonances when the potential is analytic at infinity (e.g. for $Q$) are illustrated in Figure \ref{f4}. We decompose any $\chi \in C^{\infty}_{0} ( \R^{n} )$ as $\chi = \chi_{1} + \chi_{2}$ with
\begin{equation*}
\pi ( \Sigma^{e} ( I ) ) \cap \supp \chi_{1} = \emptyset , \qquad \pi ( \Sigma^{i} ( I ) ) \cap \supp \chi_{2} = \emptyset ,
\end{equation*}
and $\pi ( x , \xi ) = x$. Combining Nakamura, Stefanov and Zworski \cite{NaStZw03_01} in the case of Schr\"{o}dinger operators with \eqref{a10}, we immediately obtain

\begin{corollary}[Shape resonances]\sl \label{a60}
Let $P , \varphi , \chi$ be as before, $0 < \delta \ll 1$ and $C_{1} > 0$. Then, there exist $R > 1$, $\delta < c ( h ) < 2 \delta$ and $C_{2} > 0$ such that
\begin{align*}
\chi e^{- i t P / h} \chi \varphi ( P ) = \sum_{z \in \Omega ( h ) \cap \res ( Q )} \chi_{1} \CR \big( e^{- i t \bullet / h} & ( \bullet - Q )^{- 1} , z \big) \chi_{1} \varphi ( Q ) \\
&+ \chi_{2} \CO \big( \big\< t - C_{2} ) _{+} / h \big\>^{- \infty} \big) \chi_{2} + \CO ( h^{\infty} ),
\end{align*}
uniformly for $t \geq 0$ where $Q$ is as in \ref{h3}, $\res ( Q )$ is the set of its resonances,
\begin{equation*}
\Omega ( h ) = \chsupp \varphi + ] - c ( h ) , c ( h ) [ - i [ C_{1} h , 0 ] ,
\end{equation*}
and $\CR ( f ( \bullet ) , z )$ denotes the residue of a meromorphic family of operator $f$ at $z$.
\end{corollary}

Thus, we have the same result as in \cite{NaStZw03_01} without the analyticity assumption. A lot is known concerning the asymptotic of shape resonances. First, the resonances of $Q$ in $\Omega ( h )$ are exponentially close to the eigenvalues of the Dirichlet restriction of $Q$ near the well. In particular, their imaginary part is exponentially small. Finally, their precise asymptotic can be obtained under some geometric assumptions. We refer to Helffer and Sj{\"o}strand \cite{HeSj86_01} for potentials analytic near the whole real axis. The case of potentials analytic at infinity (as the potential $W$ defined in \ref{h3} for instance) has been treated by Fujii{\'e}, Lahmar-Benbernou and Martinez \cite{FuLaMa11_01} (see also Nakamura, Stefanov and Zworski \cite{NaStZw03_01}).
\end{example}

\begin{figure}%[!h]
\begin{center}
\begin{picture}(0,0)%
\includegraphics{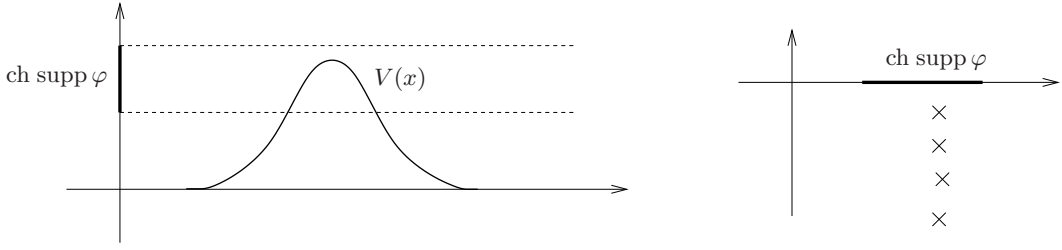}%
\end{picture}%
\setlength{\unitlength}{1105sp}%
\begingroup\makeatletter\ifx\SetFigFont\undefined%
\gdef\SetFigFont#1#2#3#4#5{%
  \reset@font\fontsize{#1}{#2pt}%
  \fontfamily{#3}\fontseries{#4}\fontshape{#5}%
  \selectfont}%
\fi\endgroup%
\begin{picture}(23662,5444)(-2564,-383)
\put(17176,3614){\makebox(0,0)[lb]{\smash{{\SetFigFont{9}{10.8}{\rmdefault}{\mddefault}{\updefault}$\chsupp \varphi$}}}}
\put(-2549,3239){\makebox(0,0)[lb]{\smash{{\SetFigFont{9}{10.8}{\rmdefault}{\mddefault}{\updefault}$\chsupp \varphi$}}}}
\put(5701,3164){\makebox(0,0)[lb]{\smash{{\SetFigFont{9}{10.8}{\rmdefault}{\mddefault}{\updefault}$V ( x  )$}}}}
\end{picture}%
\end{center}
\caption{The geometric setting and the resonances of $Q$ in Example \ref{a62}.} \label{f5}
\end{figure}

\begin{example}\rm \label{a62}
As another application of our result, one can described the quantum propagator truncated near the maximum of the potential. In order to have totally explicit expressions, we work in dimension $n = 1$. We assume that $V$ has a non-degenerate maximum at $x = 0$, i.e.
\begin{equation*}
V (x) = E_{0} - \frac{\lambda^{2}}{4} x^{2} + \CO ( x^{3} ) ,
\end{equation*}
with $E_{0} , \lambda > 0$, and that the trapped set at energy $E_{0}$ satisfies $K ( E_{0} ) = \{ ( 0,0) \}$ (see Figure \ref{f5}). As a consequence, $x = 0$ is the unique global maximum of $V$, and there exists a pointed neighborhood of $E_{0}$ in which all the energy levels are non trapping.

Under these hypotheses, the distribution of the resonances of $Q$, defined in \ref{h3} with $R$ large enough, near $E_{0}$ is known. In any complex neighborhood of $E_0$ of size $h$, these resonances are the
\begin{equation} \label{a64}
z_{k} ( h ) = E_{0} - i h \lambda \Big( \frac{1}{2} + k \Big) + \CO ( h^{2} ) ,
\end{equation}
with $k \in \N$. They actually have a complete expansion in powers of $h$ and are simple. The simplest way to prove this result is to compute the coefficients of the $2 \times 2$ scattering matrix (for a complex energy $z$) and to find their poles. For that, it is enough to propagate through the fixed point $( 0 , 0 )$ the Jost solutions at infinity (that is $e^{\pm i \sqrt{z} x / h}$ since $W$ is compactly supported). We omit the details and just mention that similar results have been obtained under various assumptions by Sj\"{o}strand \cite{Sj87_01}, Briet, Combes and Duclos \cite{BrCoDu87_02}, the third author \cite{Ra96_01} and Fujii\'e, Zerzeri and two authors \cite{BoFuRaZe19_01}.

Let $\Pi_{k}$ denote the generalized spectral projection of $Q$ at the resonance $z_{k}$, that is the residue of the meromorphic extension of the truncated resolvent $( z - Q )^{- 1}$ at $z_{k}$. A direct adaptation of Theorem 4.1 of \cite{BoFuRaZe11_01} to potentials analytic only at infinity gives that
\begin{equation} \label{a65}
\Pi_{k} = c_{k} f_{k} \big\< \overline{f_{k}} , \cdot \big\> .
\end{equation}
Here, the constant $c_{k}$ is given by
\begin{equation*}
c_{k} ( h ) = h^{- k - \frac{1}{2}} \frac{e^{- i \frac{\pi}{2} ( k + \frac{1}{2} )}}{\sqrt{2 \pi} k !} \lambda^{k + \frac{1}{2}} ,
\end{equation*}
and the function $f_{k}$ is a locally uniformly bounded and outgoing solution of the Schr\"{o}dinger equation $( P - z_{k} ) f_{k}= 0$ which can be written $f_{k} ( x , h ) = a_{k} ( x , h ) e^{i \varphi_{+} ( x ) / h}$ near $0$ where $\varphi_{+}$ is the generating function of the outgoing manifold of the Hamiltonian vector field at $( 0 , 0 )$ with $\varphi_{+} ( 0 ) = 0$ and $a_{k}$ is a classical symbol
\begin{equation*}
a_{k} ( x , h ) \sim \sum_{j = 0}^{\infty} a_{j} ( x ) h^{j} ,
\end{equation*}
and $a_{0} ( x ) = x^{k} + \CO ( x^{k + 1} )$, see \cite{BoFuRaZe11_01} for more details. In this situation, we have the

\begin{corollary}[Barrier-top resonances]\sl \label{a63}
Let $P$ be as before, $\chi \in C^{\infty}_{0} ( \R^{n} )$ and $\one_{\{ E_{0} \}} \prec \varphi \in C^{\infty}_{0} ( \R )$ be supported near $E_{0}$. Then, for all $K \in \N$ there exists $C_{K} > 0$ such that
\begin{equation} \label{a66}
\chi e^{- i t P / h} \chi \varphi ( P ) = \sum_{0 \leq k < K} e^{- i t z_{k} / h} \chi \Pi_{k} \chi + \CO \big( e^{- \lambda K t} h^{- C_{K}} \big) + \CO ( h^{\infty} ),
\end{equation}
uniformly for $t \geq 0$.
\end{corollary}

Thus, we remove the analyticity assumption in \cite[Theorem 6.1]{BoFuRaZe11_01} in dimension $n = 1$. Thanks to \eqref{a64} and \eqref{a65}, the sum in the right hand side of \eqref{a66} is explicit. Corollary \ref{a63} is a direct consequence of \eqref{a10} and \cite[Section 6]{BoFuRaZe11_01} adapted to potentials analytic only at infinity. Since the weighted resolvent of $P$ is polynomially bounded, we could have used here Proposition \ref{a38} instead of Theorem \ref{a3}. The assumption that $\varphi = 1$ near $E_{0}$ is used to show that $\varphi ( P ) \chi \overline{f_{k}} = \chi \overline{f_{k}} + \CO ( h^{\infty} )$ and then to remove the energy cut-off in \eqref{a66}.
\end{example}

Until now we only have generalized resonance expansions of the propagator, but other results of the resonance theory prove that $e^{- i t P / h} u$ is close to $e^{- i t z / h} u$ for $u$ resonant state associated to a resonance $z$ of $P$. In the semiclassical regime, this idea comes back to G{\'e}rard and Sigal \cite{GeSi92_01}. As previously, it is possible to extend these results to our $C^{\infty}$ setting. For instance, combining Theorem 3.7 and Remark 4 of \cite{GeSi92_01} with our Theorem \ref{a3}, we obtain

\begin{corollary}[Propagation of resonant states]\sl \label{a12}
Assume \ref{h1}, $\chi \in C^{\infty}_{0} ( \R^{n} )$ and $C , E_{0} > 0$. Let $\chi \prec \widetilde{\chi} \in C^{\infty}_{0} ( \R^{n} )$ be such that $\widetilde{\chi} = 1$ in a large neighborhood of $0$ and let $Q$ be as in \ref{h3} with $R$ large enough. Then, for any resonant state $u = u ( h )$ associated to a resonance $z = z ( h ) \in B ( E_{0} , C h )$ of $Q$ with $\vert \im z \vert \gtrsim h^{C}$, we have
\begin{equation*}
\chi e^{- i t P / h} \widetilde{\chi} u = e^{- i t z / h} \chi u + \CO ( h^{\infty} ) \Vert \widetilde{\chi} u \Vert ,
\end{equation*}
uniformly for $t \geq 0$.
\end{corollary}

Note that the potential is not (always) assumed to be analytic in the paper of G{\'e}rard and Sigal \cite{GeSi92_01}. For example, they treat potential maxima in Section 4.C under the assumption \ref{h1}. The result obtained is different from Corollary \ref{a63}. In their case, the initial data must be a quasimode and the asymptotic holds for all $t \geq 0$. In our case, the initial data is general and the asymptotic holds for $t \gtrsim \vert \ln h \vert$ (in fact, the resonance expansion regime is not valid before the Ehrenfest time, as explained in Remark 6.2 of a paper with Fujii\'e and Zerzeri \cite{BoFuRaZe11_01}). For systems coming from quantum chemistry, Briet and Martinez \cite{BrMa17_01,BrMa19_01} have obtained the resonance expansion of the quantity $t \longmapsto \< e^{- i t P / h} \varphi ( P ) u , u \>$ where $u$ is a quasimode. When their interaction operator $W$ is differential (and not pseudodifferential), it may be possible in their results to remove the analyticity assumption on the coefficients.

\section{Spectral shift function} \label{s3}

In this part, we assume \ref{h1} with $\rho > n$ and denote $P_{0} = - h^{2} \Delta$. Then, the spectral shift function (SSF) associated to the pair $( P , P_{0} )$ is defined as the distribution given by
\begin{equation} \label{a73}
\forall \varphi \in C^{\infty}_{0} ( \R ) , \qquad \big\< s_{P , P_{0}}^{\prime} , \varphi \big\> = \tr ( \varphi ( P ) - \varphi ( P_{0} ) ) ,
\end{equation}
with the convention that $s ( \lambda ) = 0$ for $\lambda \ll - 1$. As a matter of fact, it coincides with the counting function of eigenvalues for $\lambda < 0$. On the other hand, this function is $C^{\infty}$ for $\lambda > 0$. We send back the reader to the book of Yafaev \cite{Ya92_01} and the survey of Robert \cite{Ro99_01} for general informations about the spectral shift function.

Under general assumptions, Robert \cite[Theorem 1.6]{Ro94_01} has proved that the SSF always satisfies a Weyl law
\begin{equation} \label{a75}
s_{P , P_{0}} ( \lambda ) = s_{0} ( \lambda ) h^{- n} + \CO ( h^{1 - n} ) ,
\end{equation}
for $\lambda$ near a noncritical positive energy with
\begin{equation*}
s_{0} ( \lambda ) = \frac{1}{( 2 \pi )^{n}} \int_{\R^{n}} \bigg( \int_{\xi^{2} + V ( x ) \leq \lambda} d \xi - \int_{\xi^{2} \leq \lambda} d \xi \bigg) d x .
\end{equation*}
In this manner, resonance theory is not useful for the asymptotic of the SSF. The situation is rather different for that concerns its derivative. For this function, the Breit--Wigner formula claims that
\begin{equation} \label{b66}
s_{P , P_{0}}^{\prime} ( \lambda ) = \sum_{\fract{z \in \Res ( P )}{z \text{ close to } \lambda}} \frac{\vert \im z \vert}{\pi \vert \lambda - z \vert^{2}} + \text{remainder} .
\end{equation}
As recalled in the introduction, such a formula has been first obtained by Melrose \cite[Section 4]{Me88_01} in obstacle scattering and G{\'e}rard, Martinez and Robert \cite{GeMaRo89_01} in the well in the island situation. In the semiclassical regime and without assumption on the trapping, it has been established successively for potentials analytic at infinity by Petkov and Zworski \cite{PeZw99_01,PeZw00_01}, Bruneau and Petkov \cite{BrPe03_01}, Dimassi and Petkov \cite{DiPe03_01}, \ldots

Here, we show formulas similar to \eqref{b66} and extend the previous results under the assumption that $V$ satisfies only \ref{h1} with $\rho > n$. For that, we follow the general strategy explained at the end of the introduction and already used in Section \ref{s2}. The first step is to prove a comparison result
\begin{equation} \label{b67}
s_{P , P_{0}}^{\prime} ( \lambda ) = s_{Q , P_{0}}^{\prime} ( \lambda ) + \text{remainder} ,
\end{equation}
where $Q = - h^{2} \Delta + W ( x )$ is such that $W$ satisfies \ref{h1} with $\rho > n$ and coincides with $V$ on a sufficiently large region. Afterwards, a Breit--Wigner formula for $P$ can be obtained applying \eqref{b67}, choosing the potential $W$ analytic at infinity and using a Breit--Wigner formula for $Q$ from a previous paper.

We begin with a result like \eqref{b67} which only requires that $P = Q$ near the trapped set and that the weighted resolvent of $P$ is polynomially bounded. More precisely, let $I \subset ] 0 , + \infty [$ be a compact interval. The trapped set $K_{\bullet}$ being defined in \eqref{a92}, we assume that
\begin{hypaa} \label{h5}
The symbols $p , q$ of $P , Q$ satisfy
\begin{equation*}
p = q \text{ near the trapped sets } K_{p} ( I ) = K_{q} ( I ) .
\end{equation*}
\end{hypaa}
Under this assumption, we have the following result.

\begin{proposition}[Comparison of the derivative of SSF under a polynomial estimate]\sl \label{a91}
Let $P$, $Q$ satisfying \ref{h1} with $\rho > n$, \ref{h4} for $\lambda$ near $I$ and \ref{h5}. Then,
\begin{equation*}
s_{P , P_{0}}^{\prime} ( \lambda ) = s_{Q , P_{0}}^{\prime} ( \lambda ) + \sigma ( \lambda ; h ) ,
\end{equation*}
uniformly for $\lambda \in I$, where the function $\sigma ( \lambda ; h )$ has a complete expansion in powers of $h$
\begin{equation*}
\sigma ( \lambda ; h ) \asymp \sigma_{0} ( \lambda ) h^{- n} + \sigma_{1} ( \lambda ) h^{1 - n} + \cdots ,
\end{equation*}
and the functions $\sigma_{\bullet} ( \lambda )$ are $C^{\infty}$ near $I$.
\end{proposition}

The proof of Proposition \ref{a91} is postponed until Section \ref{s6}. The regular symbol $\sigma ( \lambda ; h )$ takes into account somehow the difference at infinity between $V$ and $W$. For instance, if $P$ (and thus $Q$) is non-trapping (and thus noncritical) near $\lambda$, we have
\begin{equation*}
\sigma_{0} ( \lambda ) = \frac{1}{( 2 \pi )^{n}} \partial_{\lambda} \int_{\R^{n}} \bigg( \int_{\xi^{2} + V ( x ) \leq \lambda} d \xi - \int_{\xi^{2} + W ( x ) \leq \lambda} d \xi \bigg) d x ,
\end{equation*}
from Theorem 1.3 of \cite{Ro94_01} and \eqref{a75}. In particular, there is no hope to replace $\sigma ( \lambda ; h )$ by $\CO ( h^{\infty} )$ in Proposition \ref{a91}. Nevertheless, if $W$ is as in \ref{h3}, one can show that $\sigma_{0} ( \lambda )$ goes to $0$ as $R \to + \infty$. In some sense, Proposition \ref{a91} can be seen as the equivalent of Proposition \ref{a38} for the SSF.

As in Section \ref{s2}, there is no hope to remove the assumption \ref{h4} in Proposition \ref{a91}. More precisely, considering the geometric setting in the proof of Remark \ref{a58} and using the forthcoming Corollary \ref{a86}, there exist operators $P , Q$ satisfying \ref{h1} with $\rho > n$ ($V , W$ are compactly supported) and \ref{h5} such that $s_{P , P_{0}}^{\prime}$ and $s_{Q , P_{0}}^{\prime}$ are not of the same order. That is
\begin{equation*}
s_{P , P_{0}}^{\prime} ( \lambda ) \nsim s_{Q , P_{0}}^{\prime} ( \lambda ) + \CO ( h^{- n } ) .
\end{equation*}
To obtain a general theorem, we suppose \ref{h3} as in Theorem \ref{a3} and our main result concerning the SSF is the following.

\begin{theorem}[General comparison of the derivative of SSF]\sl \label{a74}
Assume \ref{h1} with $\rho > n$ and \ref{h3}. Let $I \subset ] 0 , + \infty [$ be a compact interval. For $R > 0$ large enough, we have
\begin{equation*}
s_{P , P_{0}}^{\prime} ( \lambda ) = s_{Q , P_{0}}^{\prime} ( \lambda ) + \sigma ( \lambda ; h ) + \CO ( h^{\infty} ) \dist ( \lambda , \res ( Q ) )^{- 1} ,
\end{equation*}
uniformly for $\lambda \in I$, where the function $\sigma ( \lambda ; h )$ has a complete expansion in powers of $h$
\begin{equation*}
\sigma ( \lambda ; h ) \asymp \sigma_{0} ( \lambda ) h^{- n} + \sigma_{1} ( \lambda ) h^{1 - n} + \cdots ,
\end{equation*}
and the functions $\sigma_{\bullet} ( \lambda )$ are $C^{\infty}$ near $I$.
\end{theorem}

As in Theorem \ref{a3}, the $R$'s for which this result holds true may depend on $I$ and $P$. The proof of Theorem \ref{a74} relies on the estimate of the difference of the resolvents given in Theorem \ref{a6}, on a representation formula for the SSF by Robert and and on an estimate of the resolvent obtained by Stefanov.

\begin{proof}
From Theorem 1.10 of Robert \cite{Ro94_01}, the derivative of the SSF satisfies the following asymptotic representation formula:
\begin{equation} \label{a77}
\begin{aligned}
s_{P , P_{0}}^{\prime} ( \lambda ) ={}& \tr \big( \chi \big( E^{\prime}_{P} ( \lambda ) - E^{\prime}_{P_{0}} ( \lambda ) \big) \chi \big)   \\
&+ \tr \big( \Op ( k_{+} ) ( P_{0} - \lambda - i 0 )^{- 1} \big) + \tr \big( \Op ( k_{-} ) ( P_{0} - \lambda + i 0 )^{- 1} \big)  \\
&+ \tr \big( X_{1}^{\pm} ( P - \lambda \mp i 0 )^{- 1} Y_{1}^{\pm} ( P_{0} - \lambda \mp i 0 )^{- 1} Z_{1}^{\pm} \big)  \\
&+ \tr \big( X_{0}^{\pm} ( P_{0} - \lambda \mp i 0 )^{- 1} Y_{0}^{\pm} ( P_{0} - \lambda \mp i 0 )^{- 1} Z_{0}^{\pm} \big) ,
\end{aligned}
\end{equation}
for all $\lambda \in I$. In the two last lines we mean that we have a $(+)$ and a $(-)$ term. The cut-off function $\chi \in C^{\infty}_{0} ( \R^{n} )$ is equal to $1$ on a sufficiently large neighborhood of $0$. For $T$ selfadjoint, we use the notation
\begin{equation} \label{a83}
E^{\prime}_{T} ( \lambda ) = \frac{1}{2 i \pi} \big( ( T - \lambda - i 0 )^{- 1} - ( T - \lambda + i 0 )^{- 1} \big) .
\end{equation}
The symbols $k_{\pm}$ are classical symbols satisfying $\partial_{x}^{\alpha} \partial_{\xi}^{\beta} k_{\pm} ( x , \xi ) = \CO ( \< x \>^{- \rho - \vert \alpha \vert} \< \xi \>^{- \infty} )$ for every multiindexes $\alpha , \beta \in \N^{n}$. The operators $X_{\bullet}^{\pm} , Y_{\bullet}^{\pm} , Z_{\bullet}^{\pm}$ are negligible in the sense that, for all $M >0$, we have
\begin{align}
&\big\Vert \< x \>^{M} Y_{\bullet}^{\pm} ( P_{0} - \lambda \mp i 0 )^{- 1} Z_{\bullet}^{\pm} \< x \>^{M} \big\Vert_{\tr} = \CO ( h^{\infty} ) ,  \label{a78} \\
&\big\Vert \< x \>^{M} X_{\bullet}^{\pm} \< x \>^{M} \big\Vert = \CO ( h^{\infty} ) , \label{a79}
\end{align}
uniformly for $\lambda \in I$. Furthermore, \eqref{a77} can be differentiated in $\lambda$ at any order and we have also estimates like \eqref{a78} and \eqref{a79}. We send the reader to the paper of Robert \cite{Ro94_01} for more details. The SSF $s_{Q , P_{0}} ( \lambda )$ satisfies a formula similar to \eqref{a77} and, following the paper of Robert (which is based on the constructions of Isozaki and Kitada), one can see that the cut-off function $\chi$ can be chosen independent of $R$ large enough.

As explained in Section 6 of \cite{Ro94_01}, the terms on the second line of \eqref{a77} have a complete asymptotic expansion in $h$ with $C^{\infty}$ coefficients. That is
\begin{align*}
\sigma^{P} ( \lambda ; h ) : ={}& \tr \big( \Op ( k_{+} ) ( P_{0} - \lambda - i 0 )^{- 1} \big) + \tr \big( \Op ( k_{-} ) ( P_{0} - \lambda + i 0 )^{- 1} \big) \\
\asymp{}& \sigma_{0}^{P} ( \lambda ) h^{- n} + \sigma_{1}^{P} ( \lambda ) h^{1 - n} + \cdots
\end{align*}
With \eqref{a78} and \eqref{a79} in mind, the third and fourth lines of \eqref{a77} are directly estimated by $\CO ( h^{\infty} ) \Vert \< x \>^{- 1} ( P - \lambda - i 0 )^{-1} \< x \>^{- 1} \Vert$. Thus, \eqref{a77} provides
\begin{equation*}
s_{P , P_{0}}^{\prime} ( \lambda ) = \tr \big( \chi \big( E^{\prime}_{P} ( \lambda ) - E^{\prime}_{P_{0}} ( \lambda ) \big) \chi \big) + \sigma^{P} ( \lambda ; h ) + \CO ( h^{\infty} ) \big\Vert \< x \>^{- 1} ( P - \lambda - i 0 )^{-1} \< x \>^{- 1} \big\Vert .
\end{equation*}
The same way, we have
\begin{equation*}
s_{Q , P_{0}}^{\prime} ( \lambda ) = \tr \big( \chi \big( E^{\prime}_{Q} ( \lambda ) - E^{\prime}_{P_{0}} ( \lambda ) \big) \chi \big) + \sigma^{Q} ( \lambda ; h ) + \CO ( h^{\infty} ) \big\Vert \< x \>^{- 1} ( Q - \lambda - i 0 )^{-1} \< x \>^{- 1} \big\Vert ,
\end{equation*}
where $\sigma^{Q}$ has a complete asymptotic expansion in $h$ with $C^{\infty}$ coefficients. Since the cut-off functions $\chi$ can be chosen to be the same (and independent of $R$) in the two last equations, we deduce
\begin{align}
s_{P , P_{0}}^{\prime} ( \lambda ) ={}& s_{Q , P_{0}}^{\prime} ( \lambda ) + \tr \big( \chi \big( E^{\prime}_{P} ( \lambda ) - E^{\prime}_{Q} ( \lambda ) \big) \chi \big) \nonumber \\
&+ \widehat{\sigma} ( \lambda ; h ) + \CO ( h^{\infty} ) \big\Vert \< x \>^{- 1} ( Q - \lambda - i 0 )^{-1} \< x \>^{- 1} \big\Vert , \label{a81}
\end{align}
for $R$ large enough uniformly for $\lambda \in I$. Here, $\widehat{\sigma} = \sigma^{P} - \sigma^{Q}$ has a complete asymptotic expansion in $h$ with $C^{\infty}$ coefficients and Corollary \ref{a80} is used to estimate the resolvent of $P$.

To treat the trace of $\chi \big( E^{\prime}_{P} ( \lambda ) - E^{\prime}_{Q} ( \lambda ) \big) \chi$, it is enough to control the trace norm of the difference of the resolvents of $P$ and $Q$ (see \eqref{a83}). For that, we use Theorem \ref{a6} which gives an upper bound in operator norm and the following trick. Since $P$ and $Q$ coincide near the support of $\chi$, we have
\begin{align*}
\chi \big( ( P - \lambda &- i 0 )^{-1} - ( Q - \lambda - i 0 )^{-1} \big) \chi  \\
&= ( P + i )^{- 1} ( P - \lambda + \lambda + i ) \chi \big( ( P - \lambda - i 0 )^{-1} - ( Q - \lambda - i 0 )^{-1} \big) \chi  \\
&= ( P + i )^{- 1} \big( [ P , \chi ] + ( \lambda + i ) \chi \big) \big( ( P - \lambda - i 0 )^{-1} - ( Q - \lambda - i 0 )^{-1} \big) \chi .
\end{align*}
By induction, we obtain for any $N \in \N$
\begin{equation*}
\chi \big( ( P - \lambda - i 0 )^{-1} - ( Q - \lambda - i 0 )^{-1} \big) \chi = ( P + i )^{- N} T_{N} \widetilde{\chi} \big( ( P - \lambda - i 0 )^{-1} - ( Q - \lambda - i 0 )^{-1} \big) \chi ,
\end{equation*}
where $T_{N}$ is a differential operator of order $N$ with coefficients in $C^{\infty}_{0} ( \R^{n} )$ and $\widetilde{\chi}$ is any function in $C^{\infty}_{0} ( \R^{n} )$ with $\chi \prec \widetilde{\chi}$. For $N > n$, the microlocal analysis implies
\begin{align}
\big\Vert \chi \big( ( P - \lambda &- i 0 )^{-1} - ( Q - \lambda - i 0 )^{-1} \big) \chi \big\Vert_{\tr}   \nonumber \\
&\leq \big\Vert ( P + i )^{- N} T_{N} \big\Vert_{\tr} \big\Vert \widetilde{\chi} \big( ( P - \lambda - i 0 )^{-1} - ( Q - \lambda - i 0 )^{-1} \big) \chi \big\Vert   \nonumber \\
&\lesssim h^{- n} \big\Vert \widetilde{\chi} \big( ( P - \lambda - i 0 )^{-1} - ( Q - \lambda - i 0 )^{-1} \big) \chi \big\Vert \nonumber \\
&= \CO ( h^{\infty} ) \big\Vert \< x \>^{- 1} ( Q - \lambda - i 0 )^{- 1} \< x \>^{- 1} \big\Vert ,   \label{a82}
\end{align}
thanks to Theorem \ref{a6}. Combining \eqref{a81} with \eqref{a82} and \eqref{a83}, we get
\begin{equation} \label{a84}
s_{Q , P_{0}}^{\prime} ( \lambda ) = s_{Q , P_{0}}^{\prime} ( \lambda ) + \widehat{\sigma} ( \lambda ; h ) + \CO ( h^{\infty} ) \big\Vert \< x \>^{- 1} ( Q - \lambda - i 0 )^{-1} \< x \>^{- 1} \big\Vert , 
\end{equation}
uniformly for $\lambda \in I$.

It remains to estimate the weighted resolvent of $Q$ in terms of its resonances. For that, we use a result of Stefanov. Since $Q$ is a compactly supported perturbation of $- h^{2} \Delta$ and \eqref{a16} holds true, Proposition 4 of \cite{St02_01} states that, for any $\psi \in C^{\infty}_{0} ( \R^{n} )$,
\begin{equation*}
\big\Vert \psi ( Q - \lambda - i 0 )^{-1} \psi \big\Vert \leq \frac{h^{- \frac{3 n}{2} - 1}}{d ( z , h )} ,
\end{equation*}
for all $\lambda \in I$ and $h$ small enough, where $d ( z , h ) = \min ( \dist ( z , \res ( Q  ) ) , 1 )$. Applying Proposition \ref{a13}, this inequality becomes
\begin{equation} \label{a85}
\big\Vert \< x \>^{- 1} ( Q - \lambda - i 0 )^{-1} \< x \>^{- 1} \big\Vert \leq \frac{h^{- \frac{3 n}{2} - 1}}{d ( z , h )} ,
\end{equation}
Eventually, Theorem \ref{a74} is a consequence of \eqref{a84} and \eqref{a85}.
\end{proof}

The possible generalizations below Theorem \ref{a3} can still apply to Theorem \ref{a74}. For instance, one can glue the potential $V$ to the potential $\widetilde{V}$ as in Remark \ref{a72}, assuming that $\widetilde{V}$ satisfies \ref{h1} with $\rho > n$. On can also consider other types of operators.

If the weighted resolvent of $P$ is polynomially bounded (that is if \ref{h4} holds true), one can remove the term with $\dist ( \lambda , \res ( Q ) )^{- 1}$ in Theorem \ref{a74}. More precisely, combining \eqref{a84} with Remark \ref{a47}, we recover Proposition \ref{a91} under the hypothesis \ref{h3}.

Theorem \ref{a74} allows to extend all the results on the derivative of the SSF obtained for operators with compactly supported potential to the case of operators with $C^{\infty}$ potential. Such applications are now discussed. We first show that there is a Breit--Wigner formula for the derivative of the SSF associated to any operator under the assumption \ref{h1} with $\rho > n$. For that, we apply a result of Bruneau and Petkov \cite{BrPe03_01} (in the case where $\Omega$ is symmetric with respect to $\R$ to simplify the exposition). Instead, we could have considered Petkov and Zworski \cite{PeZw99_01,PeZw00_01}, Sj\"{o}strand and the first author \cite{BoSj01_01}, \ldots

\begin{corollary}[General Breit--Wigner formula]\sl \label{a76}
Assume \ref{h1} with $\rho > n$ and \ref{h3}. Let $W \Subset \Omega \Subset e^{i ] - \frac{\pi}{2} , \frac{\pi}{2} [} ] 0 , + \infty [$ be open, simply connected and relatively compact sets that are symmetric with respect to $\R$. Assume that $I = W \cap \R$ and $J = \Omega \cap \R$ are intervals. For $R > 0$ large enough, we have
\begin{equation*}
s_{P , P_{0}}^{\prime} ( \lambda ) = \sum_{z \in \Res ( Q ) \cap \Omega} \frac{\vert \im z \vert}{\pi \vert \lambda - z \vert^{2}} + \sigma ( \lambda ; h ) + \im r ( \lambda ; h ) + \frac{\CO ( h^{\infty} )}{\dist ( \lambda , \res ( Q ) )} ,
\end{equation*}
uniformly for $\lambda \in I$. Here, $\sigma ( \lambda ; h )$ is as in Theorem \ref{a74} and $r ( z ; h )$ is a function holomorphic in $\Omega$ with $\vert r ( z ; h ) \vert \lesssim h^{- n}$ uniformly for $z \in W$.
\end{corollary}

Corollary \ref{a76} is a direct consequence of Theorem 1 of \cite{BrPe03_01} and Theorem \ref{a74}. We now give explicit Breit--Wigner formulas in two geometric situations.

\begin{example}\rm \label{a87}
We compute the derivative of the SSF in the well in the island situation. We consider $P$ satisfying \ref{h1} with $\rho > n$ and the geometric assumptions of Example \ref{a61} on some interval $I$. In particular, Figure \ref{f4} provides an example of such potential and the corresponding distribution of resonances when the potential is analytic at infinity.

Let $P^{i}$ denote the operator $P$ (or $Q$) restricted to an neighborhood of the $\pi ( \Sigma^{i} ( I ) )$ with Dirichlet boundary condition. Thus, $P^{i}$ encodes the dynamics inside the island. There exists a bijection between the eigenvalues of $P^{i}$ and the resonances of $Q$
\begin{equation} \label{a88}
F : \sigma ( P^{i} ) \cap \big( I + [ \alpha_{1} ( h ) , \alpha_{2} ( h ) ] \big) \longrightarrow \res ( Q ) \cap \big( I + [ \alpha_{3} ( h ) , \alpha_{4} ( h ) ] - i [ 0 , C h ] \big) ,
\end{equation}
such that $F ( z ) = z + \CO ( e^{- \delta_{0} / h} )$ and $\alpha_{\bullet} ( h ) = \CO ( e^{- \delta_{0} / h} )$ for any $C > 0$ large enough and some $\delta_{0} > 0$. This follows from the corollary on page 201 of Nakamura, Stefanov and Zworski \cite{NaStZw03_01} in the present setting (see also Helffer and Sj{\"o}strand \cite{HeSj86_01} for analytic potentials). In particular, any resonance of $Q$ in $I + i [ - 1 , 0 ]$ satisfies
\begin{equation} \label{a90}
\vert \im z \vert \leq e^{- \delta_{0} / h} \qquad \text{ or } \qquad \vert \im z \vert \geq C h ,
\end{equation}
for $h$ small enough. In the second case, they actually satisfy $\vert \im z \vert \geq M h \vert \ln h \vert$ for any $M > 0$ (see Lemma 4.5 and Lemma 4.6 of \cite{NaStZw03_01}).

To study the SSF near an isolated resonance, we make the following separation assumption. Let $\CI ( h ) \subset I$ be an interval containing only one eigenvalue $z_{0} ( h )$ of $P^{i}$ such that
\begin{equation} \label{a89}
\dist \big( \partial \CI ( h ) , \sigma ( P^{i} ) \big) \geq h^{M} ,
\end{equation}
for some $M \geq 1$. From \eqref{a88}, there exists a unique resonance $z = z ( h )$ of $Q$ exponentially close to the real axis with real part in $\CI ( h )$ and the other resonances of $Q$ are at distance at least $h^{M} / 2$ from $\CI ( h )$.

\begin{corollary}[Shape resonances]\sl \label{a86}
Under the previous assumptions,
\begin{equation*}
s_{P , P_{0}}^{\prime} ( \lambda ) = \frac{\vert \im z \vert}{\pi \vert \lambda - z \vert^{2}} + \frac{\CO ( h^{\infty} )}{\vert \lambda - z \vert} + \CO ( h^{- n - 1} ) ,
\end{equation*}
uniformly for $\lambda \in \CI ( h )$.
\end{corollary}

The imaginary part of the resonances of $Q$, close to the minimum of the well, has been computed asymptotically by Fujii{\'e}, Lahmar-Benbernou and Martinez \cite{FuLaMa11_01} (see Helffer and Sj{\"o}strand \cite{HeSj86_01} in the analytic case) under some geometric assumptions. The two functions $\vert \im z \vert \vert \lambda - z \vert^{- 2} / \pi$ and $\CO ( h^{\infty} ) \vert \lambda - z \vert^{- 1}$ of $\lambda \in [\re z - 1 , \re z  + 1]$ form two peaks centered at $\re z$, of height $\vert \im z \vert^{- 1} / \pi$ and $\CO ( h^{\infty} ) \vert \im z \vert^{- 1}$ respectively, and of total mass $1 + \CO ( h^{\infty} )$ and $\CO ( h^{\infty} )$ respectively. Then, $\CO ( h^{\infty} ) \vert \lambda - z \vert^{- 1}$ is actually a remainder term even if it can be larger than the leading term for some values of $\lambda$.

\begin{proof}
Let $I \Subset J \Subset ] 0 , + \infty [$ be an open interval such that the geometric assumptions, \eqref{a88} and \eqref{a90} hold true with $I$ replaced by $J$, with perhaps new constants. Choosing $W = I + i ] - 1 , 1 [$ and $\Omega = J + i ] - 2 , 2 [$, Corollary \ref{a76} gives
\begin{equation*}
s_{P , P_{0}}^{\prime} ( \lambda ) = \frac{\vert \im z \vert}{\pi \vert \lambda - z \vert^{2}} + \sum_{\fract{\rho \in \Res ( Q ) \cap \Omega}{\rho \neq z}} \frac{\vert \im \rho \vert}{\pi \vert \lambda - \rho \vert^{2}} + \frac{\CO ( h^{\infty} )}{\vert \lambda - z \vert} + \frac{\CO ( h^{\infty} )}{\dist ( \lambda , \res ( Q ) \setminus \{ z \})} + \CO ( h^{- n} ) ,
\end{equation*}
uniformly for $\lambda \in I$. Since the number of resonances of $Q$ in $\Omega$ is bounded by $\CO ( h^{- n} )$, the inequalities \eqref{a90} and \eqref{a89} imply
\begin{align*}
\sum_{\fract{\rho \in \Res ( Q ) \cap \Omega}{\rho \neq z}} \frac{\vert \im \rho \vert}{\pi \vert \lambda - \rho \vert^{2}} &= \sum_{\fract{\rho \in \Res ( Q ) \cap \Omega}{\rho \neq z , \ \vert \im z \vert \leq e^{- \delta_{0} / h}}} \frac{\vert \im \rho \vert}{\pi \vert \lambda - \rho \vert^{2}} + \sum_{\fract{\rho \in \Res ( Q ) \cap \Omega}{\rho \neq z , \ \vert \im z \vert \geq C h}} \frac{\vert \im \rho \vert}{\pi \vert \lambda - \rho \vert^{2}}   \\
&\lesssim \sum_{\fract{\rho \in \Res ( Q ) \cap \Omega}{\rho \neq z , \ \vert \im z \vert \leq e^{- \delta_{0} / h}}} \frac{e^{- \delta_{0} / h}}{h^{2 M}} + \sum_{\fract{\rho \in \Res ( Q ) \cap \Omega}{\rho \neq z , \ \vert \im z \vert \geq C h}} \frac{1}{h} \\
&\lesssim h^{- n - 1} ,
\end{align*}
uniformly for $\lambda \in \CI ( h ) $. The same way,
\begin{equation*}
\frac{\CO ( h^{\infty} )}{\dist ( \lambda , \res ( Q ) \setminus \{ z \})} = \frac{\CO ( h^{\infty} )}{h^{M}} = \CO ( h^{\infty} ) ,
\end{equation*}
and the corollary follows.
\end{proof}

When the potential is dilation analytic outside the island, a similar formula has been obtained by G{\'e}rard, Martinez and Robert \cite{GeMaRo89_01}. By comparison with Corollary \ref{a86}, they have $e^{- \varepsilon / h}$ for all $\varepsilon > 0$ instead of $h^{M}$ for some $M > 0$ in the separation assumption \eqref{a89}. Under their hypothesis, we would have an additional rest of the form $e^{\varepsilon / h}$ for all $\varepsilon > 0$. Moreover, their remainder terms
\begin{equation*}
\frac{\CO \big( e^ {- ( 2 S_{0} - \varepsilon ) / h} \big)}{\vert \lambda - z \vert} + \CO ( h^{- n} ) ,
\end{equation*}
for all $\varepsilon > 0$, are better than ours. In the previous formula, $S_{0} > 0$ is the Agmon distance between the well and the sea. We have proved Corollary \ref{a86} using Corollary \ref{a76}. But, it may be possible to first generalize \cite{GeMaRo89_01} to analytic potentials at infinity and then to apply Theorem \ref{a74} to obtain Corollary \ref{a86}. Such a proof may lead to different remainder terms.

If each element of $I$ is a noncritical energy level for $P$, the remainder term $\CO ( h^{- n - 1} )$ can be replaced by $\CO ( h^{- n} )$. Indeed, applying \cite[Corollary 1]{BrPe03_01} instead of \cite[Theorem 1]{BrPe03_01} leads to another version of Corollary \ref{a76} where the $h$-independent sets $W$, $\Omega$ are replaced by neighborhoods of size $h$ of $\lambda$. We conclude using upper bounds on the number of resonances in small domains.

Instead of computing the derivative of the SSF, some results are devoted to the jump of the SSF across the real part of a resonance, that is the quantity
\begin{equation*}
s_{P , P_{0}}^{\prime} ( \re z + \delta ) - s_{P , P_{0}}^{\prime} ( \re z - \delta ) ,
\end{equation*}
with $0 < \delta \ll 1$. Some of them require the analyticity of the potential $V$ like Theorem 6.4 of Robert \cite{Ro94_01}, whereas others require only a symbol assumption on $V$ like Nakamura \cite{Na99_01}. The first ones can be generalized to the $C^{\infty}$ setting using Theorem \ref{a74}.
\end{example}

\begin{example}\rm \label{b2}
At the end of this part, we give an explicit Breit-Wigner formula in dimension $n = 1$ in presence of a homoclinic trajectory in the $C^{\infty}$ setting. In the analytic setting, this has been done by Fujii\'e and the third author below Theorem 2.2 of \cite{FuRa03_01}.

\begin{figure}%[!h]
\begin{center}
\begin{picture}(0,0)%
\includegraphics{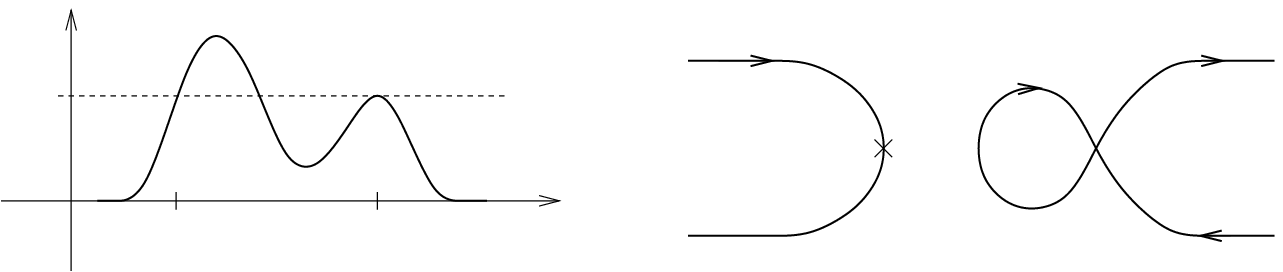}%
\end{picture}%
\setlength{\unitlength}{1105sp}%
\begingroup\makeatletter\ifx\SetFigFont\undefined%
\gdef\SetFigFont#1#2#3#4#5{%
  \reset@font\fontsize{#1}{#2pt}%
  \fontfamily{#3}\fontseries{#4}\fontshape{#5}%
  \selectfont}%
\fi\endgroup%
\begin{picture}(21880,4544)(-1221,-383)
\put(13576,1589){\makebox(0,0)[rb]{\smash{{\SetFigFont{9}{10.8}{\rmdefault}{\mddefault}{\updefault}$( x_{\ell} , 0)$}}}}
\put(17851,1589){\makebox(0,0)[lb]{\smash{{\SetFigFont{9}{10.8}{\rmdefault}{\mddefault}{\updefault}$( 0 , 0)$}}}}
\put(16426,3089){\makebox(0,0)[b]{\smash{{\SetFigFont{9}{10.8}{\rmdefault}{\mddefault}{\updefault}$\gamma_{0}$}}}}
\put(19651,3614){\makebox(0,0)[b]{\smash{{\SetFigFont{9}{10.8}{\rmdefault}{\mddefault}{\updefault}$\gamma_{\rm out}$}}}}
\put(19651,-286){\makebox(0,0)[b]{\smash{{\SetFigFont{9}{10.8}{\rmdefault}{\mddefault}{\updefault}$\gamma_{\rm in}$}}}}
\put(-374,2564){\makebox(0,0)[rb]{\smash{{\SetFigFont{9}{10.8}{\rmdefault}{\mddefault}{\updefault}$E_{0}$}}}}
\put(5251,164){\makebox(0,0)[b]{\smash{{\SetFigFont{9}{10.8}{\rmdefault}{\mddefault}{\updefault}$0$}}}}
\put(1801,164){\makebox(0,0)[b]{\smash{{\SetFigFont{9}{10.8}{\rmdefault}{\mddefault}{\updefault}$x_{\ell}$}}}}
\put(3001,3389){\makebox(0,0)[lb]{\smash{{\SetFigFont{9}{10.8}{\rmdefault}{\mddefault}{\updefault}$V ( x  )$}}}}
\end{picture}
\end{center}
\caption{The geometric setting of Example \ref{b2}.} \label{f6}
\end{figure}

We assume that the potential $V$ satisfies \ref{h1} with $\rho > n = 1$ and is as in Figure \ref{f6}. In particular, $V$ has a local non-degenerate maximum at $x = 0$, i.e.
\begin{equation*}
V (x) = E_{0} - \frac{\mu^{2}}{4} x^{2} + \CO ( x^{3} ) ,
\end{equation*}
with $E_{0} , \mu > 0$. The trapped set at energy $E_{0}$ consists of the hyperbolic fixed point $( 0 , 0 )$ and a homoclinic trajectory $\gamma _{0} ( t ) = ( x_{0} ( t ) , \xi_{0} ( t ) )$. To state our result, we need to define some geometric quantities. The action along $\gamma_{0}$ is
\begin{equation*}
A_{0} = \int_{\gamma_{0}} \xi \, d x .
\end{equation*}
Let $\gamma_{\rm in} ( t ) , \gamma_{\rm out} ( t )$ be the Hamiltonian trajectories of energy $E_{0}$ such that
\begin{equation*}
\begin{aligned}
\gamma_{\rm in} ( t ) &= ( x_{\rm in} ( t ) , \xi_{\rm in} ( t ) ) = \big( - 2 \sqrt{E_{0}} t , - \sqrt{E_{0}} \big)  \qquad &&\text{ for } t \ll - 1 , \\
\gamma_{\rm out} ( t ) &= ( x_{\rm out} ( t ) , \xi_{\rm out} ( t ) ) = \big( 2 \sqrt{E_{0}} t , \sqrt{E_{0}} \big)  \qquad &&\text{ for } t \gg 1 .
\end{aligned}
\end{equation*}
They have the following asymptotic behaviors at $0$
\begin{equation*}
\begin{aligned}
x_{0} ( t ) &= g_{0}^{\pm} e^{\pm \mu t} + o ( e^{\pm \mu t} ) \qquad  &&\text{ as } t \to \mp \infty ,  \\
x_{\rm in} ( t ) &= g_{\rm in} e^{- \mu t} + o ( e^{- \mu t} )   &&\text{ as } t \to + \infty ,  \\
x_{\rm out} ( t ) &= g_{\rm out} e^{\mu t} + o ( e^{\mu t} )     &&\text{ as } t \to - \infty .
\end{aligned}
\end{equation*}
(see e.g. Helffer and Sj\"{o}strand \cite[(2.7)]{HeSj85_01}). In fact, $( x_{\rm in} ( t ) , \xi_{\rm in} ( t ) ) = ( x_{\rm out} ( - t ) , - \xi_{\rm out} ( - t ) )$, $g_{\rm in} = g_{\rm out} > 0$ and $g_{0}^{\pm} < 0$. The rescaled spectral parameter is defined by
\begin{equation*}
\sigma = \frac{\lambda - E_{0}}{h} .
\end{equation*}
With these notations, one can define the three following quantities:
\begin{align}
\CQ_{0} &= e^{i A_{0} / h} \Gamma \Big( \frac{1}{2} - i \frac{\sigma}{\mu} \Big) \frac{1}{\sqrt{2 \pi}} e^{- \frac{\pi}{2} \frac{\sigma}{\mu}} \big( \mu \vert g_{0}^{+} \vert \vert g_{0}^{-} \vert \big)^{i \frac{\sigma}{\mu}} ,   \nonumber \\
\CA &= e^{i A_{0} / h} \Gamma^{2} \Big( \frac{1}{2} - i \frac{\sigma}{\mu} \Big) \frac{i}{2 \pi} e^{\pi \frac{\sigma}{\mu}} \big( \mu^{2} \vert g_{\rm in} \vert \vert g_{\rm out} \vert \vert g_{0}^{-} \vert \vert g_{0}^{+} \vert \big)^{i \frac{\sigma}{\mu}} ,   \label{b13} \\
\CB &= \Gamma \Big( \frac{1}{2} - i \frac{\sigma}{\mu} \Big) \frac{i}{\sqrt{2 \pi}} e^{- \frac{\pi}{2} \frac{\sigma}{\mu}} \big( \mu \vert g_{\rm out} \vert \vert g_{\rm in} \vert \big)^{i \frac{\sigma}{\mu}} .  \nonumber
\end{align}
In this setting, the derivative of the spectral shift function satisfies

\begin{corollary}[Homoclinic resonances]\sl \label{b3}
Let $C > 0$. Under the previous assumptions,
\begin{equation*}
s_{P , P_{0}}^{\prime} ( \lambda ) = \frac{\vert \ln h \vert}{2 \pi \mu h} \frac{2 h^{- 2 i \frac{\sigma}{\mu}} \CA \big( 1 - h^{- i \frac{\sigma}{\mu}} \CQ_{0} \big)^{- 1} + h^{- 3 i \frac{\sigma}{\mu}} \CA \CQ_{0} \big( 1 - h^{- i \frac{\sigma}{\mu}} \CQ_{0} \big)^{- 2} - h^{- i \frac{\sigma}{\mu}} \CB}{h^{- 2 i \frac{\sigma}{\mu}} \CA \big( 1 - h^{- i \frac{\sigma}{\mu}} \CQ_{0} \big)^{- 1} - h^{- i \frac{\sigma}{\mu}} \CB} + \CO ( h^{- 1} ) ,
\end{equation*}
uniformly for $\lambda \in [ E_{0} - C h , E_{0} + C h ]$.
\end{corollary}

\begin{figure}%[!h]
\begin{center}\setlength{\unitlength}{987sp}
\begin{picture}(18000,10000)(1500,-1000)
\includegraphics[width=320pt]{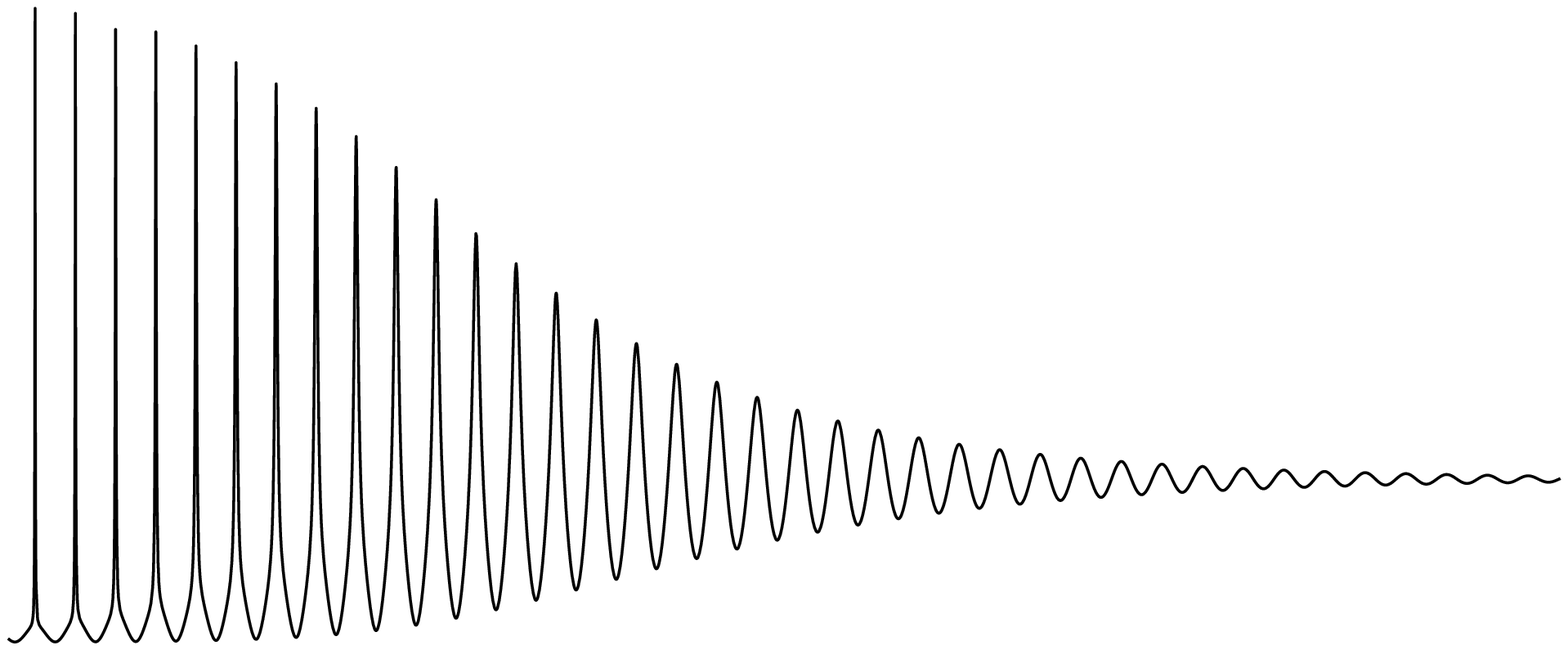}
\put(-22000,-1500){\vector(0,1){11500}}
\put(-23000,-500){\vector(1,0){24500}}
\put(-15500,-800){\line(0,1){600}}
\put(-15100,-1400){\makebox(0,0)[rb]{\smash{$E_{0}$}}}
\put(-1700,-800){\line(0,1){600}}
\put(-290,-1400){\makebox(0,0)[rb]{\smash{$E_{0} + C h$}}}
\end{picture}
\end{center}
\caption{The derivative of the SSF given by Corollary \ref{b3}.} \label{f7}
\end{figure}

This result is illustrated in Figure \ref{f7} and has been analyzed in \cite{FuRa03_01}. The asymptotic of the scattering matrix $S ( \lambda )$ and the SSF $s_{P , P_{0}} ( \lambda )$ can be found in the proof of Corollary \ref{b3}. To show this result, one can not simply rely on Theorem \ref{a74} and \cite{FuRa03_01}. Indeed, \cite{FuRa03_01} is based on the exact WKB analysis that requires the analyticity of the potential $V$ in a whole neighborhood of the real axis. Thus, we give a new proof that uses only marginally the results of this section.

\begin{proof}
From Section 4.3 (B) of Fujii\'e, Zerzeri and two authors \cite{BoFuRaZe18_01} and Remark \ref{a47}, we know that \ref{h4} holds true. Then, applying Proposition \ref{a91} (or Theorem \ref{a74}), we can always assume that $V$ is compactly supported. Let $\theta ( \lambda )$ denote the scattering phase defined by
\begin{equation}  \label{b19}
\det S ( \lambda ) = e^{2 i \theta ( \lambda )} ,
\end{equation}
$S ( \lambda )$ being the (unitary) scattering matrix at energy $\lambda \in \R$. The Birman--Krein formula states that
\begin{equation}\label{b5}
\theta^{\prime} ( \lambda ) = \pi s_{P , P_{0}}^{\prime} ( \lambda ) .
\end{equation}
Thus, to compute the SSF, it is enough to compute the coefficients of the scattering matrix. For compactly supported potentials in dimension $n = 1$, the scattering matrix has the following representation formula (see e.g. Proposition 2.1 of Petkov and Zworski \cite{PeZw01_01}):
\begin{equation}
S ( \lambda ) = \left( \begin{array}{cc}
S_{+ , +} ( \lambda ) & S_{+ , -} ( \lambda ) \\
S_{- , +} ( \lambda ) & S_{- , -} ( \lambda )
\end{array} \right) ,
\end{equation}
with
\begin{equation} \label{b14}
S_{a , b} ( \lambda ) = \frac{i}{2 h \sqrt{\lambda}} \Big\< [ P , \chi ] e^{a i \sqrt{\lambda} x / h} , ( P - \lambda - i 0 )^{- 1} [ P , \chi_{b} ] e^{b i \sqrt{\lambda} x / h} \Big\> ,
\end{equation}
where $a , b = \pm$ and the cut-off functions $\chi , \chi_{\pm}$ are as in Figure \ref{f8}. We send back the reader to \cite{FuRa03_01} for more details on the scattering theory in dimension $1$. From Section 4.3 (B) of \cite{BoFuRaZe18_01}, the resolvent of the distorted operator $P_{\theta}$, with $\theta = h \vert \ln h \vert$, is polynomially bounded in 
\begin{equation*}
\Omega = E_{0} + [ - C h , C h ] + i \big[ - \nu h \vert \ln h \vert^{- 1} , \nu h \vert \ln h \vert^{- 1} \big] ,
\end{equation*}
for $\nu$ small enough depending on $C$. Then, if the distortion occurs outside the support of $\chi$, one can replace the resolvent of $P$ by that of $P_{\theta}$ in \eqref{b14} and the quantity $( P_{\theta} - \lambda )^{- 1} [ P , \chi_{b} ] e^{b i \sqrt{\lambda} x / h}$ is polynomially bounded for $\lambda \in \Omega$. In particular, we can use the semiclassical $C^{\infty}$ microlocal analysis to treat these quantities.

\begin{figure}%[!h]
\begin{center}
\begin{picture}(0,0)%
\includegraphics{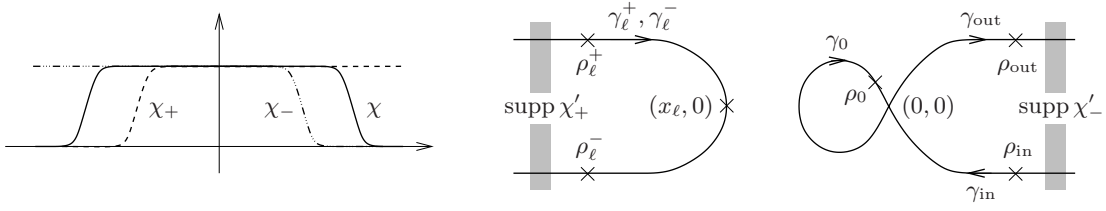}%
\end{picture}%
\setlength{\unitlength}{1105sp}%
\begingroup\makeatletter\ifx\SetFigFont\undefined%
\gdef\SetFigFont#1#2#3#4#5{%
  \reset@font\fontsize{#1}{#2pt}%
  \fontfamily{#3}\fontseries{#4}\fontshape{#5}%
  \selectfont}%
\fi\endgroup%
\begin{picture}(24712,4216)(-1221,-355)
\put(11926,1589){\makebox(0,0)[rb]{\smash{{\SetFigFont{9}{10.8}{\rmdefault}{\mddefault}{\updefault}$\supp \chi_{+}^{\prime}$}}}}
\put(2026,1589){\makebox(0,0)[lb]{\smash{{\SetFigFont{9}{10.8}{\rmdefault}{\mddefault}{\updefault}$\chi_{+}$}}}}
\put(5326,1589){\makebox(0,0)[rb]{\smash{{\SetFigFont{9}{10.8}{\rmdefault}{\mddefault}{\updefault}$\chi_{-}$}}}}
\put(6901,1589){\makebox(0,0)[lb]{\smash{{\SetFigFont{9}{10.8}{\rmdefault}{\mddefault}{\updefault}$\chi$}}}}
\put(17476,3089){\makebox(0,0)[b]{\smash{{\SetFigFont{9}{10.8}{\rmdefault}{\mddefault}{\updefault}$\gamma_{0}$}}}}
\put(20701,-286){\makebox(0,0)[b]{\smash{{\SetFigFont{9}{10.8}{\rmdefault}{\mddefault}{\updefault}$\gamma_{\rm in}$}}}}
\put(20701,3614){\makebox(0,0)[b]{\smash{{\SetFigFont{9}{10.8}{\rmdefault}{\mddefault}{\updefault}$\gamma_{\rm out}$}}}}
\put(14701,1589){\makebox(0,0)[rb]{\smash{{\SetFigFont{9}{10.8}{\rmdefault}{\mddefault}{\updefault}$( x_{\ell} , 0)$}}}}
\put(18151,1889){\makebox(0,0)[rb]{\smash{{\SetFigFont{9}{10.8}{\rmdefault}{\mddefault}{\updefault}$\rho_{0}$}}}}
\put(18901,1589){\makebox(0,0)[lb]{\smash{{\SetFigFont{9}{10.8}{\rmdefault}{\mddefault}{\updefault}$( 0 , 0)$}}}}
\put(23476,1589){\makebox(0,0)[rb]{\smash{{\SetFigFont{9}{10.8}{\rmdefault}{\mddefault}{\updefault}$\supp \chi_{-}^{\prime}$}}}}
\put(21526,689){\makebox(0,0)[b]{\smash{{\SetFigFont{9}{10.8}{\rmdefault}{\mddefault}{\updefault}$\rho_{\rm in}$}}}}
\put(21526,2564){\makebox(0,0)[b]{\smash{{\SetFigFont{9}{10.8}{\rmdefault}{\mddefault}{\updefault}$\rho_{\rm out}$}}}}
\put(13126,3614){\makebox(0,0)[b]{\smash{{\SetFigFont{9}{10.8}{\rmdefault}{\mddefault}{\updefault}$\gamma_{\ell}^{+} , \gamma_{\ell}^{-}$}}}}
\put(11926,2564){\makebox(0,0)[b]{\smash{{\SetFigFont{9}{10.8}{\rmdefault}{\mddefault}{\updefault}$\rho_{\ell}^{+}$}}}}
\put(11926,689){\makebox(0,0)[b]{\smash{{\SetFigFont{9}{10.8}{\rmdefault}{\mddefault}{\updefault}$\rho_{\ell}^{-}$}}}}
\end{picture}%
\end{center}
\caption{The cut-off functions $\chi_{\star}$ and the geometry of Corollary \ref{b3}.} \label{f8}
\end{figure}

We compute the four coefficients of $S ( \lambda )$ for $\lambda = E_{0} + h \sigma \in \Omega$ and begin with $S_{- , +} ( \lambda )$. Let $\gamma_{\ell}^{\pm} ( t ) = ( x_{\ell}^{\pm} ( t ) , \xi_{\ell}^{\pm} ( t ) )$ be two parameterizations of the Hamiltonian trajectory ``on the left'' such that $x_{\ell}^{\pm} ( t ) = \pm 2 \sqrt{E_{0}} t$ for $\mp t \gg 1$ and
\begin{equation*}
A_{\ell} = 2 \int_{- \infty}^{x_{\ell}} \big( \sqrt{E_{0} - V ( x )} - \sqrt{E_{0}} \big) d x + 2 \sqrt{E_{0}} x_{\ell} ,
\end{equation*}
its action. Then, there exists $T_{\ell} \in \R$ such that $\gamma_{\ell}^{+} ( t ) = \gamma_{\ell}^{-} ( t + T_{\ell} )$. Let also $\rho_{\ell}^{\pm}$ be two points on this trajectory according to Figure \ref{f8}. We consider
\begin{equation*}
v = ( P_{\theta} - \lambda )^{- 1} [ P , \chi_{+} ] e^{i \sqrt{\lambda} x / h} .
\end{equation*}
Working as in Section 4.2 of \cite{BoFuRaZe11_01}, we get $v = e^{i \sqrt{\lambda} x / h}$ microlocally near $\rho_{\ell}^{+}$. Moreover, $v$ satisfies the evolution equation $( P - \lambda ) v = 0$ microlocally along $\gamma_{\ell}^{\pm}$ between $\rho_{\ell}^{+}$ and $\rho_{\ell}^{-}$. Then, the propagation of singularities implies that
\begin{equation*}
v = \big( i e^{- i \sigma T_{\ell}} e^{i A_{\ell} / h} + S ( h ) \big) e^{- i \sqrt{\lambda} x / h} ,
\end{equation*}
microlocally near $\rho_{\ell}^{-}$. Finally, a direct computation in \eqref{b14} provides
\begin{equation} \label{b15}
S_{- , +} ( \lambda ) = i e^{- i \sigma T_{\ell}} e^{i A_{\ell} / h} + \CO ( h ) ,
\end{equation}
uniformly for $\lambda \in \Omega$. More simply, the diagonal terms satisfy
\begin{equation} \label{b16}
S_{+ , +} ( \lambda ) = \CO ( h^{\infty} ) \qquad \text{and} \qquad S_{- , -} ( \lambda ) = \CO ( h^{\infty} ) ,
\end{equation}
uniformly for $\lambda \in \Omega$.

It remains to compute $S_{+ , -} ( \lambda )$. For that, we define
\begin{equation*}
u = ( P_{\theta} - \lambda )^{- 1} [ P , \chi_{-} ] e^{- i \sqrt{\lambda} x / h} .
\end{equation*}
We choose three points $\rho_{\star}$, with $\star = 0 , {\rm in} , {\rm out}$ as in Figure \ref{f8}. In particular, $\rho_{\star} = \gamma_{\star} ( t_{\star} )$ for some $t_{\star} \in \R$. Let $u_{\star}$ be a microlocal restriction of $u$ to a neighborhood of $\rho_{\star}$. We send back the reader to \cite{BoFuRaZe18_01} for more details on such techniques. Let $\varphi_{-}^{0}$ (resp. $\varphi_{+}^{0}$) denote the generating phase function of the incoming (resp. outgoing) Lagrangian manifold $\Lambda_{-}$ (resp. $\Lambda_{+}$) associated to the fixed point $( 0 , 0 )$ with $\varphi_{\pm}^{0} ( 0 ) = 0$. This means that $\Lambda_{\pm} = \{ ( x , \varphi_{\pm}^{0} {}^{\prime} ( x ) ) \}$ for $x$ near $0$. As in the previous paragraph, we have
\begin{equation} \label{b17}
u = e^{- i \sqrt{\lambda} x / h} ,
\end{equation}
microlocally near $\rho_{\rm in}$. From \cite[Section 8]{BoFuRaZe18_01}, $u_{0} , u_{\rm in}$ (resp. $u_{\rm out}$) are Lagrangian distributions of order $0$ with Lagrangian manifold $\Lambda_{-}$ (resp. $\Lambda_{+}$). Inspired by \cite[(11.25)]{BoFuRaZe18_01}, we normalized the symbols of $u_{\star}$ as
\begin{align*}
u_{0} ( x ) &= e^{i \sigma t_{0}} \sqrt{\frac{\mu \vert g_{-}^{0} \vert}{\vert \partial_{t} x_{0} ( t_{-}^{0} ) \vert}} a_{0} ( x , h ) e^{i \varphi_{-}^{0} ( x ) / h} ,  \\
u_{\rm in} ( x ) &= e^{i \sigma t_{\rm in}} \sqrt{\frac{\mu \vert g_{\rm in} \vert}{\vert \partial_{t} x_{\rm in} ( t_{\rm in} ) \vert}} a_{\rm in} ( x , h ) e^{i \varphi_{-}^{0} ( x ) / h} ,  \\
u_{\rm out} ( x ) &= e^{i \sigma t_{\rm out}} \sqrt{\frac{\mu \vert g_{\rm out} \vert}{\vert \partial_{t} x_{\rm out} ( t_{\rm out} ) \vert}} a_{\rm out} ( x , h ) e^{i \varphi_{+}^{0} ( x ) / h} .
\end{align*}
The action along the trajectories $\gamma_{\rm in} , \gamma_{\rm out}$ is
\begin{equation*}
A_{\rm in} = A_{\rm out} = \int_{0}^{+ \infty} \big( \sqrt{E_{0} - V ( x )} - \sqrt{E_{0}} \big) \, d x .
\end{equation*}
From \eqref{b17} and the previous normalization, we have
\begin{equation} \label{b11}
a_{\rm in} ( x_{\rm in} , h ) = e^{i A_{\rm in} / h} \sqrt{\frac{2 \sqrt{E_{0}}}{\mu \vert g_{\rm in} \vert}} + \CO ( h ) .
\end{equation}
On the other hand, Lemma 11.5 of \cite{BoFuRaZe18_01} yields
\begin{equation} \label{b7}
a_{0} ( x_{0} , h ) = h^{- i \frac{\sigma}{\mu}} \CQ_{0} a_{0} ( x_{0} , h ) + h^{- i \frac{\sigma}{\mu}} \CQ_{\rm in} a_{\rm in} ( x_{\rm in} , h ) + \CO ( h^{\zeta} ),
\end{equation}
where $\CQ_{0}$ is given by \eqref{b13} and 
\begin{equation*}
\CQ_{\rm in} = e^{i A_{0} / h} \Gamma \Big( \frac{1}{2} - i \frac{\sigma}{\mu} \Big) \frac{1}{\sqrt{2 \pi}} \sqrt{\frac{\big\vert g_{\rm in} \big\vert}{\vert g_{-}^{0} \vert}} e^{i \frac{\pi}{2}} e^{\frac{\pi}{2} \frac{\sigma}{\mu}}  \big( \mu \vert g_{+}^{0} \vert \vert g_{\rm in} \vert \big)^{i \frac{\sigma}{\mu}} ,
\end{equation*}
and some $0 < \zeta < 1$. From Section 11.2 of \cite{BoFuRaZe18_01}, the quantization function $1 - h^{- i \frac{\sigma}{\mu}} \CQ_{0} ( z , h )$ has a uniformly bounded inverse for $z \in \Omega$ assuming $\nu$ small enough. Then \eqref{b7} becomes
\begin{equation} \label{b8}
a_{0} ( x_{0} , h ) = h^{- i \frac{\sigma}{\mu}} \CQ_{\rm in} \big( 1 - h^{- i \frac{\sigma}{\mu}} \CQ_{0} \big)^{- 1} a_{\rm in} ( x_{\rm in} , h ) + \CO ( h^{\zeta} ) .
\end{equation}
Now, working as in \cite[(11.29)]{BoFuRaZe18_01}, we get
\begin{equation} \label{b9}
a_{\rm out} ( x_{\rm out} , h ) = h^{- i \frac{\sigma}{\mu}} \CR_{0} a_{0} ( x_{0} , h ) + h^{- i \frac{\sigma}{\mu}} \CR_{\rm in} a_{\rm in} ( x_{\rm in} , h ) + \CO ( h^{\zeta} ),
\end{equation}
with
\begin{align*}
\CR_{0} &= \Gamma \Big( \frac{1}{2} - i \frac{\sigma}{\mu} \Big) \frac{1}{\sqrt{2 \pi}} \sqrt{\frac{\big\vert g_{-}^{0} \big\vert}{\big\vert g_{\rm out} \big\vert}} e^{\frac{\pi}{2} \frac{\sigma}{\mu}} \big( \mu \vert g_{\rm out} \vert \vert g_{-}^{0} \vert \big)^{i \frac{\sigma}{\mu}} ,  \\
\CR_{\rm in} &= \Gamma \Big( \frac{1}{2} - i \frac{\sigma}{\mu} \Big) \frac{1}{\sqrt{2 \pi}} e^{- i \frac{\pi}{2}} \sqrt{\frac{\big\vert g_{\rm in} \big\vert}{\vert g_{\rm out} \vert}} e^{- \frac{\pi}{2} \frac{\sigma}{\mu}} \big( \mu \vert g_{\rm out} \vert \vert g_{\rm in} \vert \big)^{i \frac{\sigma}{\mu}} .
\end{align*}
Combining \eqref{b8} and \eqref{b9}, it comes
\begin{equation} \label{b10}
a_{\rm out} ( x_{\rm out} , h ) = \Big(  h^{- 2 i \frac{\sigma}{\mu}} \CR_{0} \CQ_{\rm in} \big( 1 - h^{- i \frac{\sigma}{\mu}} \CQ_{0} \big)^{- 1} + h^{- i \frac{\sigma}{\mu}} \CR_{\rm in} \Big) a_{\rm in} ( x_{\rm in} , h ) + \CO ( h^{\zeta} ),
\end{equation}
uniformly for $z \in \Omega$. On the other hand, we have
\begin{equation*}
u = \Big( e^{i A_{\rm out} / h} \sqrt{\frac{\mu \vert g_{\rm out} \vert}{2 \sqrt{E_{0}}}} a_{\rm out} ( x_{\rm out} , h ) + S ( h ) \Big) e^{i \sqrt{\lambda} x / h} ,
\end{equation*}
microlocally near $\rho_{\rm out}$. Thanks to \eqref{b14}, this leads to
\begin{equation} \label{b12}
S_{+ , -} ( \lambda ) = e^{i A_{\rm out} / h} \sqrt{\frac{\mu \vert g_{\rm out} \vert}{2 \sqrt{E_{0}}}} a_{\rm out} ( x_{\rm out} , h ) +\CO ( h ).
\end{equation}
Combining \eqref{b11}, \eqref{b10} and \eqref{b12}, we eventually obtain
\begin{equation} \label{b18}
S_{+ , -} ( \lambda ) = e^{i (A_{\rm in} +  A_{\rm out} ) / h} \Big(  h^{- 2 i \frac{\sigma}{\mu}} \CA \big( 1 - h^{- i \frac{\sigma}{\mu}} \CQ_{0} \big)^{- 1} - h^{- i \frac{\sigma}{\mu}} \CB \Big) + \CO ( h^{\zeta} ) ,
\end{equation}
where $\CA , \CB , \CQ_{0}$ are given by \eqref{b13}.

From the computation of the coefficients of $S ( \lambda )$ carried out in \eqref{b15}, \eqref{b16} and \eqref{b18}, one can write $\det S ( \lambda ) = \alpha ( \lambda ) + r ( \lambda )$ with
\begin{equation*}
\alpha ( \lambda ) = - i e^{- i \sigma T_{\ell}} e^{i ( A_{\ell} + A_{\rm in} +  A_{\rm out} ) / h} \Big(  h^{- 2 i \frac{\sigma}{\mu}} \CA \big( 1 - h^{- i \frac{\sigma}{\mu}} \CQ_{0} \big)^{- 1} - h^{- i \frac{\sigma}{\mu}} \CB \Big) ,
\end{equation*}
and $r ( \lambda ) = \CO ( h^{\zeta} )$ uniformly for $\lambda \in \Omega$. From \eqref{b19} and \eqref{b5}, we deduce
\begin{equation} \label{b20}
s_{P , P_{0}}^{\prime} ( \lambda ) = \frac{1}{2 \pi i} \frac{\alpha^{\prime} ( \lambda ) + r^{\prime} ( \lambda )}{\alpha ( \lambda ) + r ( \lambda )} .
\end{equation}
Since $S ( \lambda )$ is unitary for $\lambda \in \R \cap \Omega$, we have $\vert \alpha ( \lambda ) \vert = 1 + \CO ( h^{\zeta} )$ uniformly for $\lambda \in \R \cap \Omega$. On the other hand, since $\det S ( \lambda )$ and $\alpha ( \lambda )$ are holomorphic functions in $\Omega$, so is $r ( \lambda )$. Taking first $\Omega$ slightly larger, the Cauchy formula implies $r^{\prime} ( \lambda ) = \CO ( h^{\zeta - 1} \vert \ln h \vert ) = \CO ( h^{- 1} )$ on $\R \cap \Omega$. Summing up \eqref{b20} becomes
\begin{equation} \label{b21}
s_{P , P_{0}}^{\prime} ( \lambda ) = \frac{1}{2 \pi i} \frac{\alpha^{\prime} ( \lambda )}{\alpha ( \lambda )} + \CO ( h^{- 1} ) ,
\end{equation}
uniformly for $\lambda \in \R \cap \Omega$. It remains to compute $\alpha^{\prime} ( \lambda )$. The derivative of $\CA$, $\CB$ and $\CQ_{0}$ with respect to $\lambda$ gives a term of order $\CO ( h^{- 1} )$, whereas
\begin{equation*}
\partial_{\lambda} h^{- i \frac{\sigma}{\mu}}  = i \frac{\vert \ln h \vert}{\mu h} h^{- i \frac{\sigma}{\mu}} ,
\end{equation*}
is of order $h^{- 1} \vert \ln h \vert$. Then, \eqref{b21} gives
\begin{equation*}
s_{P , P_{0}}^{\prime} ( \lambda ) = \frac{\vert \ln h \vert}{2 \pi \mu h} \frac{2 h^{- 2 i \frac{\sigma}{\mu}} \CA \big( 1 - h^{- i \frac{\sigma}{\mu}} \CQ_{0} \big)^{- 1} + h^{- 3 i \frac{\sigma}{\mu}} \CA \CQ_{0} \big( 1 - h^{- i \frac{\sigma}{\mu}} \CQ_{0} \big)^{- 2} - h^{- i \frac{\sigma}{\mu}} \CB}{h^{- 2 i \frac{\sigma}{\mu}} \CA \big( 1 - h^{- i \frac{\sigma}{\mu}} \CQ_{0} \big)^{- 1} - h^{- i \frac{\sigma}{\mu}} \CB} + \CO ( h^{- 1} ) ,
\end{equation*}
and the corollary follows.
\end{proof}
\end{example}

\section{Scattering amplitude} \label{s4}

In this part, we assume \ref{h1} with $\rho > 1$ and denote $P_{0} = - h^{2} \Delta$. Since $P$ is a short-range perturbation of $P_{0}$, the scattering matrix $S_{P , P_{0}} ( \lambda )$ associated to the pair $( P , P_{0} )$ at energy $\lambda > 0$ is a well-defined bounded operator on $L^{2} ( \S^{n - 1} )$. The scattering amplitude $S_{P , P_{0}} ( \lambda , \theta , \omega )$ is defined as the distribution kernel of $S_{P , P_{0}} ( \lambda )$. It happens that $S_{P , P_{0}} ( \lambda , \theta , \omega )$ is smooth outside the diagonal $\theta = \omega$ (see e.g. Isozaki and Kitada \cite{IsKi86_01}). As usual, $\omega \in \S^{n - 1}$ (resp. $\theta \in \S^{n - 1}$) is called the initial (resp. final) direction. Depending on the paper \cite{LaMa99_01,PeZw01_01,RoTa89_01,St02_01}, the scattering amplitude may be normalized in different ways (that is modulo a constant); we have chosen not to multiply by such a constant.

Under the assumptions which allow to define the resonances, we generally know that the scattering amplitude $\lambda \longmapsto S_{P , P_{0}} ( \lambda , \theta , \omega )$ has a meromorphic extension to (a part of) the complex plane with poles at the resonances. This has been first proved by Lax and Phillips \cite{LaPh67_01} for obstacle scattering and then by several authors in various situations. For instance, Agmon \cite{Ag86_01} has studied the Schr\"{o}dinger operators whose potential is analytic at infinity and satisfies \ref{h1} with $\rho > 1$. We refer to G{\'e}rard and Martinez \cite{GeMa89_02} for more details and references about this question.

We want to obtain the meromorphic extension, modulo small error terms, of the scattering amplitude under \ref{h1} with $\rho > 1$ only. Following the strategy of the previous sections, we first compare the scattering amplitudes associated to the pair $( P , P_{0} )$ and to the pair $( Q , P_{0} )$. As before, we could expect to have
\begin{equation} \label{b22}
S_{P , P_{0}} ( \lambda , \theta , \omega ) \approx S_{Q , P_{0}} ( \lambda , \theta , \omega ) ,
\end{equation}
for $Q = - h^{2} \Delta + W ( x )$ satisfying \ref{h3} with $R$ large enough. Unfortunately, there is no hope to have such a comparison result because the scattering amplitude crucially depends on the potential at infinity.

\begin{example}\rm \label{b27}
We give an example of operator showing that \eqref{b22} is in general far from being true. In dimension $n = 2$, we take $V ( x ) = E_{0} e^{- x^{2}}$ non-radially slightly perturbed at infinity. Such a potential is illustrated in Figure \ref{f10}. More generally, we could have chosen any potential $V$ non-radial at infinity, whose trapped set at energy $E_{0}$ is given by the unique, non-degenerate and isotropic global maximum of $V$, which satisfies the general assumptions of Section 2 of Alexandrova and two authors \cite{AlBoRa08_01} and whose Hamiltonian trajectories with energy $E_{0}$ behave at infinity ``nicely'' (that is like those of $E_{0} e^{- x^{2}}$).

\begin{figure}%[!h]
\begin{center}
\begin{picture}(0,0)%
\includegraphics{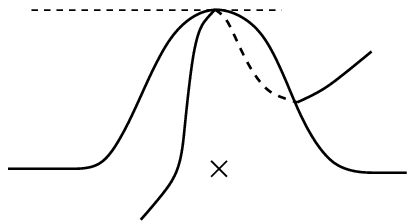}%
\end{picture}%
\setlength{\unitlength}{987sp}%
\begingroup\makeatletter\ifx\SetFigFont\undefined%
\gdef\SetFigFont#1#2#3#4#5{%
  \reset@font\fontsize{#1}{#2pt}%
  \fontfamily{#3}\fontseries{#4}\fontshape{#5}%
  \selectfont}%
\fi\endgroup%
\begin{picture}(8084,4120)(-8489,-4905)
\put(-4199,-4561){\makebox(0,0)[lb]{\smash{{\SetFigFont{9}{10.8}{\rmdefault}{\mddefault}{\updefault}$0$}}}}
\put(-8474,-1036){\makebox(0,0)[lb]{\smash{{\SetFigFont{9}{10.8}{\rmdefault}{\mddefault}{\updefault}$E_{0}$}}}}
\put(-2849,-1711){\makebox(0,0)[lb]{\smash{{\SetFigFont{9}{10.8}{\rmdefault}{\mddefault}{\updefault}$V ( x )$}}}}
\end{picture}%
$\qquad \qquad$
\begin{picture}(0,0)%
\includegraphics{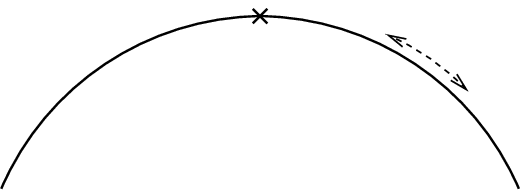}%
\end{picture}%
\setlength{\unitlength}{908sp}%
\begingroup\makeatletter\ifx\SetFigFont\undefined%
\gdef\SetFigFont#1#2#3#4#5{%
  \reset@font\fontsize{#1}{#2pt}%
  \fontfamily{#3}\fontseries{#4}\fontshape{#5}%
  \selectfont}%
\fi\endgroup%
\begin{picture}(10860,4314)(-5429,809)
\put(3751,3764){\makebox(0,0)[lb]{\smash{{\SetFigFont{9}{10.8}{\rmdefault}{\mddefault}{\updefault}$\theta \in \S^{1}$}}}}
\put(151,4964){\makebox(0,0)[b]{\smash{{\SetFigFont{9}{10.8}{\rmdefault}{\mddefault}{\updefault}$\omega^{\perp}$}}}}
\put(  1,3164){\makebox(0,0)[b]{\smash{{\SetFigFont{9}{10.8}{\rmdefault}{\mddefault}{\updefault}$h^{- 1 / 2} \vert \ln h \vert^{- 1 / 2}$}}}}
\put(-3524,1664){\makebox(0,0)[b]{\smash{{\SetFigFont{9}{10.8}{\rmdefault}{\mddefault}{\updefault}$h^{- 1 / 2}$}}}}
\put(3826,1664){\makebox(0,0)[b]{\smash{{\SetFigFont{9}{10.8}{\rmdefault}{\mddefault}{\updefault}$1$}}}}
\end{picture}%
\end{center}
\caption{The potential of Example \ref{b27} and the order of the scattering amplitude in function of the final direction $\theta \in \S^{1}$.} \label{f10}
\end{figure}

We fix an initial direction $\omega \in \S^{1}$. From Theorem 2.6 (a) and (c) of \cite{AlBoRa08_01}, there exists a particular direction $\omega^{\perp} \in \S^{1}$ such that
\begin{equation*}
S_{P , P_{0}} ( E_{0} , \theta , \omega ) \propto \left\{ \begin{aligned}
&1 &&\text{ for } \theta < \omega^{\perp} \text{ near } \omega^{\perp} , \\
&h^{- 1 / 2} \vert \ln h \vert^{- 1 / 2} &&\text{ for } \theta = \omega^{\perp} , \\
&h^{- 1 / 2} &&\text{ for } \theta > \omega^{\perp} \text{ near } \omega^{\perp} .
\end{aligned} \right.
\end{equation*}
Let us now consider $Q$ as in \ref{h3}. Depending on $R > 0$ large enough (and also of the function $g_{0}$ appearing in \ref{h3}), $S_{Q , P_{0}} ( E_{0} , \omega^{\perp} , \omega )$ can be of order $1$, $h^{- 1 / 2} \vert \ln h \vert^{- 1 / 2}$ or $h^{- 1 / 2}$ since the perturbation is non-radial. Thus, $S_{P , P_{0}} ( \lambda , \theta , \omega )$ and $S_{Q , P_{0}} ( \lambda , \theta , \omega )$ are not always of the same order and this order may depend on $R > 0$ arbitrarily large.
\end{example}

Even if it is not possible to relate the scattering amplitudes of $P$ and $Q$, it is possible to replace the resolvent of $P$ by that of $Q$ in the usual representation of the scattering amplitude. More precisely, we have

\begin{remark}[Representation formula for the scattering amplitude]\sl \label{b28}
Assume \ref{h1} with $\rho > 1$ and \ref{h3}. Let $I \subset ] 0 , + \infty [$ be a compact interval and $\theta , \omega \in \S^{n - 1}$ be such that $\theta \neq \omega$. For $R > 0$ large enough, we have
\begin{align}
S_{P , P_{0}} ( \lambda , \theta , \omega )
&= c_{n} \big\< [ P , \chi_{+} ] g_{+} e^{i \varphi_{+} / h} , ( P - \lambda - i 0 )^{- 1} [ P , \chi_{-} ] g_{-} e^{i \varphi_{-} / h} \big\> + \CO ( h^{\infty} ) \nonumber \\
&= c_{n} \big\< [ P , \chi_{+} ] g_{+} e^{i \varphi_{+} / h} , ( Q - \lambda - i 0 )^{- 1} [ P , \chi_{-} ] g_{-} e^{i \varphi_{-} / h} \big\> + \CO ( h^{\infty} ) , \label{b31}
\end{align}
uniformly for $\lambda \in I$. Here, the cut-off functions $\chi_{\pm} ( x ) \in C^{\infty}_{0} ( \R^{n} )$ satisfy $\chi_{-} \prec \chi_{+}$ and are equal to $1$ on a (arbitrary) large region,
\begin{equation} \label{b35}
c_{n} = i \pi ( 2 \pi h )^{- n} \lambda^{\frac{n - 2}{2}} ,
\end{equation}
the phases $\varphi_{-} = \varphi_{-} ( x , \omega , \lambda )$ and $\varphi_{+} = \varphi_{+} ( x , \theta , \lambda )$ are smooth in all their variables, and the amplitudes $g_{-} = g_{-} ( x , \omega , \lambda , h )$ and $g_{+} = g_{+} ( x , \theta , \lambda , h )$ are classical symbols in $x , \theta , \omega , \lambda$ of class $S ( 1 )$. Eventually, $\varphi_{\pm}$ and $g_{\pm}$ are given by the constructions of Isozaki and Kitada for the operator $P$ (see Sections 1-2 of \cite{RoTa89_01} or Sections 2-3 of \cite{Mi04_01} for more details).
\end{remark}

In \eqref{b31}, the phases $\varphi_{\pm}$ and the symbols $g_{\pm}$ depend on $P$ at infinity only, whereas the cut-off resolvent $[ P , \chi_{+} ] ( Q - \lambda - i 0 )^{- 1} [ P , \chi_{-} ]$ takes into account the operator $P$ (or $Q$) in a compact subset. Using the constructions of Isozaki and Kitada \cite{IsKi85_01}, the first equality of \eqref{b31} has been first obtained by Robert and Tamura \cite{RoTa89_01} in the non-trapping case and then by the second author in the general setting \cite[Theorem 3.2]{Mi04_01}. The potential is assumed to be analytic at infinity in \cite{Mi04_01}, but this assumption can be removed since \eqref{a17} is now known without analyticity. The second equality of \eqref{b31} is a direct consequence of the first one and of Proposition \ref{b29}.

We do not know if the scattering amplitude has a meromorphic extension under \ref{h1} with $\rho > 1$. However Remark \ref{b28} allows to write $S_{P , P_{0}} ( \lambda , \theta , \omega )$ has a meromorphic function controlled by the norm of the resolvent of $Q$ modulo a small remainder term.

\begin{proposition}[Almost meromorphic extension of the scattering amplitude]\sl \label{b32}
Let $P$ satisfy \ref{h1} with $\rho > 1$, $E_{0} > 0$ and $\theta , \omega \in \S^{n - 1}$ be such that $\theta \neq \omega$. There exists a function $S_{\rm approx} ( z , \theta , \omega )$ meromorphic in $B ( E_{0} , C h )$ for any $C > 0$ and $h$ small enough such that
\begin{equation*}
S_{P , P_{0}} ( \lambda , \theta , \omega ) = S_{\rm approx} ( \lambda , \theta , \omega ) + \CO ( h^{\infty} ) ,
\end{equation*}
uniformly for $\lambda \in [ E_{0} - C h , E_{0} + C h ]$ and
\begin{equation} \label{b34}
\vert S_{\rm approx} ( z , \theta , \omega ) \vert \lesssim h^{2 - n} \big\Vert \one_{\SC} ( Q - z )^{- 1} \one_{\SC} \big\Vert ,
\end{equation}
uniformly for $z \in B ( E_{0} ,  C h )$ where $\SC$ is an arbitrarily far away ring and $Q$ is as in \ref{h3} with $R$ large enough.
\end{proposition}

\begin{proof}
Since $\varphi_{+} ( x , \omega , \lambda )$ is $C^{\infty}$ in $x , \lambda$, the Taylor formula gives
\begin{equation*}
\varphi_{+} ( x , \omega , E_{0} + h \sigma ) \asymp \varphi_{+} ( x , \omega , E_{0} ) + \sum_{j \geq 1} \varphi_{+}^{j} ( x , \sigma ) h^{j} ,
\end{equation*}
in $S ( 1 )$ where the $\varphi_{+}^{j} ( x , \sigma )$ are $C^{\infty}$ in $x$ and polynomial in $\sigma \in [ - C , C ]$. Then
\begin{equation} \label{b33}
e^{i \varphi_{+} / h} \asymp e^{i \varphi_{+} ( x , \omega , E_{0} ) / h} e^{i \varphi_{+}^{1} ( x , \sigma )} \Big( 1 + \sum_{j \geq 1} f_{j} ( x , \sigma ) h^{j} \Big) ,
\end{equation}
in $S ( 1 )$  where the $f_{j} ( x , \sigma )$ are $C^{\infty}$ in $x$ and polynomial in $\sigma \in [ - C , C ]$. Using the Borel lemma, one can construct a function $f ( x , \sigma , h )$, which is $C^{\infty}$ in $x$ and holomorphic in $\sigma \in \C$, such that $f ( x , \sigma , h ) \asymp 1 + \sum_{j \geq 1} f_{j} ( x , \sigma ) h^{j}$ for $\sigma \in B ( 0 , C )$ for any $C > 0$. Thus, \eqref{b33} gives $e^{i \varphi_{+} / h} = F_{+} ( x , \lambda , h ) + S ( h^{\infty} )$ with
\begin{equation*}
F_{+} ( x , \lambda , h ) = e^{i \varphi_{+} ( x , \omega , E_{0} ) / h} e^{i \varphi_{+}^{1} \big( x , \frac{\lambda - E_{0}}{h} \big)} f \Big( x , \frac{\lambda - E_{0}}{h} , h \Big),
\end{equation*}
uniformly for $\lambda \in [ E_{0} - C h , E_{0} + C h ]$ for any $C > 0$. Note that $F_{+}$ is holomorphic and bounded in $B ( E_{0} , C h )$ for any $C > 0$. Doing the same procedure, we can write $e^{i \varphi_{-} / h} = F_{-} ( x , \lambda , h ) + S ( h^{\infty} )$ and $g_{\pm} = G_{\pm} ( x , \lambda , h ) + S ( h^{\infty} )$ uniformly for $\lambda \in [ E_{0} - C h , E_{0} + C h ]$ where the functions $F_{-} , G_{\pm}$ are holomorphic and bounded in $B ( E_{0} , C h )$.

Combining the previous constructions with \eqref{a17} for $Q$, \eqref{b31} becomes
\begin{equation*}
S_{P , P_{0}} ( \lambda , \theta , \omega ) = S_{\rm approx} ( \lambda , \theta , \omega ) + \CO ( h^{\infty} ) ,
\end{equation*}
uniformly for $\lambda \in [ E_{0} - C h , E_{0} + C h ]$ with
\begin{equation} \label{b36}
S_{\rm approx} ( \lambda , \theta , \omega ) = c_{n} \big\< [ P , \chi_{+} ] G_{+} F_{+} , ( Q - \lambda - i 0 )^{- 1} [ P , \chi_{-} ] F_{-} G_{-} \big\> .
\end{equation}
Since $F_{\pm} , G_{\pm}$ are holomorphic and the cut-off resolvent of $Q$ is meromorphic, $S_{\rm approx} ( z , \theta , \omega )$ has a meromorphic extension in $B ( E_{0} , C h )$ for any $C > 0$ and $h$ small enough. Taking $\SC$ such that $\supp \nabla \chi_{\pm} \prec \one_{\SC}$, \eqref{b34} follows from \eqref{b35}, \eqref{b36} and that $[ P , \chi_{\pm} ]$ are of order $h$.
\end{proof}

In the well in the island situation (see Example \ref{a61}) and for globally analytic potentials, Lahmar-Benbernou and Martinez \cite{Be99_01,LaMa99_01} have computed the asymptotic of the residue of the scattering amplitude at the resonances. This result allows them to compute the asymptotic of the scattering amplitude on the real axis. Using Remark \ref{b28} and adapting \cite[Corollary 2.1]{LaMa99_01}, it seems possible to remove the analyticity assumption in this last result. For that, we would have to use \cite{FuLaMa11_01} instead of \cite{HeSj86_01}. As this may be quite technical, we state no result in this direction. On the other hand, the residue of the scattering amplitude at barrier-top has been computed by Fujii\'e, Zerzeri and two authors \cite[Theorem 5.1]{BoFuRaZe11_01}.

\begin{figure}%[!h]
\begin{center}
\begin{picture}(0,0)%
\includegraphics{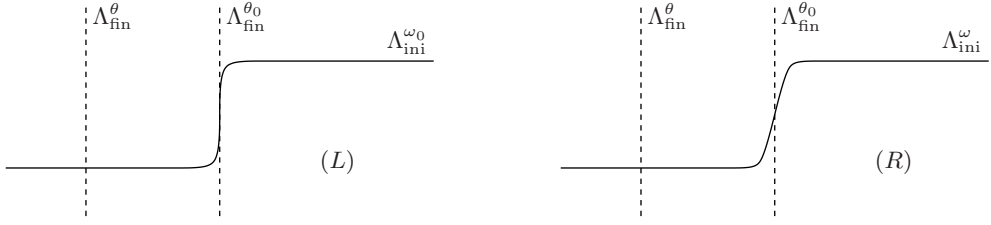}%
\end{picture}%
\setlength{\unitlength}{1105sp}%
\begingroup\makeatletter\ifx\SetFigFont\undefined%
\gdef\SetFigFont#1#2#3#4#5{%
  \reset@font\fontsize{#1}{#2pt}%
  \fontfamily{#3}\fontseries{#4}\fontshape{#5}%
  \selectfont}%
\fi\endgroup%
\begin{picture}(22116,4866)(-1232,-394)
\put(16201,4064){\makebox(0,0)[lb]{\smash{{\SetFigFont{9}{10.8}{\rmdefault}{\mddefault}{\updefault}$\Lambda_{\rm fin}^{\theta_{0}}$}}}}
\put(13201,4064){\makebox(0,0)[lb]{\smash{{\SetFigFont{9}{10.8}{\rmdefault}{\mddefault}{\updefault}$\Lambda_{\rm fin}^{\theta}$}}}}
\put(19801,3539){\makebox(0,0)[lb]{\smash{{\SetFigFont{9}{10.8}{\rmdefault}{\mddefault}{\updefault}$\Lambda_{\rm ini}^{\omega}$}}}}
\put(7351,3539){\makebox(0,0)[lb]{\smash{{\SetFigFont{9}{10.8}{\rmdefault}{\mddefault}{\updefault}$\Lambda_{\rm ini}^{\omega_{0}}$}}}}
\put(751,4064){\makebox(0,0)[lb]{\smash{{\SetFigFont{9}{10.8}{\rmdefault}{\mddefault}{\updefault}$\Lambda_{\rm fin}^{\theta}$}}}}
\put(3751,4064){\makebox(0,0)[lb]{\smash{{\SetFigFont{9}{10.8}{\rmdefault}{\mddefault}{\updefault}$\Lambda_{\rm fin}^{\theta_{0}}$}}}}
\put(5851,839){\makebox(0,0)[lb]{\smash{{\SetFigFont{9}{10.8}{\rmdefault}{\mddefault}{\updefault}$( L )$}}}}
\put(18301,839){\makebox(0,0)[lb]{\smash{{\SetFigFont{9}{10.8}{\rmdefault}{\mddefault}{\updefault}$( R )$}}}}
\end{picture}%
\end{center}
\caption{A cross section of the Lagrangian manifolds $\Lambda_{\bullet}^{\star}$.} \label{f12}
\end{figure}

We have seen previously that $S_{P , P_{0}} ( \lambda , \theta , \omega )$ and $S_{Q , P_{0}} ( \lambda , \theta , \omega )$ are not always comparable. Nevertheless, for any $\theta , \omega \in \S^{n - 1}$, one could hope to find two directions $\widetilde{\theta} ( \theta , \omega ) , \widetilde{\omega} ( \theta , \omega ) \in \S^{n - 1}$ close to $\theta , \omega$ such that
\begin{equation} \label{b55}
S_{P , P_{0}} ( \lambda , \theta , \omega ) \approx S_{Q , P_{0}} ( \lambda , \widetilde{\theta} , \widetilde{\omega} ) .
\end{equation}
This property is satisfied in Example \ref{b27}, but does not hold in general. Indeed, for a non-trapping energy $\lambda_{0} > 0$ in dimension $n  = 2$, Robert and Tamura \cite{RoTa89_01} have proved that the scattering amplitude is given by the scalar product of two Lagrangian distributions, associated to two Lagrangian manifolds $\Lambda_{\rm fin}^{\theta}$ and $\Lambda_{\rm ini}^{\omega}$ stable by the Hamiltonian flow. Assume now that these Lagrangian manifolds are as in Figure \ref{f12}; that is, the cross sections (normal to $H_{p}$) of $\Lambda_{\rm fin}^{\theta_{0}}$ and $\Lambda_{\rm ini}^{\omega_{0}}$ have a contact of order $3$ $(L)$ whereas they intersect transversally for $( \theta , \omega ) \neq ( \theta_{0} , \omega_{0} )$ $(R)$. Adapting \cite{RoTa89_01} and using Section 7.7 of H{\"o}rmander \cite{Ho90_01} to compute degenerate stationary phases, one can check that the scattering amplitude satisfies the asymptotic
\begin{equation} \label{b56}
S_{P , P_{0}} ( \lambda_{0} , \theta , \omega ) \propto \left\{
\begin{aligned}
&h^{- 3 / 4} &&\text{ for } ( \theta , \omega ) = ( \theta_{0} , \omega_{0} ) ,  \\
&h^{- 1 / 2} &&\text{ for } ( \theta , \omega ) \neq ( \theta_{0} , \omega_{0} ) .
\end{aligned} \right.
\end{equation}
If now we consider $Q$ as in \ref{h3} with $R$ large, it may happen that the cross section of the Lagrangian manifolds $\Lambda_{\rm fin}^{\theta}$ and $\Lambda_{\rm ini}^{\omega}$ for the operator $Q$ intersect transversally (see Figure \ref{f12} $(R)$) for all $( \theta , \omega )$ near $( \theta_{0} , \omega_{0} )$. Then,
\begin{equation} \label{b57}
S_{Q , P_{0}} ( \lambda_{0} , \theta , \omega ) \propto h^{- 1 / 2} ,
\end{equation}
for all $( \theta , \omega )$ near $( \theta_{0} , \omega_{0} )$. Eventually, \eqref{b56} and \eqref{b57} show that \eqref{b55} can not be true for $( \theta , \omega ) = ( \theta_{0} , \omega_{0} )$. The idea behind this counterexample is that the scattering amplitude, that is the distribution kernel of the scattering matrix, behaves badly in presence of caustics. Thus, it seems more relevant to consider the scattering matrix rather than its kernel. In the rest of this section, we obtain a positive result in this direction and link the scattering matrices $S_{P , P_{0}} ( \lambda )$ and $S_{Q , P_{0}} ( \lambda )$.

\begin{figure}%[!h]
\begin{center}
\begin{picture}(0,0)%
\includegraphics{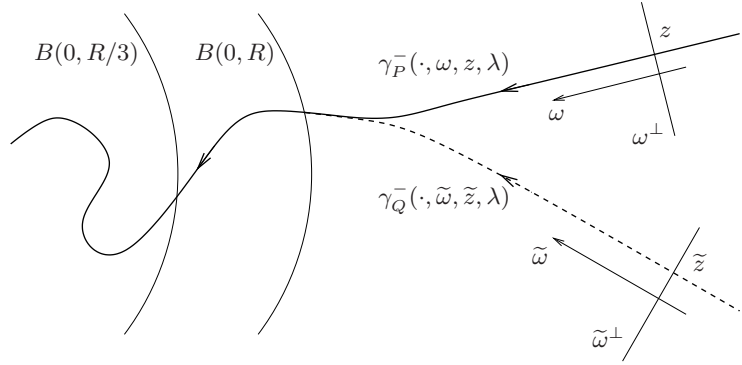}%
\end{picture}%
\setlength{\unitlength}{1105sp}%
\begingroup\makeatletter\ifx\SetFigFont\undefined%
\gdef\SetFigFont#1#2#3#4#5{%
  \reset@font\fontsize{#1}{#2pt}%
  \fontfamily{#3}\fontseries{#4}\fontshape{#5}%
  \selectfont}%
\fi\endgroup%
\begin{picture}(16416,8101)(1618,-3383)
\put(16951,-1336){\makebox(0,0)[lb]{\smash{{\SetFigFont{9}{10.8}{\rmdefault}{\mddefault}{\updefault}$\widetilde{z}$}}}}
\put(5776,3389){\makebox(0,0)[lb]{\smash{{\SetFigFont{9}{10.8}{\rmdefault}{\mddefault}{\updefault}$B ( 0 , R )$}}}}
\put(14701,-3061){\makebox(0,0)[lb]{\smash{{\SetFigFont{9}{10.8}{\rmdefault}{\mddefault}{\updefault}$\widetilde{\omega}^ {\perp}$}}}}
\put(15601,1514){\makebox(0,0)[lb]{\smash{{\SetFigFont{9}{10.8}{\rmdefault}{\mddefault}{\updefault}$\omega^ {\perp}$}}}}
\put(16201,3914){\makebox(0,0)[lb]{\smash{{\SetFigFont{9}{10.8}{\rmdefault}{\mddefault}{\updefault}$z$}}}}
\put(2176,3389){\makebox(0,0)[lb]{\smash{{\SetFigFont{9}{10.8}{\rmdefault}{\mddefault}{\updefault}$B ( 0 , R / 3 )$}}}}
\put(13351,-1186){\makebox(0,0)[lb]{\smash{{\SetFigFont{9}{10.8}{\rmdefault}{\mddefault}{\updefault}$\widetilde{\omega}$}}}}
\put(13726,2039){\makebox(0,0)[lb]{\smash{{\SetFigFont{9}{10.8}{\rmdefault}{\mddefault}{\updefault}$\omega$}}}}
\put(9901,164){\makebox(0,0)[lb]{\smash{{\SetFigFont{9}{10.8}{\rmdefault}{\mddefault}{\updefault}$\gamma_{Q}^{-} ( \cdot , \widetilde{\omega} , \widetilde{z} , \lambda )$}}}}
\put(9901,3164){\makebox(0,0)[lb]{\smash{{\SetFigFont{9}{10.8}{\rmdefault}{\mddefault}{\updefault}$\gamma_{P}^{-} ( \cdot , \omega , z , \lambda )$}}}}
\end{picture}%
\end{center}
\caption{The canonical relation $\Lambda^{\rm ini}_{Q \leftarrow P}$.} \label{f9}
\end{figure}

To state our result, we use some properties of the Hamiltonian flow at infinity whose proofs can be found in Alexandrova and two authors \cite{AlBoRa08_02}. For any energy $\lambda > 0$, asymptotic direction $\alpha \in \S^{n - 1}$ and impact parameter $z \in T^{*}_{\alpha} \S^{n - 1} \simeq \alpha^{\perp}$, there exist two trajectories
\begin{equation*}
\gamma_{P}^{\pm} ( t , \alpha , z , \lambda ) = ( x_{P}^{\pm} ( t , \alpha , z , \lambda ) , \xi_{P}^{\pm} ( t , \alpha , z , \lambda ) ) ,
\end{equation*}
of $H_{p}$ in $p ^{- 1} ( \lambda )$ such that
\begin{align*}
&\lim_{t \to \pm \infty} \big\vert x_{P}^{\pm} ( t , \alpha , z , \lambda ) - 2 \sqrt{\lambda} \alpha t - z \big\vert = 0 , \\
&\lim_{t \to \pm \infty} \big\vert \xi_{P}^{\pm} ( t , \alpha , z , \lambda ) - \sqrt{\lambda} \alpha \big\vert = 0 .
\end{align*}
The trajectories $\gamma_{Q}^{\pm}$ are defined the same way. We set
\begin{align*}
\Lambda^{\rm ini}_{Q \leftarrow P} ( \lambda ) = \big\{ \big( \widetilde{\omega} , - \sqrt{\lambda} \widetilde{z} , \omega , - \sqrt{\lambda} z \big) &\in T^{*} \S^{n - 1} \times T^{*} \S^{n - 1} ; \ \exists t_{P} , t_{Q}  \in \R , \\
&\gamma_{Q}^{-} ( t_{Q} , \widetilde{\omega} , \widetilde{z} , \lambda ) = \gamma_{P}^{-} ( t_{P} , \omega , z , \lambda ) \in B ( 0 , R / 3 ) \times \R^{n} \big\} .
\end{align*}
This relation is illustrated in Figure \ref{f9} and enjoys the following properties.

\begin{lemma}\sl \label{b24}
Let $M > 0$, $\CI \subset ] 0 , + \infty [$ be a compact interval and then $R > 0$ be large enough. for all $( \omega , z , \lambda ) \in \S^{n - 1} \times B ( 0 , M ) \times \CI$, there exists a unique $( \widetilde{\omega} , \widetilde{z} ) \in \S^{n - 1} \times \R^{n - 1}$ such that
\begin{equation*}
\big( \widetilde{\omega} , - \sqrt{\lambda} \widetilde{z} , \omega , - \sqrt{\lambda} z \big) \in \Lambda^{\rm ini}_{Q \leftarrow P} ( \lambda ) .
\end{equation*}
Moreover, for $\lambda \in \CI$ and restricted to $( \omega , z ) \in \S^{n - 1} \times B ( 0 , M )$, $\Lambda^{\rm ini}_{Q \leftarrow P} ( \lambda )$ is a canonical relation given by a canonical transformation.
\end{lemma}

\begin{proof}
A compactness argument gives that, for $R$ large enough, every trajectory $\gamma_{P}^{-} ( \cdot , \omega , z , \lambda )$ with $( \omega , z , \lambda ) \in \S^{n - 1} \times B ( 0 , M ) \times \CI$ reaches $B ( 0 , R / 4 ) \times \R^{n}$. Let $t_{P}$ denote the first time such that $x_{P}^{-} ( t_{P} , \omega , z , \lambda ) \in \partial B ( 0 , R / 4 ) = \{ x \in \R^{n} ; \ \vert x \vert = R / 4 \}$. If $R$ is large enough, $\gamma_{P}^{-} ( t_{P} , \omega , z , \lambda )$ belongs to an incoming region and $\exp ( t H_{q} ) ( \gamma_{P}^{-} ( t_{P} , \omega , z , \lambda ) )$ goes to infinity as $t \to - \infty$. This means that there exist $( \widetilde{\omega} , \widetilde{z} ) \in T^{*} \S^{n - 1}$ and $t_{Q} \in \R$ such that $\gamma_{P}^{-} ( t_{P} , \omega , z , \lambda ) = \gamma_{Q}^{-} ( t_{Q} , \widetilde{\omega} , \widetilde{z} , \lambda )$, that is $( \widetilde{\omega} , - \sqrt{\lambda} \widetilde{z} , \omega , - \sqrt{\lambda} z ) \in \Lambda^{\rm ini}_{Q \leftarrow P} ( \lambda )$.

On the other hand, let $\gamma ( t ) = ( x ( t ) , \xi ( t ) )$ be a Hamiltonian trajectory of $H_{p}$ in $p^{- 1} ( \CI )$ and let $R$ be large enough. If $x ( t _{1} ) , x ( t_{2} ) \in B ( 0 , R / 3 )$ for some $t_{1} < t_{2}$, then
\begin{equation} \label{b25}
x ( t ) \in B ( 0 , R ) ,
\end{equation}
for all $t_{1} \leq t \leq t_{2}$. Indeed, assume that \eqref{b25} does not hold. Then, there exist times $t_{1} \leq s_{1} < t < s_{2} \leq t_{2}$ such that $\vert x ( s_{\bullet} ) \vert = R / 3$, $\vert x ( t ) \vert \geq R$ and $\vert x ( s ) \vert \geq R / 3$ for all $s \in [ s_{1} , s_{2} ]$. The Hamiltonian equations yield
\begin{equation*}
s_{2} - s_{1} \geq \frac{2 R}{3 \sqrt{\lambda}} + o ( R ) .
\end{equation*}
Since $\partial_{t} ( x \cdot \xi ) = 2 \xi^{2} - x \cdot \nabla V ( x )$, we have
\begin{equation*}
x ( s_{2} ) \cdot \xi ( s_{2} ) \geq 2 \lambda ( s_{2} - s_{1} ) + x ( s_{1} ) \cdot \xi ( s_{1} ) + o_{R \to + \infty} ( s_{2} - s_{1} ) \geq \sqrt{\lambda} R + o ( R ) ,
\end{equation*}
whereas $\vert x ( s_{2} ) \cdot \xi ( s_{2} ) \vert \leq \sqrt{\lambda} R / 3 + o ( R )$. The two previous equations give a contradiction and \eqref{b25} follows. In particular, assume that $( \widetilde{\omega}_{1} , - \sqrt{\lambda} \widetilde{z}_{1} , \omega , - \sqrt{\lambda} z )$ and $( \widetilde{\omega}_{2} , - \sqrt{\lambda} \widetilde{z}_{2} , \omega , - \sqrt{\lambda} z )$ belong to $\Lambda^{\rm ini}_{Q \leftarrow P} ( \lambda )$ for some $\lambda \in \CI$. By definition, $\gamma_{P}^{-} ( t_{P}^{\bullet} , \omega , z , \lambda ) = \gamma_{Q}^{-} ( t_{Q}^{\bullet} , \widetilde{\omega}_{\bullet} , \widetilde{z}_{\bullet} , \lambda ) \in B ( 0 , R / 3 ) \times \R^{n}$ and we can suppose that $t_{P}^{1} \leq t_{P}^{2}$. By \eqref{b25}, $\gamma_{P}^{-} ( t , \omega , z , \lambda )$ stays for $t \in [ t_{P}^{1} , t_{P}^{2} ]$ in $B ( 0 , R ) \times \R^{n}$, where $P$ and $Q$ coincide. Then, the trajectories $\gamma_{Q}^{-} ( \cdot , \widetilde{\omega}_{1} , \widetilde{z}_{1} , \lambda )$ and $\gamma_{Q}^{-} ( \cdot , \widetilde{\omega}_{2} , \widetilde{z}_{2} , \lambda )$ are the same, and eventually $( \widetilde{\omega}_{1} , \widetilde{z}_{1} ) = ( \widetilde{\omega}_{2} , \widetilde{z}_{2} )$. This shows that, for $( \omega , z , \lambda ) \in \S^{n - 1} \times \R^{n - 1} \times \CI$, there exists at most one $( \widetilde{\omega} , \widetilde{z} ) \in \S^{n - 1} \times \R^{n - 1}$ such that $( \widetilde{\omega} , - \sqrt{\lambda} \widetilde{z} , \omega , - \sqrt{\lambda} z ) \in \Lambda^{\rm ini}_{Q \leftarrow P} ( \lambda )$.

Let us now define
\begin{equation} \label{b26}
C_{-}^{P} ( \lambda ) = \big\{ \big( x , \xi , \omega , - \sqrt{\lambda} z \big) \in T^{*} \R^{n} \times T^{*} \S^{n - 1} ; \ \exists t_{P} \in \R , \ ( x , \xi ) = \gamma_{P}^{-} ( t_{P} , \omega , z , \lambda ) \big\} .
\end{equation}
Restricted to $( \omega , z ) \in \S^{n - 1} \times B ( 0 , M)$, $C_{-}^{P} ( \lambda )$ is a canonical relation thanks to Lemma 4.2 of \cite{AlBoRa08_02}. With the definitions and notations of Section 25.2 of H\"{o}rmander \cite{Ho94_01} (see also the Appendix A.4 of \cite{AlBoRa08_02}), $C_{-}^{Q} ( \lambda )^{- 1} \times C_{-}^{P} ( \lambda )$ intersects cleanly $T^{*} \S^{n - 1} \times \Delta ( T^{*} \R^{n} ) \times T^{*} \S^{n - 1}$ with excess $e = 1$. By Theorem 21.2.14 of H\"{o}rmander \cite{Ho94_02}, 
\begin{equation*}
\Lambda^{\rm ini}_{Q \leftarrow P} ( \lambda ) = C_{-}^{Q} ( \lambda )^{- 1} \circ C_{-}^{P} ( \lambda ) ,
\end{equation*}
is a canonical relation after restriction to $( \omega , z ) \in \S^{n - 1} \times B ( 0 , M )$. Eventually, since $\Lambda^{\rm ini}_{Q \leftarrow P} ( \lambda )$ can be smoothly parametrizes by $( \omega , z ) \in \S^{n - 1} \times B ( 0 , M )$, this canonical relation is given by a canonical transformation.
\end{proof}

We also denote
\begin{align*}
\Lambda^{\rm fin}_{Q \leftarrow P} ( \lambda ) = \big\{ \big( \widetilde{\omega} , - \sqrt{\lambda} \widetilde{z} , \omega , - \sqrt{\lambda} z \big) &\in T^{*} \S^{n - 1} \times T^{*} \S^{n - 1} ; \ \exists t_{P} , t_{Q}  \in \R , \\
&\gamma_{Q}^{+} ( t_{Q} , \widetilde{\omega} , \widetilde{z} , \lambda ) = \gamma_{P}^{+} ( t_{P} , \omega , z , \lambda ) \in B ( 0 , R / 3 ) \times \R^{n} \big\} .
\end{align*}
Since the symbol $p ( x , \xi )$ of a Schr\"{o}dinger operator is even in $\xi$, $( x ( t ) , \xi ( t ) )$ is an integral curve of $H_{p}$ iff $( x ( - t ) , - \xi ( - t ) )$ is an integral curve of $H_{p}$. Then,
\begin{equation*}
( \widetilde{\omega} , \widetilde{\omega}^{*} , \omega , \omega^{*} ) \in \Lambda^{\rm ini}_{Q \leftarrow P} ( \lambda ) \quad \Longleftrightarrow \quad ( - \widetilde{\omega} , - \widetilde{\omega}^{*} , - \omega , - \omega^{*} ) \in \Lambda^{\rm fin}_{Q \leftarrow P} ( \lambda ) .
\end{equation*}
Inverting the rule of $P$ and $Q$ in the definition of $\Lambda^{\bullet}_{Q \leftarrow P} ( \lambda )$, for $\bullet = {\rm ini} , {\rm fin}$, one obtain a new canonical relation $\Lambda^{\bullet}_{P \leftarrow Q} ( \lambda ) = ( \Lambda^{\bullet}_{Q \leftarrow P} ( \lambda ) )^{- 1}$. The four relations $\Lambda_{\star}^{\bullet} ( \lambda )$ satisfy Lemma \ref{b24} mutatis mutandis.

We now define semiclassical Fourier integral operators. For the general theory of the FIOs in the classical setting, we refer to H\"{o}rmander \cite[Section 25.2]{Ho94_01}. The theory of the semiclassical FIOs can be found in the books of Ivrii \cite[Section 1.2]{Iv98_01}, Robert \cite{Ro87_01} or in the PhD thesis of Dozias \cite{Do94_01}. We follow the presentation of Alexandrova and two authors \cite[Appendix A.4]{AlBoRa08_02}. We restrict ourselves to the case of compactly supported operators on $\R^{n}$. Using local charts, the following definition can easily be extended to the case of compact manifolds.

\begin{definition}\sl \label{b58}
Let $r \in \R$, $\Lambda$ be a canonical relation from $T^{*} \R^{n}$ to $T^{*} \R^{m}$ and $A : L^{2}( \R^{n}) \rightarrow L^{2} ( \R^{m})$ be a linear operator bounded by $\CO (h^{- N})$, $N>0$. Then, $A$ is called a $h$-Fourier integral operator ($h$-FIO) with compactly supported symbol of order $r$ associated to $\Lambda$ and we note
\begin{equation*}
A \in \mathcal{I}_{h}^{r} ( \R^{m} \times \R^{n}, \Lambda ') ,
\end{equation*}
if, modulo an operator $\CO (h^{\infty})$, $A$ is a finite sum of operators of the form
\begin{equation} \label{b59}
h^{- r - \frac{n + m}{4} - \frac{d}{2}} \int_{\theta \in \R^{d}} e^{i \varphi ( x , y , \theta ) / h} a ( x , y , \theta ; h ) \, d \theta ,
\end{equation}
where the symbol $a \in S^0(1)$ has compact support in the variables $x, y , \theta$ (uniformly with respect to $h$) and the function $\varphi$ is a non-degenerate phase function defined near the support of $a$ with $\Lambda_{\varphi} {}' \subset \Lambda$.
\end{definition}

Using the previous geometric constructions and notations, the scattering matrices $S_{P , P_{0}} ( \lambda )$ and $S_{Q , P_{0}} ( \lambda )$ are connected by the following result.

\begin{theorem}[Comparison of scattering matrices]\sl \label{b23}
Assume \ref{h1} with $\rho > 1$ and \ref{h3}. Let $I \subset ] 0 , + \infty [$ be a compact interval and let $\chi_{\rm ini} , \chi_{\rm fin} \in C^{\infty} ( \S^{n - 1} )$ have disjoint support. For $R > 0$ large enough, there exists two $h$-FIOs
\begin{equation*}
U_{\rm fin} ( \lambda ) \in \mathcal{I}_{h}^{0} \big( \S^{n - 1} \times \S^{n - 1} , \Lambda^{\rm fin}_{P \leftarrow Q} ( \lambda ) {}^{\prime} \big) , \qquad U_{\rm ini} ( \lambda ) \in \mathcal{I}_{h}^{0} \big( \S^{n - 1} \times \S^{n - 1} , \Lambda^{\rm ini}_{Q \leftarrow P} ( \lambda ) {}^{\prime} \big) ,
\end{equation*}
with compactly supported symbols and smooth for $\lambda \in \I$, such that
\begin{equation*}
\chi_{\rm fin} S_{P , P_{0}} ( \lambda ) \chi_{\rm ini} = U_{\rm fin} ( \lambda ) S_{Q , P_{0}} ( \lambda ) U_{\rm ini} ( \lambda ) + \CO ( h^{\infty} ) ,
\end{equation*}
uniformly for $\lambda \in I$.
\end{theorem}

The two $h$-FIOs $U_{\rm ini}$ and $U_{\rm fin}$ ``exchange'' the quantum evolutions at infinity of $P$ and $Q$.  The cut-off functions $\chi_{\rm ini} , \chi_{\rm fin}$ allow to avoid the diagonal $\theta = \omega$. In the non-trapping case, Alexandrova \cite{Al06_01} (see also Alexandrova and two authors \cite[Theorem 2.5]{AlBoRa08_01}) has proved that the scattering matrix is an $h$-FIO outside the diagonal
\begin{equation} \label{b50}
S_{P , P_{0}} ( \lambda ) \in \mathcal{I}_{h}^{0} \big( \S^{n - 1} \times \S^{n - 1} , \Lambda_{\rm clas}^{P} ( \lambda ) {}^{\prime} \big) ,
\end{equation}
for the classical canonical relation
\begin{equation*}
\Lambda_{\rm clas}^{P} ( \lambda ) = \big\{ \big( \widetilde{\omega} , - \sqrt{\lambda} \widetilde{z} , \omega , - \sqrt{\lambda} z \big) \in T^{*} \S^{n - 1} \times T^{*} \S^{n - 1} ; \ \exists t \in \R , \ \gamma_{P}^{+} ( t , \widetilde{\omega} , \widetilde{z} , \lambda ) = \gamma_{P}^{-} ( t , \omega , z , \lambda ) \big\} .
\end{equation*}
Since the composition of canonical relations leads to $\Lambda_{\rm clas}^{P} = \Lambda^{\rm fin}_{P \leftarrow Q} \circ \Lambda_{\rm clas}^{Q} \circ \Lambda^{\rm ini}_{Q \leftarrow P}$, Theorem \ref{b23} can be deduced from \eqref{b50} in the non-trapping case. One can also verify that the microlocal structure of the scattering matrix at barrier-top stated in \cite[Theorem 2.4]{AlBoRa08_01} is compatible with Theorem \ref{b23}.

\begin{proof}
Since \eqref{b31} holds true uniformly for $( \theta , \omega ) \in \supp \chi_{\rm fin} \times \supp \chi_{\rm ini}$, the scattering matrix can be written as an operator on $L^{2} ( \S^{n - 1} )$
\begin{equation} \label{b37}
\chi_{\rm fin} S_{P , P_{0}} ( \lambda ) \chi_{\rm ini} = - c_{n} \chi_{\rm fin} ( K_{+}^{P} )^{*} [ P , \chi_{+} ] ( Q - \lambda - i 0 )^{- 1} [ P , \chi_{-} ] K_{-}^{P} \chi_{\rm ini} + \CO ( h^{\infty} ) ,
\end{equation}
where the kernels $K_{\pm}^{P} : L^{2} ( \S^{n - 1} ) \rightarrow L^{2}_{\rm loc} ( \R^{n} )$ are
\begin{align}
K_{+}^{P} ( x , \theta ) &= g_{+}^{P} ( x, \theta , \lambda , h ) e^{i \varphi_{+}^{P} ( x, \theta , \lambda ) / h} , \\
K_{-}^{P} ( x , \omega ) &= g_{-}^{P} ( x, \omega , \lambda , h ) e^{i \varphi_{-}^{P} ( x, \omega , \lambda ) / h} , \label{b46}
\end{align}
and the quantities $\varphi_{\pm}^{P}$ and $g_{\pm}^{P}$ are given by the constructions of Isozaki and Kitada for the operator $P$. In particular, $g_{\pm}^{P}$ are classical symbols supported in incoming/outgoing regions. Compared to Remark \ref{b28}, the dependency on the operator $P$ has been noted. Such a representation of the scattering matrix was used in Section 4.2 of \cite{AlBoRa08_02}.

We consider $K_{-}^{Q}$ and we are looking for a $h$-FIO $U_{\rm ini}$ such that $K_{-}^{Q} U_{\rm ini} = K_{-}^{P} \chi_{\rm ini} + \CO ( h^{\infty} )$. Here and in that follows, we work for $x$ in a small neighborhood of $\supp \nabla \chi_{-}$ and for $\omega$ near $\omega_{0} \in \S^{n - 1}$, say $\omega \in \Omega$. We first apply a $h$-FIO in order to adjust the canonical relations of the operators $K_{-}^{Q}$ and $K_{-}^{P}$. From \cite{AlBoRa08_02} (see (4.9) and Lemma 4.2), we have
\begin{equation} \label{b38}
K_{-}^{Q} \in \mathcal{I}_{h}^{-\frac{2 n +3}{4}} \big( \R^{n} \times \S^{n-1} , C_{-}^{Q} \, {}^{\prime} \big) ,
\end{equation}
with the canonical relation
\begin{equation*}
C_{-}^{Q} = \big\{ ( x , \xi , \omega , - \sqrt{\lambda} z ) ; \ \exists t \in \R , \ ( x , \xi ) = \gamma_{-}^{Q} ( t , \omega , z , \lambda ), \ ( x , \omega ) \in \supp \big( ( \nabla \chi_{-} ) g_{-}^{Q} \big) + B ( 0 , \varepsilon ) \big\} ,
\end{equation*}
for any $\varepsilon > 0$. We now take a $h$-FIO
\begin{equation*}
V \in \mathcal{I}_{h}^{0} \big( \S^{n - 1} \times \S^{n - 1} , \Lambda^{\rm ini}_{Q \leftarrow P} ( \lambda ) {}^{\prime} \big) ,
\end{equation*}
whose symbol is classical and elliptic near $\S^{n - 1} \times B ( 0 , M )$. Using that $C_{-}^{Q} \circ \Lambda^{\rm ini}_{Q \leftarrow P} ( \lambda ) = C_{-}^{P}$, the composition rules of the $h$-FIOs imply that
\begin{equation} \label{b49}
A : = K_{-}^{Q} V \in \mathcal{I}_{h}^{-\frac{2 n +3}{4}} \big( \R^{n} \times \S^{n-1} , C_{-}^{P} \, {}^{\prime} \big) .
\end{equation}
Since $\varphi_{-}^{P} ( \cdot , \cdot , \lambda )$ is a generating function of $C_{-}^{P}$, there exists a classical symbol $a ( x , \omega , h ) \asymp a_{0} ( x , \omega ) + a_{1} ( x , \omega ) h + \cdots$ such that the kernel of $A$ is given by
\begin{equation} \label{b39}
A ( x , \omega ) = a ( x , \omega , h ) e^{i \varphi_{-}^{P} ( x , \omega , \lambda ) / h} ,
\end{equation}
modulo $\CO ( h^{\infty} )$ in operator norm. At this stage, the phase of $A ( x , \omega )$ is that of $K_{-}^{P}$.

\begin{figure}%[!h]
\begin{center}
\begin{picture}(0,0)%
\includegraphics{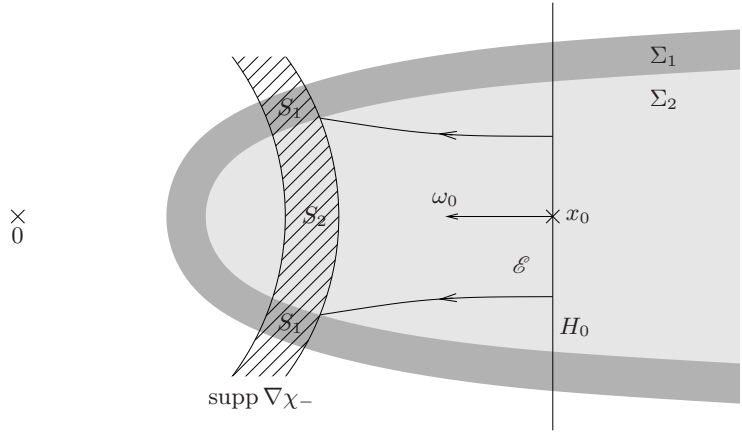}%
\end{picture}%
\setlength{\unitlength}{1105sp}%
\begingroup\makeatletter\ifx\SetFigFont\undefined%
\gdef\SetFigFont#1#2#3#4#5{%
  \reset@font\fontsize{#1}{#2pt}%
  \fontfamily{#3}\fontseries{#4}\fontshape{#5}%
  \selectfont}%
\fi\endgroup%
\begin{picture}(16373,9644)(-1371,-3983)
\put(4621,-1711){\makebox(0,0)[lb]{\smash{{\SetFigFont{9}{10.8}{\rmdefault}{\mddefault}{\updefault}$S_{1}$}}}}
\put(4621,3119){\makebox(0,0)[lb]{\smash{{\SetFigFont{9}{10.8}{\rmdefault}{\mddefault}{\updefault}$S_{1}$}}}}
\put(12976,4289){\makebox(0,0)[lb]{\smash{{\SetFigFont{9}{10.8}{\rmdefault}{\mddefault}{\updefault}$\Sigma_{1}$}}}}
\put(12976,3314){\makebox(0,0)[lb]{\smash{{\SetFigFont{9}{10.8}{\rmdefault}{\mddefault}{\updefault}$\Sigma_{2}$}}}}
\put(5176,689){\makebox(0,0)[lb]{\smash{{\SetFigFont{9}{10.8}{\rmdefault}{\mddefault}{\updefault}$S_{2}$}}}}
\put(-1199,239){\makebox(0,0)[b]{\smash{{\SetFigFont{9}{10.8}{\rmdefault}{\mddefault}{\updefault}$0$}}}}
\put(8101,1139){\makebox(0,0)[lb]{\smash{{\SetFigFont{9}{10.8}{\rmdefault}{\mddefault}{\updefault}$\omega_{0}$}}}}
\put(11101,764){\makebox(0,0)[lb]{\smash{{\SetFigFont{9}{10.8}{\rmdefault}{\mddefault}{\updefault}$x_{0}$}}}}
\put(10951,-1786){\makebox(0,0)[lb]{\smash{{\SetFigFont{9}{10.8}{\rmdefault}{\mddefault}{\updefault}$H_{0}$}}}}
\put(9901,-436){\makebox(0,0)[lb]{\smash{{\SetFigFont{9}{10.8}{\rmdefault}{\mddefault}{\updefault}$\SE$}}}}
\put(3076,-3361){\makebox(0,0)[lb]{\smash{{\SetFigFont{9}{10.8}{\rmdefault}{\mddefault}{\updefault}$\supp \nabla \chi_{-}$}}}}
\end{picture}%
\end{center}
\caption{The geometric setting in the proof of Theorem \ref{b23}.} \label{f11}
\end{figure}

We will now apply a pseudodifferential operator and use a propagation equation for $K_{-}^{\bullet}$ to adjust the symbols $a$ and $g_{-}^{\bullet}$. This can not be done directly since $C_{-}^{P}$ is not given by a canonical transformation. Let $r ( \omega , z , h ) \asymp r_{0} ( \omega , z ) + r_{1} ( \omega , z ) h + \cdots$ be a classical symbol in $C^{\infty}_{0} ( T^{*} \S^{n - 1} )$ and $R = \Op ( r )$. By the composition rules of a $h$-FIO by a pseudodifferential operator, \eqref{b49} yields
\begin{equation} \label{b47}
B : = K_{-}^{Q} V R = A R \in \mathcal{I}_{h}^{-\frac{2 n +3}{4}} \big( \R^{n} \times \S^{n-1} , C_{-}^{P} \, {}^{\prime} \big) .
\end{equation}
As before, there exists a classical symbol $b ( x , \omega , h ) \asymp b_{0} ( x , \omega ) + b_{1} ( x , \omega ) h + \cdots$ such that the kernel of $B$ is given by
\begin{equation} \label{b40}
B ( x , \omega ) = b ( x , \omega , h ) e^{i \varphi_{-}^{P} ( x , \omega , \lambda ) / h} ,
\end{equation}
modulo $\CO ( h^{\infty} )$ in operator norm. Moreover, for all $j \in \N$,
\begin{equation} \label{b43}
b_{j} ( x , \omega ) = a_{0} ( x , \omega ) r_{j} \big( \omega , - \nabla_{\omega} \varphi_{-}^{P} ( x , \omega , \lambda ) \big) + \widetilde{b}_{j} ( x , \omega ) ,
\end{equation}
where $\widetilde{b}_{j}$ depends only on $a$, $\varphi_{-}^{P}$ and on the $r_{k}$'s for $0 \leq k \leq j - 1$.

We collect properties of the symbol $g_{-}^{P}$ which can be found in Robert and Tamura \cite{RoTa89_01} and use some notations of this paper. First, $g_{-}^{P}$ is defined by $g_{-}^{P} ( x , \omega , \lambda , h ) = b_{-} ( x , \sqrt{\lambda} \omega , h )$ (see \cite[(2.1)]{RoTa89_01}) where the symbol $b_{-}$ is supported in $\Gamma_{-} ( 5 R_{0} , d_{3} , - \sigma_{4} )$ (see \cite[Page 165]{RoTa89_01}). Moreover, $b_{-} ( x , \xi , h ) = b_{- , 0} ( x , \xi ) + b_{- , 1} ( x , \xi ) h + \cdots$ satisfies the transport equations
\begin{equation} \label{b51}
2 \nabla_{x} \varphi_{-}^{P} \cdot \nabla_{x} b_{- , j} + ( \Delta \varphi_{-}^{P} ) b_{- , j} = i \Delta b_{- , j - 1} ,
\end{equation}
and the asymptotic conditions
\begin{equation} \label{b52}
b_{- , 0} \rightarrow 1 , \ b_{- , j} \rightarrow 1 \text{ for } j \geq 1 , \text{ as } \vert x \vert \rightarrow + \infty ,
\end{equation}
for $x \in \Gamma_{-} ( 6 R_{0} , d_{4} , - \sigma_{3} )$ (see \cite[Page 165]{RoTa89_01}). Let $\Sigma_{1} = \pi_{x} ( \Gamma_{-} ( 5 R_{0} , d_{3} , - \sigma_{4} ) \cap \{ \xi \in \sqrt{\lambda} \Omega \} ) \subset \R^{n}$ and $\Sigma_{2} = \pi_{x} ( \Gamma_{-} ( 6 R_{0} , d_{4} , - \sigma_{3} ) \cap \{ \xi \in \sqrt{\lambda} \Omega \} ) \subset \R^{n}$ (see Figure \ref{f11}). We also define $S_{1} = ( \Sigma_{1} \setminus \Sigma_{2} ) \cap \supp \nabla \chi_{-}$ and $S_{2} = \Sigma_{2} \cap \supp \nabla \chi_{-}$. Let $H_{0}$ be a hypersurface normal to $\omega_{0}$ near the support of $\nabla \chi_{-}$ as in Figure \ref{f11}. Finally, let $\SE$ denote the reunion of the integral curves of the vector field $\nabla_{x} \varphi_{-}^{P}$ in negative time starting from $S_{2}$ until they reach $\SE$. We assume that
\begin{equation} \label{b53}
\SE \subset \Sigma_{2} .
\end{equation}
Decreasing slightly $S_{2}$ and taking $R_{0}$ large enough, this can always be done.

The previous properties on $b_{-}$ imply that $g_{-}^{P} ( x , \omega , \lambda , h ) \asymp g_{- , 0}^{P} ( x , \omega ) + g_{- , 1}^{P} ( x , \omega ) h + \cdots$ is supported in $\Sigma_{1}$ and satisfies the usual transport equation
\begin{equation} \label{b42}
2 \nabla_{x} \varphi_{-}^{P} \cdot \nabla_{x} g_{- , j}^{P} + ( \Delta \varphi_{-}^{P} ) g_{- , j}^{P} = i \Delta g_{- , j - 1}^{P} ,
\end{equation}
near $\Sigma_{2} \times \Omega$. This shows that the constructions of Isozaki and Kitada are WKB solutions in $\Sigma_{2}$, that is
\begin{equation*}
( P - \lambda ) K_{-}^{P} = \CO ( h^{\infty} ) ,
\end{equation*}
in operator norm. Using that $P = Q$ near the support of $\nabla \chi_{-}$ where the different operators are considered, we also have
\begin{equation*}
( P - \lambda ) B = ( Q - \lambda ) K_{-}^{Q} V R = \CO ( h^{\infty} ) ,
\end{equation*}
which implies
\begin{equation} \label{b41}
2 \nabla_{x} \varphi_{-}^{P} \cdot \nabla_{x} b_{j} + ( \Delta \varphi_{-}^{P} ) b_{j} = i \Delta b_{j - 1} .
\end{equation}
From Proposition 2.4 of Isozaki and Kitada \cite{IsKi85_01},
\begin{equation} \label{b44}
\left\{
\begin{aligned}
&\S^{n - 1} \times H_{0} &&\longrightarrow &&T^{*} \S^{n - 1} \\
&( \omega , x ) &&\longmapsto &&\big( \omega , - \nabla_{\omega} \varphi_{-}^{P} ( x , \omega , \lambda ) \big)
\end{aligned} \right.
\end{equation}
is a local diffeomorphism near $\Omega \times ( H_{0} \cap \SE ) \subset \S^{n - 1} \times \R^{n}$. The functions $r_{j}$ are constructed inductively in the following way: assume that $r_{k}$ for $0 \leq k \leq j - 1$ have been arranged such that $b_{k} = g_{- , k}^{P}$ near $\SE \times \Omega$ for $0 \leq k \leq j - 1$. We choose $r_{j}$ such that
\begin{equation} \label{b45}
r_{j} \big( \omega , - \nabla_{\omega} \varphi_{-}^{P} ( x , \omega , \lambda ) \big) = \big( g_{- , j}^{P} ( x , \omega ) - \widetilde{b}_{j} ( x , \omega ) \big) / a_{0} ( x , \omega ) ,
\end{equation}
for $( x , \omega )$ near $( H_{0} \cap \SE ) \times \Omega$. This is possible since \eqref{b44} is a diffeomorphism, $\widetilde{b}_{j}$ does not depend on $r_{j}$ and $a_{0}$ is elliptic near $( H_{0} \cap \SE ) \times \Omega$ because the symbol of $V$ is elliptic and $g_{-}^{Q}$ satisfies \eqref{b52} mutatis mutandis. From \eqref{b45}, $b_{j}$ and $g_{- , j}^{P}$ coincide near $( H_{0} \cap \SE ) \times \Omega$. Moreover, these functions satisfy the same evolution equation near $\SE \times \Omega$ thanks to \eqref{b42}, \eqref{b41} and $b_{j - 1} = g_{- , j - 1}^{P}$ near $\SE \times \Omega$. Then, $b_{j} = g_{- , j}^{P}$ near $\SE \times \Omega$. Summing up,
\begin{equation} \label{b54}
b = g_{-}^{P} + \CO ( h^{\infty} ) ,
\end{equation}
near $S_{2} \times \Omega$.

Thus, \eqref{b46}, \eqref{b47}, \eqref{b40} and \eqref{b54} give
\begin{equation*}
[ P , \chi_{-} ] K_{-}^{P} = [ P , \chi_{-} ] K_{-}^{Q} V R + \widetilde{\chi}_{-} \CO ( h^{\infty} ) + \chi_{-}^{1} \CO ( 1 ) ,
\end{equation*}
for $\nabla \chi_{-} \prec \widetilde{\chi}_{-}$ and $\one_{S_{1}} \prec \chi_{-}^{1}$ supported near $\supp \nabla \chi_{-}$ and $S_{1}$ respectively. In particular,
\begin{equation} \label{b48}
[ P , \chi_{-} ] K_{-}^{P} \chi_{\rm ini} = [ P , \chi_{-} ] K_{-}^{Q} U_{\rm ini} ( \lambda ) + \widetilde{\chi}_{-} \CO ( h^{\infty} ) + \chi_{-}^{1} \CO ( 1 ) ,
\end{equation}
with
\begin{equation*}
U_{\rm ini} ( \lambda ) : = \widetilde{\chi}_{\rm ini} V R \chi_{\rm ini} \in \mathcal{I}_{h}^{0} \big( \S^{n - 1} \times \S^{n - 1} , \Lambda^{\rm ini}_{Q \leftarrow P} ( \lambda ) {}^{\prime} \big) ,
\end{equation*}
and $\one_{\pi_{\omega} ( \Lambda^{\rm ini}_{Q \leftarrow P} ( \lambda ) ( \pi_{\omega}^{- 1} ( \supp \chi_{\rm ini} ) ) )} \prec \widetilde{\chi}_{\rm ini} \in C^{\infty}_{0} ( \S^{n - 1} )$. Mimicking the proof of \eqref{b48}, there exists an operator
\begin{equation*}
U_{\rm fin} ( \lambda ) : = \widetilde{\chi}_{\rm fin} \widehat{V} \widehat{R} \chi_{\rm fin} \in \mathcal{I}_{h}^{0} \big( \S^{n - 1} \times \S^{n - 1} , \Lambda^{\rm fin}_{P \leftarrow Q} ( \lambda ) {}^{\prime} \big) ,
\end{equation*}
with $\one_{\pi_{\omega} ( \Lambda^{\rm fin}_{P \leftarrow Q} ( \lambda ) ( \pi_{\omega}^{- 1} ( \supp \chi_{\rm fin} ) ) )} \prec \widetilde{\chi}_{\rm fin} \in C^{\infty}_{0} ( \S^{n - 1} )$ such that
\begin{equation}
\chi_{\rm fin} ( K_{+}^{P} )^{*} [ P , \chi_{+} ] = U_{\rm fin} ( \lambda ) ( K_{+}^{Q} )^{*} [ P , \chi_{+} ] + \CO ( h^{\infty} ) \widetilde{\chi}_{+} + \CO ( 1 ) \chi_{+}^{1} ,
\end{equation}
for $\nabla \chi_{+} \prec \widetilde{\chi}_{+}$ and $\chi_{+}^{1}$ satisfying properties similar to $\widetilde{\chi}_{-}$ and $\chi_{-}^{1}$. Then, \eqref{b37} becomes
\begin{align} \label{b69}
\chi_{\rm fin} S_{P , P_{0}} ( \lambda ) \chi_{\rm ini} ={}& - c_{n} U_{\rm fin} ( \lambda ) ( K_{+}^{Q} )^{*} [ P , \chi_{+} ] ( Q - \lambda - i 0 )^{- 1} [ P , \chi_{-} ] K_{-}^{Q} U_{\rm ini} ( \lambda )  + \CO ( h^{\infty} )  \nonumber  \\
&+ \CO ( h^{\infty} ) \big\Vert U_{\rm fin} ( \lambda ) ( K_{+}^{Q} )^{*} [ P , \chi_{+} ] ( Q - \lambda - i 0 )^{- 1} \widetilde{\chi}_{-} \big\Vert  \nonumber  \\
&+ \CO ( h^{\infty} ) \big\Vert \widetilde{\chi}_{+} ( Q - \lambda - i 0 )^{- 1} [ P , \chi_{-} ] K_{-}^{Q} U_{\rm ini} ( \lambda ) \big\Vert  \nonumber  \\
&+ \CO ( h^{- n} ) \big\Vert U_{\rm fin} ( \lambda ) ( K_{+}^{Q} )^{*} [ P , \chi_{+} ] ( Q - \lambda - i 0 )^{- 1} \chi_{-}^{1} \big\Vert \\
&+ \CO ( h^{- n} ) \big\Vert \chi_{+}^{1} ( Q - \lambda - i 0 )^{- 1} [ P , \chi_{-} ] K_{-}^{Q} U_{\rm ini} ( \lambda ) \big\Vert \nonumber   \\
&+ \CO ( h^{\infty} ) \big\Vert \widetilde{\chi}_{+} ( Q - \lambda - i 0 )^{- 1} \chi_{-}^{1} \big\Vert + \CO ( h^{\infty} ) \big\Vert \chi_{+}^{1} ( Q - \lambda - i 0 )^{- 1} \widetilde{\chi}_{-} \big\Vert  \nonumber   \\
&+ \CO ( h^{\infty} ) \big\Vert \widetilde{\chi}_{+} ( Q - \lambda - i 0 )^{- 1} \widetilde{\chi}_{-} \big\Vert + \CO ( h^{- n} ) \big\Vert \chi_{+}^{1} ( Q - \lambda - i 0 )^{- 1} \chi_{-}^{1} \big\Vert . \nonumber
\end{align}
Since the supports of $\nabla \chi_{\pm}$, $\widetilde{\chi}_{\pm}$ and $\chi_{\pm}^{1}$ can be chosen arbitrary far away from the origin, \eqref{a17} applied to the operator $Q$ implies
\begin{equation} \label{b70}
\big\Vert U_{\rm fin} ( \lambda ) ( K_{+}^{Q} )^{*} [ P , \chi_{+} ] ( Q - \lambda - i 0 )^{- 1} \widetilde{\chi}_{-} \big\Vert + \big\Vert \widetilde{\chi}_{+} ( Q - \lambda - i 0 )^{- 1} [ P , \chi_{-} ] K_{-}^{Q} U_{\rm ini} ( \lambda ) \big\Vert  = \CO ( 1 ) ,
\end{equation}
and
\begin{equation} \label{b72}
 \big\Vert \widetilde{\chi}_{+} ( Q - \lambda - i 0 )^{- 1} \chi_{-}^{1} \big\Vert + \big\Vert \chi_{+}^{1} ( Q - \lambda - i 0 )^{- 1} \widetilde{\chi}_{-} \big\Vert + \big\Vert \widetilde{\chi}_{+} ( Q - \lambda - i 0 )^{- 1} \widetilde{\chi}_{-} \big\Vert = \CO ( h^{- 1} ) .
\end{equation}
On the other hand, $\chi_{-}^{1}$ (resp. $\chi_{+}^{1}$) is supported in an outgoing (resp. incoming) region and no Hamiltonian trajectory starting from the support of $\chi_{-}^{1}$ (resp. $\chi_{+}^{1}$) touches the microsupport of $U_{\rm fin} ( \lambda ) ( K_{+}^{Q} )^{*} [ Q , \chi_{+} ]$ or $\chi_{+}^{1}$ (resp. $[ Q , \chi_{-} ] K_{-}^{Q} U_{\rm ini} ( \lambda )$ or $\chi_{-}^{1}$) in positive (resp. negative) time because $\chi_{\rm ini}$ and $\chi_{\rm fin}$ have disjoint support. Then, the proof of Proposition 3.1 of \cite{Mi04_01} (or the proof of \eqref{a25} for an alternative approach) yields
\begin{gather}
\big\Vert U_{\rm fin} ( \lambda ) ( K_{+}^{Q} )^{*} [ P , \chi_{+} ] ( Q - \lambda - i 0 )^{- 1} \chi_{-}^{1} \big\Vert + \big\Vert \chi_{+}^{1} ( Q - \lambda - i 0 )^{- 1} [ P , \chi_{-} ] K_{-}^{Q} U_{\rm ini} ( \lambda ) \big\Vert \nonumber   \\
+ \big\Vert \chi_{+}^{1} ( Q - \lambda - i 0 )^{- 1} \chi_{-}^{1} \big\Vert = \CO ( h^{\infty} ) .  \label{b71}
\end{gather}
Combining \eqref{b69} with \eqref{b70}, \eqref{b72} and \eqref{b71} gives
\begin{equation*}
\chi_{\rm fin} S_{P , P_{0}} ( \lambda ) \chi_{\rm ini} = - c_{n} U_{\rm fin} ( \lambda ) ( K_{+}^{Q} )^{*} [ Q , \chi_{+} ] ( Q - \lambda - i 0 )^{- 1} [ Q , \chi_{-} ] K_{-}^{Q} U_{\rm ini} ( \lambda ) + \CO ( h^{\infty} ) .
\end{equation*}
Using Remark \ref{b28} for $Q$ and $\supp \widetilde{\chi}_{\rm ini} \cap \supp \widetilde{\chi}_{\rm fin} = \emptyset$, we eventually obtain
\begin{equation*}
\chi_{\rm fin} S_{P , P_{0}} ( \lambda ) \chi_{\rm ini} = U_{\rm fin} ( \lambda ) S_{Q , P_{0}} ( \lambda ) U_{\rm ini} ( \lambda ) + \CO ( h^{\infty} ) ,
\end{equation*}
and Theorem \ref{b23} follows.
\end{proof}

\section{Resolvent estimates} \label{s5}

We first recall that, thanks to Section 6 of Robert and Tamura \cite{RoTa87_01}, cut-off resolvents and truncated resolvents have equivalent norms. More precisely, their approach (together with the proof of Proposition 1.5 of \cite{BoPe13_01} which guaranties that the forecoming norms of the resolvent are at least like $h^{- 1}$) gives

\begin{proposition}\sl \label{a13}
Assume \ref{h1}, $0 < E_{1} < E_{2}$ and $s >  1 / 2$. Let $\chi \in C^{\infty}_{0} ( \R^{n} )$ be equal to $1$ on a sufficiently large neighborhood of $0$. Then, we have
\begin{equation*}
\big\Vert \< x \>^{- s} ( P - z )^{- 1} \< x \>^{- s} \big\Vert \lesssim \Vert \chi ( P - z )^{- 1} \chi \Vert \lesssim \big\Vert \< x \>^{- s} ( P - z )^{- 1} \< x \>^{- s} \big\Vert ,
\end{equation*}
uniformly for $z \in [ E_{1} , E_{2} ] + i [ - h , h ] \setminus \R$ and $h$ small enough.
\end{proposition}

In particular, the operator $Q$ defined in \ref{h3} satisfies such an estimate. Moreover, one can check that this proposition is valid uniformly for $R$ large enough. Notice that Proposition \ref{a13} shows that the norm of the weighted resolvent is essentially independent of $s > 1 / 2$.

We now recall a priori estimates on the resolvent that are valid without assumption on the trapped set or the dimension. Among the authors which have progressively proved them, we would like to cite Fern\'{a}ndez and Lavine \cite{FeLa90_01}, Burq \cite{Bu98_01,Bu02_02} and Vodev \cite{Vo02_01}. For a shorter proof and additional references, we send back the reader to Datchev \cite{Da14_01}.

\begin{theorem}\sl \label{a15}
Assume \ref{h1}, $0 < E_{1} < E_{2}$. There exist $C_{0} , R_{0} > 0$ such that, for $s > 1 / 2$,
\begin{gather}
\big\Vert \< x \>^{- s} ( P - \lambda \pm i 0 )^{- 1} \< x \>^{- s} \big\Vert \lesssim e^{C_{0} / h} , \label{a16}  \\
\big\Vert \< x \>^{- s} \one_{\vert x \vert > R_{0}} ( P - \lambda \pm i 0 )^{- 1}  \one_{\vert x \vert > R_{0}} \< x \>^{- s} \big\Vert \lesssim h^{- 1} , \label{a17}
\end{gather}
uniformly for $\lambda \in [ E_{1} , E_{2} ]$ and $h$ small enough.
\end{theorem}

Moreover, we will make use of the following additional result.

\begin{proposition}\sl \label{a14}
Assume \ref{h1}, $0 < E_{1} < E_{2}$ and $s >  1 / 2$. There exists $R_{0} > 0$ such that
\begin{equation*}
\big\Vert \< x \>^{- s} \one_{\vert x \vert > R_{0}} ( P - \lambda \pm i 0 )^{- 1} \< x \>^{- s} \big\Vert \lesssim h^{- 1 / 2} \big\Vert \< x \>^{- s} ( P - \lambda \pm i 0 )^{- 1} \< x \>^{- s} \big\Vert^{1 / 2} ,
\end{equation*}
uniformly for $\lambda \in [ E_{1} , E_{2} ]$ and $h$ small enough.
\end{proposition}

This upper bound is a kind of interpolation between \eqref{a16} and \eqref{a17}. It is sharp in the well in the island setting, see Theorem 2 of Datchev, Dyatlov and Zworski \cite{DaDyZw15_01}. Such estimate has been first proved by Stefanov \cite[proof of Theorem 3]{St02_01} for compactly supported perturbations of the Laplacian using Proposition 2.2 of Burq \cite{Bu98_01}. Then, Datchev and Vasy \cite{DaVa13_01} have proved it with the cut-off function $\one_{\vert x \vert > R_{0}}$ replaced by any plateau function vanishing near the base space projection of the trapped set, and for operators with polynomially bounded resolvent. One can recover their result combining Proposition \ref{a14} with the usual propagation of singularities. Nevertheless, there is no hope to replace in the general setting the function $\one_{\vert x \vert > R_{0}}$ by any cut-off function supported outside the trapped set (for instance, one can consider the well in the island situation \cite{DaDyZw15_01,FuLaMa11_01,HeSj86_01}). As previously, the operator $Q$ defined in \ref{h3} satisfies Proposition \ref{a14} uniformly for $R$ large enough.

\begin{proof}
We follow the approach of Stefanov \cite{St02_01}, but can no longer use Burq \cite{Bu98_01} since $P$ is not a compactly supported perturbation of $- h^{2} \Delta$. Instead, we use Theorem 2.2 of Cardoso and Vodev \cite{CaVo02_01}: for all $0 < s - 1 / 2 \ll 1$, there exist $C_{1} , C_{2} > 0$ such that
\begin{align}
\big\Vert & e^{(\varphi ( r ) - \varphi ( a ) ) / h}  v \big\Vert^{2}_{H^{1} ( [ R_{1} , a ] )} + \big\Vert \< x \>^{- s} v \big\Vert^{2}_{H^{1} ( [ a , + \infty [ )} \leq - C_{2} h \im \< \partial_{r} v , v \>_{L^{2} ( S_{a} )} \nonumber \\
&\qquad \quad + C_{1} h^{2} \big\Vert e^{(\varphi ( r ) - \varphi ( a ) ) / h} ( P - z ) v \big\Vert^{2}_{L^{2} ( [ R_{1}  , a ] )} + C_{1} h^{2} \big\Vert \< x \>^{s} ( P - z ) v \big\Vert^{2}_{L^{2} ( [ a , + \infty [ )} , \label{a18}
\end{align}
uniformly for $\re z \in [ E_{1} , E_{2} ]$, $0 < \im z < h^{2}$, $h$ small enough and any function $v \in H^{2} ( \R^{n} )$ with $\< x \>^{s} ( P - z ) v \in L^{2} ( \R^{n} )$ and $\supp v \subset \{ x \in \R^{n} ; \ \vert x \vert > R_{1} \}$. In this expression, $r = \vert x \vert$, $0 < R_{1} < a$ are sufficiently large, the different norms are taken for $r$ in the appropriate regions and $S_{a} = \{ x \in \R^{n} ; \ \vert x \vert = a \}$. This inequality was originally stated in the high frequency limit for Laplace--Beltrami operators (that is without potential). Nevertheless, the passage to the semicassical regime is straightforward. Moreover, for the generalization to semiclassical Schr\"{o}dinger operators with potentials, one can replace Proposition 2.3 and Proposition 2.4 of Cardoso and Vodev \cite{CaVo02_01} by Proposition 6.2 of Burq \cite{Bu02_02} and Proposition 2.3 of Cardoso, Popov and Vodev \cite{CaPoVo04_01} respectively. In particular, the radial phase function $\varphi$ is that of Burq \cite[Section 4.3]{Bu02_02}.

It is important to note that the non-decreasing radial function $\varphi$ can be chosen arbitrarily large. More precisely, for any $C > 0$, there exist a function $\varphi$ and a constant $a > R_{1}$ as in Section 4.3 of \cite{Bu02_02} such that \eqref{a18} holds and
\begin{equation} \label{a19}
\varphi ( a ) - \varphi ( R_{1} + 1 ) \geq C .
\end{equation}
Of course the constant $a$ (noted $R_{3}$ in \cite{Bu02_02}) depends on $C$. To achieve that goal, we first define a function $\widehat{\varphi}$ as in \cite{Bu02_02} such that $\widehat{\varphi}^{\prime} ( r ) = \kappa$ for $r \in [R_{1} , R_{2} + 1 / \ln \kappa ]$ with $R_{2} = \kappa^{- 3 / 2}$ and $0 < \kappa \ll 1$. Then, we glue $\widehat{\varphi}^{\prime}$ with $\Psi^{\prime} ( r ) = e^{1 / ( r - R_{2} )} \one_{r \leq R_{2}}$ as in \cite{Bu02_02}. This construction is valid since $22 \Psi^{\prime} \Psi^{\prime \prime} + 1 / r \geq 0$ for $r \in [R_{2} + 1 / \ln \kappa , R_{2} ]$ (see \cite[(4.24)]{Bu02_02}). Taking $\varphi$ as in \cite[(4.29)]{Bu02_02}, we obtain a function satisfying the required properties and $\varphi ( a ) - \varphi ( R_{1} + 2 ) \geq \kappa \kappa^{- 3 / 2} - M = \kappa^{- 1 / 2} - M$ for some constant $M > 0$. Eventually, we get \eqref{a19} by taking $\kappa$ small enough.

Let $\varphi$ satisfy \eqref{a19} with $C = C_{0} + 1$, $C_{0}$ being given by \eqref{a16}. We apply \eqref{a18} to the function $v = \psi ( P - z )^{- 1} \< x \>^{- s} u$ with $u \in L^{2} ( \R^{n} )$, $1 - \psi \in C^{\infty}_{0} ( \R^{n} )$ and $\supp \nabla \psi \subset \{ x ; \ \vert x \vert \in ] R_{1} ,R_{1} + 1 [ \}$. In particular,
\begin{equation*}
( P - z ) v = \psi \< x \>^{- s} u + [ P , \psi ] ( P - z )^{- 1} \< x \>^{- s} u ,
\end{equation*}
where the last term is localized in the support of $\nabla \psi$. Thus, the last two terms in \eqref{a18} are estimated by
\begin{align}
\big\Vert e^{(\varphi ( r ) - \varphi ( a ) ) / h} ( P & - z ) v \big\Vert^{2}_{L^{2} ( [ R_{1}  , a ] )} + \big\Vert \< x \>^{s} ( P - z ) v \big\Vert^{2}_{L^{2} ( [ a , + \infty [ )}   \nonumber \\
\lesssim{}& \big\Vert e^{(\varphi ( r ) - \varphi ( a ) ) / h} [ P , \psi ] ( P - z )^{- 1} \< x \>^{- s} u  \big\Vert^{2}_{L^{2} ( [ R_{1}  , R_{1} + 1 ] )} + \Vert u \Vert^{2}   \nonumber \\
\lesssim{}& e^{- 2 C / h} \big\Vert ( P + i )^{- 1} ( P + i ) [ P , \psi ] ( P - z )^{- 1} \< x \>^{- s} u  \big\Vert^{2} + \Vert u \Vert^{2}   \nonumber \\
\leq{}& e^{- 2 C / h} \big\Vert ( P + i )^{- 1} \big( ( z + i ) [ P , \psi ] + [ P , [ P , \psi ] ] \big ) ( P - z )^{- 1} \< x \>^{- s} u  \big\Vert^{2} \nonumber \\
&+ e^{- 2 C / h} \big\Vert ( P + i )^{- 1} [ P , \psi ] \< x \>^{- s} u  \big\Vert^{2} + \Vert u \Vert^{2}   \nonumber \\
\lesssim{}& e^{- 2 C / h} h \big\Vert \< x \>^{- s} ( P - z )^{- 1} \< x \>^{- s} u  \big\Vert^{2} + \Vert u \Vert^{2}  \nonumber \\
\lesssim{}& \Vert u \Vert^{2} ,  \label{a20}
\end{align}
thanks to \eqref{a16}, \eqref{a19} and $C = C_{0} + 1$. Using the Green formula and $v = w : = ( P - z )^{- 1} \< x \>^{- s} u$ near $S_{a}$, the first term in the right hand side of \eqref{a18} becomes
\begin{align}
- \im \< \partial_{r} v , v \>_{L^{2} ( S_{a} )} &= - \im \< \partial_{r} w , w \>_{L^{2} ( S_{a} )} \nonumber \\
&= - h^{- 2} \im \big\< ( P - z ) w , w \big\>_{L^{2} ( [ 0 , a ] )} - h^{- 2} \im z \Vert w \Vert^{2}_{L^{2} ( [ 0 , a ] )}  \nonumber \\
&\leq - h^{- 2} \im \big\< u , \< x \>^{- s} ( P - z )^{- 1} \< x \>^{- s} u \big\>_{L^{2} ( [ 0 , a ] )} \nonumber \\
&\leq h^{- 2} \big\Vert \< x \>^{- s} ( P - z )^{- 1} \< x \>^{- s} u \big\Vert \Vert u \Vert ,  \label{a21}
\end{align}
for $\im z > 0$. Combining \eqref{a18} with \eqref{a20} and \eqref{a21}, we get
\begin{align}
\big\Vert \< x \>^{- s} \one_{\vert x \vert > a} ( P - z )^{- 1} \< x \>^{- s} u \big\Vert^{2} &= \big\Vert \< x \>^{- s} v  \big\Vert_{L^{2} ( [ a , + \infty [ )}^{2} \nonumber \\
&\lesssim h^{- 1} \big\Vert \< x \>^{- s} ( P - z )^{- 1} \< x \>^{- s} u \big\Vert \Vert u \Vert + \Vert u \Vert^{2} \nonumber \\
&\lesssim h^{- 1} \big\Vert \< x \>^{- s} ( P - z )^{- 1} \< x \>^{- s} \big\Vert \Vert u \Vert^{2} , \label{a22}
\end{align}
for $\re z \in [ E_{1} , E_{2} ]$ and $0 < \im z < h^{2}$. For the last inequality, we have used the proof of Proposition 1.5 of \cite{BoPe13_01} which guaranties that $\Vert \< x \>^{- s} ( P - z )^{- 1} \< x \>^{- s} \Vert \gtrsim h^{- 1}$. Choosing $R_{0} = a$ and taking the limit $\im z \to 0$, \eqref{a22} implies Proposition \ref{a14} for $( P - \lambda - i 0 )^{- 1}$ and $0 < s - 1 / 2 \ll 1$. Using Proposition \ref{a14} and taking the complex conjugate, we obtain the same estimate for $( P - \lambda \pm i 0 )^{- 1}$ and all $s > 1 / 2$.
\end{proof}

\begin{theorem}[Comparison of resolvents]\sl \label{a6}
Assume \ref{h1}, \ref{h3}, $0 < E_{1} < E_{2}$, $s >  1 / 2$ and let $\chi \in C^{\infty}_{0} ( \R^{n} )$. For $R > 0$ large enough, we have
\begin{equation*}
\chi ( P - \lambda \pm i 0 )^{- 1} \chi = \chi ( Q  - \lambda \pm i 0 )^{- 1} \chi + \CO ( h^{\infty} ) \big\Vert \< x \>^{- s} ( Q - \lambda \pm i 0 )^{- 1} \< x \>^{- s} \big\Vert ,
\end{equation*}
uniformly for $\lambda \in [ E_{1} , E_{2} ]$ and $h$ small enough.
\end{theorem}

\begin{proof}
We first assume that $\chi \in C^{\infty}_{0} ( \R^{n} )$ is like in Proposition \ref{a13}. Consider $\widetilde{\chi} \in C^{\infty}_{0} ( \R^{n} )$ with $\chi \prec \widetilde{\chi} \prec \one_{B ( 0 , R )}$. Since $P = Q$ near $\widetilde{\chi}$, we get
\begin{equation*}
( Q - z ) \widetilde{\chi} ( P - z )^{- 1} = ( P - z ) \widetilde{\chi} ( P - z )^{- 1} = \widetilde{\chi} + [ P , \widetilde{\chi} ] ( P - z )^{- 1} ,
\end{equation*}
and then
\begin{equation} \label{a23}
\chi ( P - z )^{- 1} \chi = \chi ( Q - z )^{- 1} \chi + \chi ( Q - z )^{- 1} [ P , \widetilde{\chi} ] ( P - z )^{- 1} \chi ,
\end{equation}
for $\im z > 0$. Let $1 - \psi \in C^{\infty}_{0} ( \R^{n} )$ be such that $\nabla \widetilde{\chi} \prec \psi$. In order to deal with $[ P , \widetilde{\chi} ]$ which is a differential operator of order $1$ whose coefficients are supported near $\nabla \widetilde{\chi}$, we write
\begin{align*}
[ P , \widetilde{\chi} ] &= \psi ( P + i )^{- 1} ( P + i ) [ P , \widetilde{\chi} ]   \\
&= \psi ( P + i )^{- 1} [ P , \widetilde{\chi} ] ( P - z ) + \psi \big( ( P + i )^{- 1} [ P , \widetilde{\chi} ] ( z + i ) + ( P + i )^{- 1} [ P , [ P , \widetilde{\chi} ] ] \big) \psi .
\end{align*}
Since $[ P , \widetilde{\chi} ]  \chi = 0$, \eqref{a23} becomes
\begin{equation} \label{a24}
\chi ( P - z )^{- 1} \chi = \chi ( Q - z )^{- 1} \chi + \chi ( Q - z )^{- 1} \psi M \psi ( P - z )^{- 1} \chi ,
\end{equation}
where $M = ( P + i )^{- 1} [ P , \widetilde{\chi} ] ( z + i ) + ( P + i )^{- 1} [ P , [ P , \widetilde{\chi} ] ]$ is a pseudodifferential operator whose symbol $m \in S ( h )$ is localized near $\supp \nabla \widetilde{\chi} \times \R^{n}$ modulo $S ( h^{\infty} )$.

\begin{figure}%[!h]
\begin{center}
\begin{picture}(0,0)%
\includegraphics{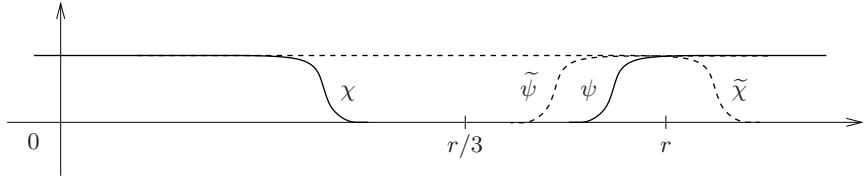}%
\end{picture}%
\setlength{\unitlength}{1105sp}%
\begingroup\makeatletter\ifx\SetFigFont\undefined%
\gdef\SetFigFont#1#2#3#4#5{%
  \reset@font\fontsize{#1}{#2pt}%
  \fontfamily{#3}\fontseries{#4}\fontshape{#5}%
  \selectfont}%
\fi\endgroup%
\begin{picture}(19244,3944)(-1221,-383)
\put(15226,1439){\makebox(0,0)[b]{\smash{{\SetFigFont{9}{10.8}{\rmdefault}{\mddefault}{\updefault}$\widetilde{\chi}$}}}}
\put(11851,1439){\makebox(0,0)[b]{\smash{{\SetFigFont{9}{10.8}{\rmdefault}{\mddefault}{\updefault}$\psi$}}}}
\put(10501,1439){\makebox(0,0)[b]{\smash{{\SetFigFont{9}{10.8}{\rmdefault}{\mddefault}{\updefault}$\widetilde{\psi}$}}}}
\put(6451,1439){\makebox(0,0)[b]{\smash{{\SetFigFont{9}{10.8}{\rmdefault}{\mddefault}{\updefault}$\chi$}}}}
\put(-599,239){\makebox(0,0)[b]{\smash{{\SetFigFont{9}{10.8}{\rmdefault}{\mddefault}{\updefault}$0$}}}}
\put(9076,164){\makebox(0,0)[b]{\smash{{\SetFigFont{9}{10.8}{\rmdefault}{\mddefault}{\updefault}$r / 3$}}}}
\put(13576,164){\makebox(0,0)[b]{\smash{{\SetFigFont{9}{10.8}{\rmdefault}{\mddefault}{\updefault}$r$}}}}
\end{picture}%
\end{center}
\caption{The various functions in the proof of Theorem \ref{a6}.} \label{f1}
\end{figure}

We now take the limit $z \to \lambda - i 0$ in \eqref{a24} and use some estimates of Robert and Tamura \cite{RoTa89_01} based on the constructions of Isozaki and Kitada \cite{IsKi85_01}. If $\widetilde{\chi} = 1$ near a sufficiently large neighborhood of $0$, we can write
\begin{equation} \label{a31}
M = \omega_{+} + \omega_{-} + M g ( P ) ( P - \lambda )^{- 1} ( P - \lambda ) + \Psi ( h^{\infty} \< x \>^{- \infty} ) ,
\end{equation}
where the symbols of $\omega_{\pm} \in \Psi ( h )$ are compactly supported in
\begin{equation*}
\Gamma_{\pm} = \big\{ ( x , \xi ) \in T^{*} \R^{n} ; \ \vert x \vert > r , \ E_{1} / 2 < p ( x , \xi ) < 2 E_{2} \text{ and } \pm \cos ( x , \xi ) > \mp 1 / 2 \big\} ,
\end{equation*}
$R > r \gg 1$ and $1 - g \in C^{\infty}_{0} ( \R )$ is equal to $1$ near $[ E_{1} , E_{2} ]$. The proof of Lemma 2.3 (iii) of Robert and Tamura \cite{RoTa89_01} shows that, if $r$ is large enough (that is if $\supp \nabla \widetilde{\chi}$ is sufficiently far away from $\supp \chi$), we have
\begin{equation} \label{a25}
\big\Vert \chi ( Q - \lambda - i 0  )^{- 1} \psi \omega_{+} \< x \>^{s} \big\Vert = \CO ( h^{\infty} ) \big\Vert \chi ( Q - \lambda - i 0 )^{- 1} \widetilde{\psi} \< x \>^{- s} \big\Vert ,
\end{equation}
with $1 - \widetilde{\psi} \in C^{\infty}_{0} ( \R^{n} )$ and $\one_{\{ \vert x \vert > r / 2 \}} \prec \widetilde{\psi} \prec \one_{\{ \vert x \vert > r / 3 \}}$. Their result was stated in the non-trapping regime and then without the norm in the right hand side of \eqref{a25}. But, adapted to the general setting, it gives \eqref{a25}. Indeed, we can write as in \cite[page 170]{RoTa89_01}
\begin{equation} \label{a26}
J_{+ c} ( h ) ( I_{h} ( e_{N} ( h ) ; \phi_{+} ) )^{*} = \psi \omega_{+} \< x \>^{s} - h^{N} \omega_{N} ,
\end{equation}
where $J_{+ c} ( h )$ was defined in \cite[pages 163-164]{RoTa89_01}. In particular, the image of the Fourier integral operator $J_{+ c} ( h )$ (and then also of $\omega_{N}$) is supported in $\{ \vert x \vert > r / 2 \}$ and $e_{N} \in S ( \< x \>^{- \infty} )$ since the symbols of $\omega_{+}$ is compactly supported. The constructions of Isozaki and Kitada used to obtain \eqref{a26} require no assumption on the trapping since they are made at infinity. Moreover, the phase function $\phi_{+}$, the symbols $c_{+} , e_{N}$ and the remainder term $\omega_{N}$ satisfy estimates uniform with respect to $R$ large enough. Next, we get as in \cite[page 171]{RoTa89_01}
\begin{equation} \label{a28}
( Q - \lambda - i 0 )^{- 1} \psi \omega_{+} \< x \>^{s} = \sum_{k = 1}^{3} Q_{k} ( \lambda , h ) ,
\end{equation}
where
\begin{align}
Q_{1} &= \frac{i}{h} \int_{0}^{\infty} e^{i t \lambda / h} U_{N} ( t ; h ) \, d t ,  \\
Q_{2} &= h^{N} ( Q - \lambda - i 0 )^{- 1} \omega_{N} , \\
Q_{3} &= \frac{i}{h} \int_{0}^{\infty} ( Q - \lambda - i 0 )^{- 1} e^{i s \lambda / h} R_{N} ( s ; h ) \, d s .
\end{align}
For $r$ large enough, no Hamiltonian trajectory starting from the support of the symbol of $\omega_{+}$ can touch the support of $\chi$ in positive time. Then, we deduce as in \cite{RoTa89_01}
\begin{equation*}
\chi Q_{1} = \CO ( h^{\infty} ) = \CO ( h^{\infty} ) \big\Vert \chi ( Q - \lambda - i 0 )^{- 1} \widetilde{\psi} \< x \>^{- s} \big\Vert .
\end{equation*}
The properties of the support of $\omega_{N} \in \Psi ( \< x \>^{- N} )$ imply directly
\begin{equation*}
\chi Q_{2} = \CO ( h^{N} ) \big\Vert \chi ( Q - \lambda - i 0 )^{- 1} \widetilde{\psi} \< x \>^{- s} \big\Vert .
\end{equation*}
Eventually, $R_{N} = \widetilde{\psi} R_{N}$ and $\Vert \< x \>^{s} R_{N} \Vert = \CO ( h^{N} \< t \>^{- 2} )$ from \cite[pages 163-164]{RoTa89_01}. Thus,
\begin{equation*}
\chi Q_{3} = \CO ( h^{N} ) \big\Vert \chi ( Q - \lambda - i 0 )^{- 1} \widetilde{\psi} \< x \>^{- s} \big\Vert .
\end{equation*}
Since $N$ can be taken arbitrarily large, the last inequalities together with \eqref{a28} prove \eqref{a25}. Instead of following Robert and Tamura \cite{RoTa89_01} to show \eqref{a25}, we could have used the arguments of the third author \cite[Section 3]{Mi04_01}. Coming back to the proof of Theorem \ref{a6},
\begin{align}
\chi ( Q - \lambda - i 0 )^{- 1} & \psi \omega_{+} \psi ( P - \lambda - i 0 )^{- 1} \chi \nonumber \\
&= \CO ( h^{\infty} ) \big\Vert \chi ( Q - \lambda - i 0 )^{- 1} \widetilde{\psi} \< x \>^{- s} \big\Vert \big\Vert  \< x \>^{- s} \widetilde{\psi} ( P - \lambda - i 0 )^{- 1} \chi \big\Vert ,  \label{a67}
\end{align}
uniformly for $R$ large enough. The same way, we have
\begin{align}
\chi ( Q - \lambda - i 0 )^{- 1} & \psi \omega_{-} \psi ( P - \lambda - i 0 )^{- 1} \chi \nonumber \\
&= \CO ( h^{\infty} ) \big\Vert \chi ( Q - \lambda - i 0 )^{- 1} \widetilde{\psi} \< x \>^{- s} \big\Vert \big\Vert  \< x \>^{- s} \widetilde{\psi} ( P - \lambda - i 0 )^{- 1} \chi \big\Vert .  \label{a29}
\end{align}
On the other hand, the semiclassical pseudodifferential calculus yields
\begin{align*}
M g ( P ) ( P - \lambda )^{- 1} & ( P - \lambda ) \psi ( P - \lambda - i 0 )^{- 1} \chi \\
&= M g ( P ) ( P - \lambda )^{- 1} \psi \chi + M g ( P ) ( P - \lambda )^{- 1} [ P , \psi ] ( P - \lambda - i 0 )^{- 1} \chi
\nonumber \\
&= \CO ( h^{\infty} ) \big\Vert  \< x \>^{- s} \widetilde{\psi} ( P - \lambda - i 0 )^{- 1} \chi \big\Vert .
\end{align*}
since $\psi \chi = 0$, $[ P , \psi ] = [ P , \psi ] \widetilde{\psi}$ and the support of the symbol of $M$ is away from the one of $[ P , \psi ] $. Thus,
\begin{align}
\chi ( Q - \lambda - i 0 )^{- 1} & \psi M g ( P ) ( P - \lambda )^{- 1} ( P - \lambda ) \psi ( P - \lambda - i 0 )^{- 1} \chi \nonumber \\
&= \CO ( h^{\infty} ) \big\Vert \chi ( Q - \lambda - i 0 )^{- 1} \widetilde{\psi} \< x \>^{- s} \big\Vert \big\Vert  \< x \>^{- s} \widetilde{\psi} ( P - \lambda - i 0 )^{- 1} \chi \big\Vert .  \label{a30}
\end{align}
Combining \eqref{a24} with \eqref{a31} and the estimates \eqref{a67}, \eqref{a29} and \eqref{a30}, we have shown
\begin{align}
\chi ( P - \lambda & - i 0 )^{- 1} \chi = \chi ( Q - \lambda - i 0 )^{- 1} \chi   \nonumber\\
&+ \CO ( h^{\infty} ) \big\Vert \chi ( Q - \lambda - i 0 )^{- 1} \one_{\vert x \vert > R_{0}} \< x \>^{- s} \big\Vert \big\Vert  \< x \>^{- s} \one_{\vert x \vert > R_{0}} ( P - \lambda - i 0 )^{- 1} \chi \big\Vert , \label{a32}
\end{align}
for $R$ large enough.

Applying Proposition \ref{a14}, the previous equation gives
\begin{align}
\chi ( P - \lambda - i 0 )^{- 1} \chi &= \chi ( Q - \lambda - i 0 )^{- 1} \chi + \CO ( h^{\infty} ) \big\Vert \< x \>^{- s} ( Q - \lambda - i 0 )^{- 1} \< x \>^{- s} \big\Vert  \nonumber\\
&\qquad \qquad \qquad \qquad \qquad \ + \CO ( h^{\infty} ) \big\Vert  \< x \>^{- s} ( P - \lambda - i 0 )^{- 1} \< x \>^{- s} \big\Vert . \label{a68}
\end{align}
Since we have assumed that $\chi \in C^{\infty}_{0} ( \R^{n} )$ is like in Proposition \ref{a13}, this yields
\begin{equation} \label{a33}
\left\{ \begin{aligned}
&\big\Vert \< x \>^{- s} ( P - \lambda \pm i 0  )^{- 1} \< x \>^{- s} \big\Vert \lesssim \big\Vert \< x \>^{- s} ( Q - \lambda \pm i 0  )^{- 1} \< x \>^{- s} \big\Vert , \\
&\big\Vert \< x \>^{- s} ( Q - \lambda \pm i 0  )^{- 1} \< x \>^{- s} \big\Vert \lesssim \big\Vert \< x \>^{- s} ( P - \lambda \pm i 0  )^{- 1} \< x \>^{- s} \big\Vert .
\end{aligned} \right.
\end{equation}
Finally, \eqref{a68} and \eqref{a33} imply Theorem \ref{a6} when $\chi = 1$ on a sufficiently large neighborhood of $0$. The general case follows directly.
\end{proof}

As a byproduct of the proof of Theorem \ref{a6} (see \eqref{a33}), we have the following result.

\begin{corollary}\sl \label{a80}
Assume \ref{h1}, \ref{h3}, $0 < E_{1} < E_{2}$ and $s >  1 / 2$. For $R > 0$ large enough,
\begin{equation*}
\big\Vert \< x \>^{- s} ( Q - \lambda \pm i 0  )^{- 1} \< x \>^{- s} \big\Vert \lesssim \big\Vert \< x \>^{- s} ( P - \lambda \pm i 0  )^{- 1} \< x \>^{- s} \big\Vert \lesssim \big\Vert \< x \>^{- s} ( Q - \lambda \pm i 0  )^{- 1} \< x \>^{- s} \big\Vert ,
\end{equation*}
uniformly for $\lambda \in [ E_{1} , E_{2} ]$ and $h$ small enough.
\end{corollary}

If the cut-off function $\chi$ is supported sufficiently far away from the origin, Theorem \ref{a6} can be strengthened in the spirit of \eqref{a17}. More precisely,

\begin{proposition}\sl \label{b29}
Assume \ref{h1}, \ref{h3} and $0 < E_{1} < E_{2}$. There exists $R_{0} > 0$ such that, for all $\chi \in C^{\infty}_{0} ( \R^{n} )$ and then $R > 0$ large enough, we have 
\begin{equation*}
\chi \one_{\vert x \vert > R_{0}} ( P - \lambda \pm i 0 )^{- 1} \one_{\vert x \vert > R_{0}} \chi = \chi \one_{\vert x \vert > R_{0}} ( Q  - \lambda \pm i 0 )^{- 1} \one_{\vert x \vert > R_{0}} \chi + \CO ( h^{\infty} ) ,
\end{equation*}
uniformly for $\lambda \in [ E_{1} , E_{2} ]$ and $h$ small enough.
\end{proposition}

\begin{proof}
We adapt the proof of Theorem \ref{a6}. As in \eqref{a24}, we have
\begin{align}
\chi \one_{\vert x \vert > R_{0}} ( P - z )^{- 1} \one_{\vert x \vert > R_{0}} \chi = \chi \one_{\vert x \vert > R_{0}} & ( Q - z )^{- 1} \one_{\vert x \vert > R_{0}} \chi  \nonumber \\
&+ \chi \one_{\vert x \vert > R_{0}} ( Q - z )^{- 1} \psi M \psi ( P - z )^{- 1} \one_{\vert x \vert > R_{0}} \chi , \label{b30}
\end{align}
with the same operator $M$. If $R_{0} , r$ are large enough, \eqref{a25} becomes
\begin{equation*}
\big\Vert \chi \one_{\vert x \vert > R_{0}} ( Q - \lambda - i 0  )^{- 1} \psi \omega_{+} \< x \>^{s} \big\Vert = \CO ( h^{\infty} ) \big\Vert \chi \one_{\vert x \vert > R_{0}} ( Q - \lambda - i 0 )^{- 1} \widetilde{\psi} \< x \>^{- s} \big\Vert = \CO ( h^{\infty} ) ,
\end{equation*}
thanks to \eqref{a17}. Thus, the remainder terms in \eqref{a67}, \eqref{a29} and \eqref{a30} can be replaced by $\CO ( h^{\infty} )$ when $\chi$ is replaced by $\chi \one_{\vert x \vert > R_{0}}$. Eventually, the proposition follows from \eqref{b30}.
\end{proof}

\section{Proof of the additional results of Sections \ref{s2} and \ref{s3}} \label{s6}

In this part, we prove some secondary results of Sections \ref{s2} and \ref{s3}. By comparison with the proof of Theorems \ref{a3} and \ref{a74}, we will only use here semiclassical microlocal analysis.

\begin{proof}[Proof of Proposition \ref{a2}]
We can assume $t \geq 0$. Consider $g \in C^{\infty}_{0} ( T^{*} \R^{n} )$ satisfying $\one_{F_{p} ( \chi , \varphi )} \prec g$. Since $p = q$ near the support of $g$, the pseudodifferential calculus gives
\begin{align*}
\partial_{s}\big( e^{- i s P / h} \Op ( g ) e^{i s Q / h} \big) &= - \frac{i}{h} e^{- i s P / h} \big( P \Op ( g ) - \Op ( g ) Q \big) e^{i s Q / h} \\
&= - \frac{i}{h} e^{- i s P / h} [ P , \Op ( g ) ] e^{i s Q / h} + \CO ( h^{\infty} ) ,
\end{align*}
uniformly for $s \in \R$. This gives the Duhamel formula
\begin{equation*}
e^{- i t P / h} \Op ( g ) = \Op ( g ) e^{- i t Q / h} - \frac{i}{h} \int_{0}^{t} e^{- i s P / h} [ P , \Op ( g ) ] e^{- i ( t - s ) Q / h} d s + \CO ( t h^{\infty} ) ,
\end{equation*}
and eventually
\begin{align}
\chi e^{- i t P / h} \varphi ( P ) \chi ={}& \chi e^{- i t P / h} \Op ( g ) \varphi ( P ) \chi + \CO ( h^{\infty} ) \nonumber \\
={}& \chi \Op ( g ) e^{- i t Q / h} \varphi ( P ) \chi   \nonumber \\
&- \frac{i}{h} \int_{0}^{t} \chi e^{- i s P / h} [ P , \Op ( g ) ] e^{- i ( t - s ) Q / h} \varphi ( P ) \chi d s + \CO ( t h^{\infty} ) \nonumber  \\
={}& \chi e^{- i t Q / h} \varphi ( Q ) \chi + \CR + \CO ( t h^{\infty} ) ,  \label{a34}
\end{align}
uniformly for $t \in \R$ where
\begin{equation*}
\CR = - \frac{i}{h} \int_{0}^{t} \chi e^{- i s P / h} [ P , \Op ( g ) ] e^{- i ( t - s ) Q / h} \varphi ( Q ) \chi d s .
\end{equation*}

Let us define the set
\begin{equation*}
G = \big\{ \rho \in \supp ( \nabla g ) ; \ \exp ( s H_{q} ) ( \rho ) \in \supp ( \chi \varphi ( q ) ) \text{ for some } s < 0 \big\} .
\end{equation*}
One can verify that $G$ is a compact subset of $T^{*} \R^{n}$. Moreover, if the support of $g$ is close enough to $\pi_{x} ( F_{p} ( \chi , \varphi ) )$, the assumption \ref{h2} implies that one can replace the condition $\exp ( s H_{q} ) ( \rho ) \in \supp ( \chi \varphi ( q ) )$ by $\exp ( s H_{p} ) ( \rho ) \in \supp ( \chi \varphi ( q ) )$ in the definition of $G$. Finally, since $\chi \varphi ( q ) = \chi \varphi ( p )$ and the support of $\nabla g$ is disjoint from $F_{p} ( \chi , \varphi )$, we have $\exp ( t H_{p} ) ( \rho ) \notin \supp \chi$ for all $\rho \in G$ and $t \geq 0$. In particular, it implies that $\exp ( t H_{p} ) ( \rho ) \to \infty$ as $t \to + \infty$ for all $\rho \in G$ because $K_{p} ( \supp \varphi ) \Subset \supp \chi$. Thus let $\psi \in C^{\infty}_{0} ( T^{*} \R^{n} )$ be such that $\one_{G} \prec \psi$ and, for all $\rho \in \supp \psi$, $\exp ( t H_{p} ) ( \rho )$ goes to infinity without touching the support of $\chi$ in positive time.

We first treat
\begin{equation} \label{a41}
J_{1} = \chi e^{- i s P / h} [ P , \Op ( g ) ] \Op ( \psi ) .
\end{equation}
Since no Hamiltonian trajectory of $p$ starting from the support of $\psi$ touches the support of $\chi$ in positive time, the microlocal analysis provides $J_{1} = \CO ( h^{\infty} )$ uniformly for $s$ on any compact of $[ 0 , + \infty [$. For the large values of $s$, we use the Isozaki--Kitada constructions. Let $\omega_{+} ( x , \xi )$ be a smooth function satisfying $\partial_{x}^{\alpha} \partial_{\xi}^{\beta} \omega_{+} ( x , \xi ) = \CO ( \< x \>^{- \vert \alpha \vert} \< \xi \>^{- \infty} )$ for all $\alpha , \beta \in \N^{n}$, supported in some outgoing region
\begin{equation} \label{a96}
\big\{ ( x , \xi ) \in T^{*} \R^{n} ; \ \vert x \vert > R , \ 1 / C < \vert \xi \vert < C \text{ and } \cos ( x , \xi ) > 1 / 2 \big\} ,
\end{equation}
for some $C > 0$ and $ R \gg 1$ such that $\omega_{+} = 1$ near $\exp ( S H_{p} ) ( \supp \psi )$ for (any) $S$ large enough. We can write
\begin{equation} \label{a42}
J_{1} = \chi e^{- i ( s - S ) P / h} \Op ( \omega_{+} ) e^{- i S P / h} [ P , \Op ( g ) ] \Op ( \psi ) + \CO ( h^{\infty} ) ,
\end{equation}
uniformly for $s \geq S$. In this setting, Robert and Tamura \cite[page 171]{RoTa89_01} (or \cite[Section 5]{RoTa87_01}) have constructed a parametrix $U_{N} ( t )$ for $e^{- i t P / h} \Op ( \omega_{+} )$, $t \geq 0$ up to $\CO ( h^{N} )$ for any $N \in \N$. Then, we get for any $N \in \N$
\begin{equation} \label{a36}
J_{1} = \chi U_{N} e^{- i s P / h} [ P , \Op ( g ) ] \Op ( \psi ) + \CO ( h^{N} ) = \CO ( h^{N} ) ,
\end{equation}
uniformly for $s \geq S$. The last equality is obtained as in \cite[Lemma 2.3 (iii)]{RoTa89_01} and use that no Hamiltonian trajectory of $p$ starting from the support of $\omega_{+}$ touches the support of $\chi$ in positive time. Summing up, we have proved
\begin{equation} \label{a35}
\chi e^{- i s P / h} [ P , \Op ( g ) ] \Op ( \psi ) e^{- i ( t - s ) Q / h} \varphi ( Q ) \chi = \CO ( h^{\infty} ) ,
\end{equation}
uniformly for $0 \leq s \leq t$.

We now consider
\begin{equation} \label{a44}
J_{2} = \chi \varphi ( Q ) e^{- i u Q / h} \Op ( 1 - \psi ) [ P , \Op ( g ) ] = \chi \varphi ( Q ) e^{- i u Q / h} \widetilde{\varphi} ( Q ) \Op ( 1 - \psi ) [ P , \Op ( g ) ] ,
\end{equation}
for $u < 0$ and $\varphi \prec \widetilde{\varphi} \in C^{\infty}_{0} ( ] 0 , + \infty [ )$. By the definition of $\psi$, there is no Hamiltonian trajectory of $q$ from the support of $( 1 - \psi ) \nabla g$ to the support of $\chi \varphi ( q )$ in negative time. Thus, we get $J_{2} = \CO ( h^{\infty} )$ uniformly for $u$ on any compact of $] - \infty , 0 ]$. Moreover, mimicking the proof of \eqref{a36}, one can show that, for all $N \in \N$, $J_{2} = \CO ( h^{N} )$ uniformly for $u \leq U \ll 0$. Summing up and taking the adjoint, we deduce
\begin{equation} \label{a37}
\chi e^{- i s P / h} [ P , \Op ( g ) ] \Op ( 1 - \psi ) e^{- i ( t - s ) Q / h} \varphi ( Q ) \chi = \CO ( h^{\infty} ) ,
\end{equation}
uniformly for $0 \leq s \leq t$.

Eventually, \eqref{a34} with \eqref{a35} and \eqref{a37} give $\CR = \CO ( t h^{\infty} )$ and
\begin{align}
\chi e^{- i t P / h} \varphi ( P ) \chi = \chi e^{- i t Q / h} \varphi ( Q ) \chi + \CO ( t h^{\infty} ) = \chi e^{- i t Q / h} \varphi ( Q ) \chi + \CO ( h^{\infty} ) ,
\end{align}
uniformly for $0 \leq t \leq h^{- C}$.
\end{proof}

\begin{proof}[Proof of Proposition \ref{a38}]
We mix the proofs of Proposition \ref{a2} and Theorem \ref{a3}. In particular, we will use the quantities $g , G , \psi , \omega_{+}$ constructed in the proof of Proposition \ref{a2}. Since $p = q$ near the support of $g \in C^{\infty}_{0} ( T^{*} \R^{n} )$, we get
\begin{align}
( P - z )^{- 1} \Op ( g ) ( Q - z ) &= ( P - z )^{- 1} \Op ( g ) ( P - z ) + ( P - z )^{- 1} \< x \>^{- 1} \CO ( h^{\infty} ) \< x \>^{- 1}  \nonumber \\
&= \Op ( g ) + ( P - z )^{- 1} [ \Op ( g ) , P ] + ( P - z )^{- 1} \< x \>^{- 1} \CO ( h^{\infty} ) \< x \>^{- 1}  , \label{a39}
\end{align}
for $\im z > 0$. Let $\varphi \prec \widetilde{\varphi} \in C^{\infty}_{0} ( ] 0 , +  \infty [ )$. If the support of $\widetilde{\varphi}$ is close enough to the support of $\varphi$, the pseudodifferential calculus gives
\begin{align*}
\widetilde{\varphi} ( P ) \chi &= \Op ( g ) \widetilde{\varphi} ( P ) \chi + \< x \>^{- 1} \CO ( h^{\infty} ) , \\
\chi \Op ( g ) \widetilde{\varphi} ( Q ) &= \chi \widetilde{\varphi} ( Q ) + \CO ( h^{\infty} ) \< x \>^{- 1} ,
\end{align*}
and \ref{h2} implies $\widetilde{\varphi} ( P ) \chi = \widetilde{\varphi} ( Q ) \chi + \< x \>^{- 1} \CO ( h^{\infty} )$. Moreover, since $V$ and $W$ coincide near the support of $\chi$, we have
\begin{equation*}
\chi ( P - z )^{- 1} ( 1 - \widetilde{\varphi} ) ( P ) \chi = \chi ( Q - z )^{- 1} ( 1 - \widetilde{\varphi} ) ( Q ) \chi +\CO ( h^{\infty} ) ,
\end{equation*}
uniformly for $\re z \in \supp \varphi$. Thus, combining \eqref{a39} with the previous equations and the polynomial bounds on the weighted resolvent of $P$ and $Q$, we deduce
\begin{align}
\chi ( P - z )^{- 1} \chi ={}& \chi ( P - z )^{- 1} \widetilde{\varphi} ( P ) \chi + \chi ( P - z )^{- 1} ( 1 - \widetilde{\varphi} ) ( P ) \chi    \nonumber \\
={}& \chi ( P - z )^{- 1} \Op ( g ) \widetilde{\varphi} ( P ) \chi + \chi ( Q - z )^{- 1} ( 1 - \widetilde{\varphi} ) ( Q ) \chi + \CO ( h^{\infty} ) \nonumber \\
={}& \chi \Op ( g ) ( Q - z )^{- 1} \widetilde{\varphi} ( P ) \chi + \chi ( Q - z )^{- 1} ( 1 - \widetilde{\varphi} ) ( Q ) \chi  \nonumber \\
&+ \chi ( P - z )^{- 1} [ \Op ( g ) , P ] ( Q - z )^{- 1} \widetilde{\varphi} ( P ) \chi + \CO ( h^{\infty} )   \nonumber \\
={}& \chi \Op ( g ) \widetilde{\varphi} ( Q ) ( Q - z )^{- 1} \chi + \chi ( 1 - \widetilde{\varphi} ) ( Q ) ( Q - z )^{- 1} \chi + \CR + \CO ( h^{\infty} )  \nonumber \\
={}& \chi ( Q - z )^{- 1} \chi + \CR + \CO ( h^{\infty} ) , \label{a45}
\end{align}
uniformly for $\re z \in \supp \varphi$, $\im z > 0$ with
\begin{equation*}
\CR = \chi ( P - z )^{- 1} [ \Op ( g ) , P ] ( Q - z )^{- 1} \widetilde{\varphi} ( Q ) \chi .
\end{equation*}

We now estimate the operator $\CR$. As in \eqref{a41}, we define
\begin{equation} \label{a49}
K_{1} = \chi ( P - z )^{- 1} [ \Op ( g ) , P ] \Op ( \psi ) = K_{1}^{\rm loc} + K_{1}^{\rm inf} ,
\end{equation}
where
\begin{align*}
K_{1}^{\rm loc} &= \frac{i}{h} \int_{0}^{S} \chi e^{- i s ( P - z ) / h} [ \Op ( g ) , P ] \Op ( \psi ) \, d s ,  \\
K_{1}^{\rm inf} &= \chi ( P - z )^{- 1} e^{- i S ( P - z ) / h} [ \Op ( g ) , P ] \Op ( \psi ) .
\end{align*}
From the discussion above \eqref{a41}, there is no Hamiltonian trajectory of $p$ starting from the support of $\psi$ and touching the support of $\chi$ in positive time. Therefore, $K_{1}^{\rm loc} = \CO ( h^{\infty} )$. On the other hand, we have
\begin{equation*}
K_{1}^{\rm inf} = \chi ( P - z )^{- 1} \Op ( \omega_{+} ) e^{- i S ( P - z ) / h} [ \Op ( g ) , P ] \Op ( \psi ) + \CO ( h^{\infty} ) ,
\end{equation*}
as in \eqref{a42}. Now, Lemma 2.3 (iii) of Robert and Tamura \cite{RoTa89_01} yields $\chi ( P - z )^{- 1} \Op ( \omega_{+} ) = \CO ( h^{\infty} )$. This result was originally proved in the non-trapping regime but extends directly to operators with polynomially bounded resolvent. Summing up, we have just shown
\begin{equation} \label{a43}
\chi ( P - z )^{- 1} [ \Op ( g ) , P ] \Op ( \psi ) ( Q - z )^{- 1} \widetilde{\varphi} ( Q ) \chi = \CO ( h^{\infty} ) .
\end{equation}
It remains to study
\begin{equation}
K_{2} = \chi \widetilde{\varphi} ( Q ) ( Q - \overline{z} )^{- 1} \Op ( 1 - \psi ) [ \Op ( g ) , P ] ,
\end{equation}
(see \eqref{a44}). Working as in \eqref{a37} and \eqref{a43}, we get
\begin{equation} \label{a52}
[ \Op ( g ) , P ] \Op ( 1 - \psi ) ( Q - z )^{- 1} \widetilde{\varphi} ( Q ) \chi = K_{2}^{*} = \CO ( h^{\infty} ) ,
\end{equation}
since there is no Hamiltonian trajectory of $q$ from the support of $\chi \widetilde{\varphi} ( q )$ to the support of $( 1 - \psi ) \nabla g$ in positive time. Thus,
\begin{equation*}
\chi ( P - z )^{- 1} [ \Op ( g ) , P ] \Op ( 1 - \psi ) ( Q - z )^{- 1} \widetilde{\varphi} ( Q ) \chi = \CO ( h^{\infty} ) ,
\end{equation*}
and eventually
\begin{equation} \label{a46}
\CR = \CO ( h^{\infty} ) ,
\end{equation}
uniformly for $\re z \in \supp \varphi$, $\im z > 0$.

Using \eqref{a45} with \eqref{a46} and taking the adjoint, we deduce
\begin{equation}
\chi ( P - \lambda \pm i 0 )^{- 1} \chi = \chi ( Q - \lambda \pm i 0 )^{- 1} \chi + \CO ( h^{\infty} ) ,
\end{equation}
uniformly for $\lambda \in \supp \varphi$. Then, Proposition \ref{a38} follows from this equation and the Stone formula \eqref{a7}.
\end{proof}

\begin{proof}[Proof of Remark \ref{a47}]
To show this result, we use some elements of the proof of Proposition \ref{a2} and Proposition \ref{a38}. In particular, we need the functions $\widetilde{\varphi} , g , \psi , \omega_{+}$ build there. Let us assume that $Q$ satisfies \ref{h4}, $s > 1 / 2$ and $\im z > 0$. We are looking for a parametrix of $( P - z )^{- 1}$. Consider $P_{0} = P - i h f ( x )$ where $f \in C^{\infty}_{0} ( \R^{n} ; [ 0 , 1 ] )$ satisfies $\pi_{x} ( K_{p} ( \supp \varphi ) ) \prec f \prec \chi$. For such dissipative operators, Royer \cite[Theorem 1.1]{Ro10_01} has proved that
\begin{equation} \label{a55}
\big\Vert \< x \>^{- s} ( P_{0} - \lambda \pm i 0  )^{- 1} \< x \>^{- s} \big\Vert \lesssim h^{- 1} ,
\end{equation}
uniformly for $\lambda \in \supp \widetilde{\varphi}$. Let us define
\begin{equation*}
A_{1} = ( 1 - \chi ) \widetilde{\varphi} ( P ) ( P_{0} - z )^{- 1} .
\end{equation*}
Since $P$ and $P_{0}$ coincide near the support of $1 - \chi$, we have
\begin{align}
( P - z ) A_{1} &= ( 1 - \chi ) \widetilde{\varphi} ( P ) ( P - z ) ( P_{0} - z )^{- 1} - [ P , \chi ] \widetilde{\varphi} ( P ) ( P_{0} - z )^{- 1} \nonumber\\
&= \widetilde{\varphi} ( P ) - B_{1} + \< x \>^{- s} \CO ( h^{\infty} ) \< x \>^{s} , \label{a48}
\end{align}
with $B_{1} = \chi \widetilde{\varphi} ( P ) + [ P , \chi ] \widetilde{\varphi} ( P ) ( P_{0} - z )^{- 1}$. We then choose
\begin{equation*}
A_{2} = \Op ( g ) ( Q - z )^{- 1} B_{1} .
\end{equation*}
Using that the symbols of $P$ and $Q$ coincide near the support of $g$ and since $g = 1$ near the support of $\chi \widetilde{\varphi} ( p )$, we get
\begin{align}
( P - z ) A_{2} ={}& ( Q  - z ) \Op ( g ) ( Q - z )^{- 1} B_{1} + \< x \>^{- s} \CO ( h^{\infty} ) \< x \>^{s}   \nonumber \\
={}& B_{1} + [ Q , \Op ( g ) ] \Op ( \psi ) ( Q - z )^{- 1} B_{1}  \nonumber \\
&+ [ Q , \Op ( g ) ] \Op ( 1 - \psi ) ( Q - z )^{- 1} B_{1} + \< x \>^{- s} \CO ( h^{\infty} ) \< x \>^{s}   \nonumber \\
={}& B_{1} - B_{2} + \< x \>^{- s} \CO ( h^{\infty} ) \< x \>^{s}  \label{a50}
\end{align}
with $B_{2} =  - [ Q , \Op ( g ) ] \Op ( \psi ) ( Q - z )^{- 1} B_{1}$. To eliminate the term containing $\Op ( 1 - \psi )$, we have used an estimate similar to \eqref{a52}. As a parametrix for $( P - z )^{- 1} B_{2}$, we choose
\begin{equation*}
A_{3} = \frac{i}{h} \int_{0}^{S} e^{- i s ( P - z ) / h} B_{2} \, d s + ( 1 - \widetilde{\chi} ) ( P_{0} - z )^{- 1} \Op ( \omega _{+} ) e^{- i S ( P - z ) / h} B_{2} .
\end{equation*}
following \eqref{a49}. Here $g \prec \widetilde{\chi} \in C^{\infty}_{0} ( \R^{n} )$ is supported near $\supp g$. In particular, the support of $\omega_{+}$ is far away from that of $\widetilde{\chi}$. As in \eqref{a42}, we have $( 1 - \widetilde{\chi} ) \Op ( \omega _{+} ) e^{- i S ( P - z ) / h} B_{2} = e^{- i S ( P - z ) / h} B_{2} + \< x \>^{- s} \CO ( h^{\infty} ) \< x \>^{s}$. Then,
\begin{align}
( P - z ) A_{3} &= \frac{i}{h} \int_{0}^{S} ( P - z ) e^{- i s ( P - z ) / h} B_{2} \, d s + ( 1 - \widetilde{\chi} ) \Op ( \omega _{+} ) e^{- i S ( P - z ) / h} B_{2} - B_{3}    \nonumber \\
&= \frac{i}{h} \int_{0}^{S} ( P - z ) e^{- i s ( P - z ) / h} B_{2} \, d s + e^{- i S ( P - z ) / h} B_{2} - B_{3} + \< x \>^{- s} \CO ( h^{\infty} ) \< x \>^{s}  \nonumber \\
&= B_{2} - B_{3} + \< x \>^{- s} \CO ( h^{\infty} ) \< x \>^{s} ,    \label{a51}
\end{align}
with $B_{3} = [ P , \widetilde{\chi} ] ( P_{0} - z )^{- 1} \Op ( \omega _{+} ) e^{- i S ( P - z ) / h} B_{2}$. Since there is no Hamiltonian trajectory from the outgoing region $\supp \omega_{+}$ to the support of $\nabla \widetilde{\chi}$ in positive time, Lemma 2.3 (iii) of Robert and Tamura \cite{RoTa89_01} gives as in \eqref{a43}
\begin{equation} \label{a53}
B_{3} = \< x \>^{- s} \CO ( h^{\infty} ) \< x \>^{s} .
\end{equation}
That $P_{0}$ is dissipative does not create any difficulty to prove this estimate but we have used here \eqref{a55}. From \eqref{a48}, \eqref{a50}, \eqref{a51} and \eqref{a53}, we get
\begin{equation*}
( P - z ) ( A_{1} + A_{2} + A_{3} ) = \widetilde{\varphi} ( P ) + \< x \>^{- s} \CO ( h^{\infty} ) \< x \>^{s} ,
\end{equation*}
which can be written
\begin{equation*}
\< x \>^{- s} ( P - z )^{- 1} \widetilde{\varphi} ( P ) \< x \>^{- s} = \< x \>^{- s} ( A_{1} + A_{2} + A_{3} ) \< x \>^{- s} + \< x \>^{- s} (P - z )^{- 1} \< x \>^{- s} \CO ( h^{\infty} ) .
\end{equation*}
Since the functional calculus gives
\begin{equation*}
\< x \>^{- s} ( P - z )^{- 1} ( 1 - \widetilde{\varphi} ) ( P ) \< x \>^{- s} = \CO ( 1 ) ,
\end{equation*}
uniformly for $\re z$ near the support of $\varphi$, the previous equation becomes
\begin{equation} \label{a54}
\< x \>^{- s} ( P - z )^{- 1} \< x \>^{- s} = \< x \>^{- s} ( A_{1} + A_{2} + A_{3} ) \< x \>^{- s} ( 1 + \CO ( h^{\infty} ) ) + \CO ( 1 ) .
\end{equation}
Using again \eqref{a55}, one can verify that
\begin{align*}
\big\Vert \< x \>^{- s} A_{1} \< x \>^{- s} \big\Vert &\lesssim h^{- 1} ,  \\
\big\Vert \< x \>^{- s} ( A_{2} + A_{3} ) \< x \>^{- s} \big\Vert &\lesssim \big\Vert \< x \>^{- s} ( Q - \lambda \pm i 0  )^{- 1} \< x \>^{- s} \big\Vert ,
\end{align*}
in the limit $z \to \lambda + i 0$. Since the norm of the weighted resolvent of $Q$ is at least like $h^{- 1}$, \eqref{a54} yields
\begin{equation}
\big\Vert \< x \>^{- s} ( P - \lambda - i 0  )^{- 1} \< x \>^{- s} \big\Vert \lesssim \big\Vert \< x \>^{- s} ( Q - \lambda - i 0  )^{- 1} \< x \>^{- s} \big\Vert ,
\end{equation}
uniformly for $\lambda$ near the support of $\varphi$, which implies Remark \ref{a47}.
\end{proof}

We end this section with the

\begin{proof}[Proof of Proposition \ref{a91}]
We first remark that \ref{h5} is an open condition in the sense that, if \ref{h5} holds true for some interval $I$, then it holds true near $I$. Then, it is enough to prove the proposition near $\lambda_{0} \in I$ assuming \ref{h4} and \ref{h5} near $\lambda_{0}$. Working as in \eqref{a81} and using that the weighted resolvents of $P$ and $Q$ are polynomially bounded thanks to \ref{h4} and Remark \ref{a47}, we can write 
\begin{equation} \label{a94}
s_{P , P_{0}}^{\prime} ( \lambda ) = s_{Q , P_{0}}^{\prime} ( \lambda ) + \tr \big( \chi \big( E^{\prime}_{P} ( \lambda ) - E^{\prime}_{Q} ( \lambda ) \big) \chi \big) + \widetilde{\sigma} ( \lambda ; h ) ,
\end{equation}
uniformly near $\lambda_{0}$ where $\widetilde{\sigma} ( \lambda ; h )$ has a complete asymptotic expansion in $h$ with $C^{\infty}$ coefficients.

Consider now $\one_{\{ \lambda_{0} \}} \prec \varphi \in C^{\infty}_{0} ( \R )$ and $\one_{K_{p} ( \lambda_{0} )} \prec f \in C^{\infty}_{0} ( T^{*} \R^{n} )$. In particular, $E^{\prime}_{P} ( \lambda ) = E^{\prime}_{P} ( \lambda ) \varphi ( P )$ for $\lambda$ near $\lambda_{0}$. If the support of $f$ is sufficiently close to $K_{p} ( \lambda_{0} )$, the pseudodifferential calculus gives
\begin{align*}
\chi \Op ( f ) = \Op ( f ) + \CO ( h^{\infty} ) \< x \>^{- 1} ,
\end{align*}
in trace norm. Then, combining with \ref{h4} and Remark \ref{a47}, we have
\begin{align}
\tr \big( \chi \big( E^{\prime}_{P} ( \lambda ) - E^{\prime}_{Q} ( \lambda ) \big) \chi \big) ={}& \tr \big( \chi \Op ( f ) \big( E^{\prime}_{P} ( \lambda ) - E^{\prime}_{Q} ( \lambda ) \big) \Op ( f ) \chi \big)   \nonumber \\
&+ \tr \big( \chi \Op ( f ) \big( E^{\prime}_{P} ( \lambda ) \varphi ( P ) - E^{\prime}_{Q} ( \lambda ) \varphi ( Q ) \big) \Op ( 1 - f ) \chi \big)   \nonumber  \\
&+ \tr \big( \chi \Op ( 1 - f ) \big( E^{\prime}_{P} ( \lambda ) \varphi ( P ) - E^{\prime}_{Q} ( \lambda ) \varphi ( Q ) \big) \Op ( f ) \chi \big)  \nonumber  \\
&+ \tr \big( \chi \Op ( 1 - f ) \big( E^{\prime}_{P} ( \lambda ) \varphi ( P ) - E^{\prime}_{Q} ( \lambda ) \varphi ( Q ) \big) \Op ( 1 - f ) \chi \big)   \nonumber  \\
={}& \tr \big( \Op ( f ) \big( E^{\prime}_{P} ( \lambda ) - E^{\prime}_{Q} ( \lambda ) \big) \Op ( f ) \big) \nonumber \\
&+ \sum_{\pm} \tr \big( \Op ( k_{\pm}^{p} ) ( P - \lambda \mp i 0 )^{-1} \big)  \label{a93} \\
&+ \sum_{\pm} \tr \big( \Op ( k_{\pm}^{q} ) ( Q - \lambda \mp i 0 )^{-1} \big) + \CO ( h^{\infty} ) , \nonumber
\end{align}
for some $k_{\pm}^{r} \in S ( 1 )$ supported in $\supp ( \chi ( x ) ( 1 - f ) ( x , \xi ) \varphi ( r ) ( x , \xi ) )$.

Using \eqref{a45} and \eqref{a46} with $\chi$ replaced by $\Op ( f )$, which guarantees that all the estimates hold in norm trace, we obtain
\begin{equation} \label{a95}
\tr \big( \Op ( f ) \big( E^{\prime}_{P} ( \lambda ) - E^{\prime}_{Q} ( \lambda ) \big) \Op ( f ) \big) = \CO ( h^{\infty} ) ,
\end{equation}
uniformly for $\lambda$ near $\lambda_{0}$. We now deal with $\tr ( \Op ( k_{+}^{p} ) ( P - \lambda - i 0 )^{-1} )$. Let $\rho \in \supp ( \chi ( 1 - f ) )$ with $p ( \rho ) = \lambda_{0}$. Since $\rho \notin K_{p} ( \lambda_{0} )$, the Hamiltonian trajectory $\exp ( t H_{p} ) ( \rho )$ goes to $\infty$ as $t \to + \infty$ or $t \to - \infty$. By continuity, this is also true in a neighborhood of $\rho$. Thus, using a compactness argument and assuming that the support of $\varphi$ is close enough to $\lambda_{0}$, there exist, for all $\delta > 0$, a finite number of compactly supported symbols $g_{\ell} \prec \widetilde{g}_{\ell} \in S ( 1 )$, $\ell = 1 , \ldots , L$, such that
\begin{equation*}
k_{+}^{p} = \sum_{1 \leq \ell \leq L} g_{\ell} ,
\end{equation*}
and, for all $\rho \in \supp \widetilde{g}_{\ell}$, the curve $\exp ( t H_{p} ) ( \rho )$ goes to $\infty$ as $t \to \pm \infty$ without coming back to the support of $\widetilde{g}_{\ell}$ for $\pm t > \delta$. In particular,
\begin{equation} \label{a99}
\tr \big( \Op ( k_{+}^{p} ) ( P - \lambda - i 0 )^{-1} \big) = \sum_{1 \leq \ell \leq L} \tr \big( \Op ( g_{\ell} ) ( P - \lambda - i 0 )^{-1} \Op ( \widetilde{g}_{\ell} ) \big) + \CO ( h^{\infty} ) ,
\end{equation}
for $\lambda$ near $\lambda_{0}$.

Consider $1 \leq \ell \leq L$ such that the Hamiltonian flow escapes to infinity from the support of $\widetilde{g}_{\ell}$ without coming back in positive time larger than $\delta$. As in \eqref{a49}, we write
\begin{equation} \label{a97}
\tr \big( \Op ( g_{\ell} ) ( P - \lambda - i 0 )^{-1} \Op ( \widetilde{g}_{\ell} ) \big) = G_{\rm small} + G_{\rm large} + G_{\rm infinity} ,
\end{equation}
with
\begin{align*}
G_{\rm small} &= \frac{i}{h} \int_{0}^{\delta} \tr \Big( \Op ( g_{\ell} ) e^{- i s ( P - \lambda ) / h} \Op ( \widetilde{g}_{\ell} ) \Big) \, d s ,  \\
G_{\rm large} &= \frac{i}{h} \int_{\delta}^{S} \tr \Big( \Op ( g_{\ell} ) e^{- i s ( P - \lambda ) / h} \Op ( \widetilde{g}_{\ell} ) \Big) \, d s ,  \\
G_{\rm infinity} &= \tr \Big( \Op ( g_{\ell} ) ( P - \lambda - i 0 )^{-1} e^{- i S ( P - \lambda ) / h} \Op ( \widetilde{g}_{\ell} ) \Big) ,
\end{align*}
where $S > 0$ will be fixed large enough. From Lemma 3.1 of Robert and Tamura \cite{RoTa88_01}, $G_{\rm small}$ has a complete asymptotic expansion in $h$ with $C^{\infty}$ coefficients for $\delta > 0$ small enough. Since no trajectory from $\supp \widetilde{g}_{\ell}$ comes back in time larger than $\delta$, the propagation of singularities gives $G_{\rm large} = \CO ( h^{\infty} )$. Finally, since the trajectories from $\supp \widetilde{g}_{\ell}$ escape to infinity in positive time, $e^{- i S ( P - \lambda ) / h} \Op ( \widetilde{g}_{\ell} )$ is microlocalized in an outgoing region (see \eqref{a96}) for $S$ large enough. Then, applying Lemma 2.3 (iii) of Robert and Tamura \cite{RoTa89_01} (see the discussion above \eqref{a43}), we get $G_{\rm infinity} = \CO ( h^{\infty} )$. Summing up, the left hand side of \eqref{a97} has a complete asymptotic expansion in $h$ with $C^{\infty}$ coefficients, say
\begin{equation} \label{a98}
\tr \big( \Op ( g_{\ell} ) ( P - \lambda - i 0 )^{-1} \Op ( \widetilde{g}_{\ell} ) \big) \asymp \sigma_{0}^{\ell} ( \lambda ) h^{- n} + \sigma_{1}^{\ell} ( \lambda ) h^{1 - n} + \cdots ,
\end{equation}
for $\lambda$ near $\lambda_{0}$. Consider now $1 \leq \ell \leq L$ such that the Hamiltonian flow escapes to infinity from the support of $\widetilde{g}_{\ell}$ without coming back in negative time smaller than $- \delta$. Taking the adjoint, we have
\begin{equation*}
\overline{\tr \big( \Op ( g_{\ell} ) ( P - \lambda - i 0 )^{-1} \Op ( \widetilde{g}_{\ell} ) \big)} = \tr \big( \Op ( \overline{\widetilde{g}_{\ell}} ) ( P - \lambda + i 0 )^{-1} \Op ( \overline{g_{\ell}} ) \big) .
\end{equation*}
Then, working as in \eqref{a97} but in negative time, we obtain also \eqref{a98} in that case.

From \eqref{a99} and \eqref{a98}, the function
\begin{equation} \label{b1}
\tr \big( \Op ( k_{+}^{p} ) ( P - \lambda - i 0 )^{-1} \big) ,
\end{equation}
has a complete asymptotic expansion in $h$ with $C^{\infty}$ coefficients near $\lambda_{0}$. The same holds true for $P$ replaced by $Q$ and $- i 0$ replaced by $+ i 0$. Combining \eqref{a93} with \eqref{a95} and \eqref{b1}, the function $\tr ( \chi ( E^{\prime}_{P} ( \lambda ) - E^{\prime}_{Q} ( \lambda ) ) \chi )$ has a complete asymptotic expansion in $h$ with $C^{\infty}$ coefficients near $\lambda_{0}$. Eventually, Proposition \ref{a91} follows from \eqref{a94}.
\end{proof}

\bibliographystyle{amsplain}
\providecommand{\MRhref}[2]{%
  \href{http://www.ams.org/mathscinet-getitem?mr=#1}{#2}
}
\providecommand{\href}[2]{#2}

%\bibliographystyle{amsplain}
%\bibliography{reso}

\begin{thebibliography}{10}

\bibitem{Ag86_01}
S.~Agmon, \emph{Spectral theory of {S}chr\"{o}dinger operators on {E}uclidean
  and on non-{E}uclidean spaces}, Comm. Pure Appl. Math. \textbf{39} (1986),
  no.~S, suppl., S3--S16, Frontiers of the mathematical sciences: 1985.

\bibitem{AgCo71_01}
J.~Aguilar and J.-M. Combes, \emph{A class of analytic perturbations for
  one-body {S}chr\"odinger {H}amiltonians}, Comm. Math. Phys. \textbf{22}
  (1971), 269--279.

\bibitem{Al06_01}
I.~Alexandrova, \emph{Structure of the short range amplitude for general
  scattering relations}, Asymptotic Analysis (2006), no.~50, 13--30.

\bibitem{AlBoRa08_02}
I.~Alexandrova, J.-F. Bony, and T.~Ramond, \emph{Resolvent and scattering
  matrix at the maximum of the potential}, Serdica Math. J. \textbf{34} (2008),
  no.~1, 267--310.

\bibitem{AlBoRa08_01}
I.~Alexandrova, J.-F. Bony, and T.~Ramond, \emph{Semiclassical scattering amplitude at the maximum of the
  potential}, Asymptot. Anal. \textbf{58} (2008), no.~1-2, 57--125.

\bibitem{BoFuRaZe11_01}
J.-F. Bony, S.~Fujii\'e, T.~Ramond, and M.~Zerzeri, \emph{Spectral projection,
  residue of the scattering amplitude and {S}chr\"{o}dinger group expansion for
  barrier-top resonances}, Ann. Inst. Fourier \textbf{61} (2011), no.~4,
  1351--1406.

\bibitem{BoFuRaZe18_01}
J.-F. Bony, S.~Fujii\'e, T.~Ramond, and M.~Zerzeri, \emph{Resonances for homoclinic trapped sets}, Ast\'{e}risque (2018),
  no.~405, vii+314.

\bibitem{BoFuRaZe19_01}
J.-F. Bony, S.~Fujii\'e, T.~Ramond, and M.~Zerzeri, \emph{Barrier-top resonances for non globally analytic potentials}, J.
  Spectr. Theory \textbf{9} (2019), no.~1, 315--348.

\bibitem{BoPe13_01}
J.-F. Bony and V.~Petkov, \emph{Semiclassical estimates of the cut-off
  resolvent for trapping perturbations}, J. Spectr. Theory \textbf{3} (2013),
  no.~3, 399--422.

\bibitem{BoSj01_01}
J.-F. Bony and J.~Sj\"{o}strand, \emph{Traceformula for resonances in small
  domains}, J. Funct. Anal. \textbf{184} (2001), no.~2, 402--418.

\bibitem{BrCoDu87_02}
P.~Briet, J.-M. Combes, and P.~Duclos, \emph{On the location of resonances for
  {S}chr{\"o}dinger operators in the semiclassical limit {II}: Barrier top
  resonances}, Comm. in Partial Differential Equations \textbf{2} (1987),
  no.~12, 201--222.

\bibitem{BrMa17_01}
P.~Briet and A.~Martinez, \emph{Estimates on the molecular dynamics for the
  predissociation process}, J. Spectr. Theory \textbf{7} (2017), no.~2,
  487--517.

\bibitem{BrMa19_01}
P.~Briet and A.~Martinez, \emph{Molecular dynamics at an energy-level crossing}, J. Differential
  Equations \textbf{267} (2019), no.~10, 5662--5700.

\bibitem{BrPe03_01}
V.~Bruneau and V.~Petkov, \emph{Meromorphic continuation of the spectral shift
  function}, Duke Math. J. \textbf{116} (2003), no.~3, 389--430.

\bibitem{Bu98_01}
N.~Burq, \emph{D\'{e}croissance de l'\'{e}nergie locale de l'\'{e}quation des
  ondes pour le probl\`eme ext\'{e}rieur et absence de r\'{e}sonance au
  voisinage du r\'{e}el}, Acta Math. \textbf{180} (1998), no.~1, 1--29.

\bibitem{Bu02_02}
N.~Burq, \emph{Lower bounds for shape resonances widths of long range
  {S}chr\"{o}dinger operators}, Amer. J. Math. \textbf{124} (2002), no.~4,
  677--735.

\bibitem{BuZw01_01}
N.~Burq and M.~Zworski, \emph{Resonance expansions in semi-classical
  propagation}, Comm. Math. Phys. \textbf{223} (2001), no.~1, 1--12.

\bibitem{CaMaRa05_01}
C.~Cancelier, A.~Martinez, and T.~Ramond, \emph{Quantum resonances without
  analyticity}, Asymptot. Anal. \textbf{44} (2005), no.~1-2, 47--74.

\bibitem{CaPoVo04_01}
F.~Cardoso, G.~Popov, and G.~Vodev, \emph{Semi-classical resolvent estimates
  for the {S}chr\"{o}dinger operator on non-compact complete {R}iemannian
  manifolds}, Bull. Braz. Math. Soc. \textbf{35} (2004), no.~3, 333--344.

\bibitem{CaVo02_01}
F.~Cardoso and G.~Vodev, \emph{Uniform estimates of the resolvent of the
  {L}aplace-{B}eltrami operator on infinite volume {R}iemannian manifolds.
  {II}}, Ann. Henri Poincar\'{e} \textbf{3} (2002), no.~4, 673--691.

\bibitem{Da14_01}
K.~Datchev, \emph{Quantitative limiting absorption principle in the
  semiclassical limit}, Geom. Funct. Anal. \textbf{24} (2014), no.~3, 740--747.

\bibitem{DaDyZw15_01}
K.~Datchev, S.~Dyatlov, and M.~Zworski, \emph{Resonances and lower resolvent
  bounds}, J. Spectr. Theory \textbf{5} (2015), no.~3, 599--615.

\bibitem{DaVa12_02}
K.~Datchev and A.~Vasy, \emph{Gluing semiclassical resolvent estimates via
  propagation of singularities}, Int. Math. Res. Not. IMRN (2012), no.~23,
  5409--5443.

\bibitem{DaVa13_01}
K.~Datchev and A.~Vasy, \emph{Semiclassical resolvent estimates at trapped sets}, Ann. Inst.
  Fourier \textbf{62} (2012), no.~6, 2379--2384 (2013).

\bibitem{DiPe03_01}
M.~Dimassi and V.~Petkov, \emph{Spectral shift function and resonances for
  non-semi-bounded and {S}tark {H}amiltonians}, J. Math. Pures Appl. (9)
  \textbf{82} (2003), no.~10, 1303--1342.

\bibitem{DiSj99_01}
M.~Dimassi and J.~Sj{\"o}strand, \emph{Spectral asymptotics in the
  semi-classical limit}, London Mathematical Society Lecture Note Series, vol.
  268, Cambridge University Press, 1999.

\bibitem{DoMcTh66_01}
C.~L. Dolph, J.~B. McLeod, and D.~Thoe, \emph{The analytic continuation of the
  resolvent kernel and scattering operator associated with the {S}chroedinger
  operator}, J. Math. Anal. Appl. \textbf{16} (1966), 311--332.

\bibitem{Do94_01}
S.~Dozias, \emph{Op\'{e}rateurs h-pseudodiff\'{e}rentiels \`{a} flot
  p\'{e}riodique}, Ph. D. Thesis., Universit\'{e} Paris Nord, 1994.

\bibitem{Dy12_01}
S.~Dyatlov, \emph{Asymptotic distribution of quasi-normal modes for {K}err--de
  {S}itter black holes}, Ann. Henri Poincar\'e \textbf{13} (2012), no.~5,
  1101--1166.

\bibitem{DyZw16_01}
S.~Dyatlov and M.~Zworski, \emph{Mathematical theory of scattering resonances},
  preprint at {\tt http://math.} {\tt berkeley.edu/$\sim$zworski/} (2019),
  1--640.

\bibitem{FeLa90_01}
C.~Fern\'{a}ndez and R.~Lavine, \emph{Lower bounds for resonance widths in
  potential and obstacle scattering}, Comm. Math. Phys. \textbf{128} (1990),
  no.~2, 263--284.

\bibitem{FuLaMa11_01}
S.~Fujii{\'e}, A.~Lahmar-Benbernou, and A.~Martinez, \emph{Width of shape
  resonances for non globally analytic potentials}, J. Math. Soc. Japan
  \textbf{63} (2011), no.~1, 1--78.

\bibitem{FuRa03_01}
S.~Fujii{\'e} and T.~Ramond, \emph{Breit-{W}igner formula at barrier tops}, J.
  Math. Phys. \textbf{44} (2003), no.~5, 1971--1983.

\bibitem{GeMa89_02}
C.~G{\'e}rard and A.~Martinez, \emph{Prolongement m\'eromorphe de la matrice de
  scattering pour des probl\`emes \`a deux corps \`a longue port\'ee}, Ann.
  Inst. H. Poincar\'e Phys. Th\'eor. \textbf{51} (1989), no.~1, 81--110.

\bibitem{GeMaRo89_01}
C.~G\'{e}rard, A.~Martinez, and D.~Robert, \emph{Breit-{W}igner formulas for
  the scattering phase and the total scattering cross-section in the
  semi-classical limit}, Comm. Math. Phys. \textbf{121} (1989), no.~2,
  323--336.

\bibitem{GeSi92_01}
C.~G{\'e}rard and I.~Sigal, \emph{Space-time picture of semiclassical
  resonances}, Comm. Math. Phys. \textbf{145} (1992), no.~2, 281--328.

\bibitem{HeMa87_01}
B.~Helffer and A.~Martinez, \emph{Comparaison entre les diverses notions de
  r\'esonances}, Helv. Phys. Acta \textbf{60} (1987), no.~8, 992--1003.

\bibitem{HeSj85_01}
B.~Helffer and J.~Sj{\"o}strand, \emph{Multiple wells in the semiclassical
  limit. {III}. {I}nteraction through nonresonant wells}, Math. Nachr.
  \textbf{124} (1985), 263--313.

\bibitem{HeSj86_01}
B.~Helffer and J.~Sj{\"o}strand, \emph{R\'esonances en limite semi-classique}, M\'em. Soc. Math. France
  (1986), no.~24-25, iv+228.

\bibitem{Ho90_01}
L.~H{\"o}rmander, \emph{The analysis of linear partial differential operators.
  {I}}, second ed., Grundlehren der Mathematischen Wissenschaften, vol. 256,
  Springer, 1990, Distribution theory and Fourier analysis.

\bibitem{Ho94_02}
L.~H{\"o}rmander, \emph{The analysis of linear partial differential operators. {III}},
  Grundlehren der Mathematischen Wissenschaften, vol. 274, Springer, 1994,
  Pseudo-differential operators.

\bibitem{Ho94_01}
L.~H{\"o}rmander, \emph{The analysis of linear partial differential operators. {IV}},
  Grundlehren der Mathematischen Wissenschaften, vol. 275, Springer, 1994,
  Fourier integral operators.

\bibitem{Hu86_01}
W.~Hunziker, \emph{Distortion analyticity and molecular resonance curves}, Ann.
  Inst. H. Poincar\'e Phys. Th\'eor. \textbf{45} (1986), no.~4, 339--358.

\bibitem{IsKi85_01}
H.~Isozaki and H.~Kitada, \emph{Modified wave operators with time-independent
  modifiers}, J. Fac. Sci. Univ. Tokyo Sect. IA Math. \textbf{32} (1985),
  no.~1, 77--104.

\bibitem{IsKi86_01}
H.~Isozaki and H.~Kitada, \emph{Scattering matrices for two-body {S}chr\"odinger operators},
  Sci. Papers College Arts Sci. Univ. Tokyo \textbf{35} (1986), no.~2, 81--107.

\bibitem{Iv98_01}
V.~Ivrii, \emph{Microlocal analysis and precise spectral asymptotics}, Springer
  Monographs in Mathematics, Springer, 1998.

\bibitem{JeNe06_01}
A.~Jensen and G.~Nenciu, \emph{The {F}ermi golden rule and its form at
  thresholds in odd dimensions}, Comm. Math. Phys. \textbf{261} (2006), no.~3,
  693--727.

\bibitem{Be99_01}
A.~Lahmar-Benbernou, \emph{Estimation des r\'esidus de la matrice de diffusion
  associ\'es \`a des r\'esonances de forme. {I}}, Ann. Inst. H. Poincar\'e
  Phys. Th\'eor. \textbf{71} (1999), no.~3, 303--338.

\bibitem{LaMa99_01}
A.~Lahmar-Benbernou and A.~Martinez, \emph{Semiclassical asymptotics of the
  residues of the scattering matrix for shape resonances}, Asymptot. Anal.
  \textbf{20} (1999), no.~1, 13--38.

\bibitem{LaMa02_01}
A.~Lahmar-Benbernou and A.~Martinez, \emph{On
  {H}elffer-{S}j\"{o}strand's theory of resonances}, Int. Math. Res. Not.
  (2002), no.~13, 697--717.

\bibitem{LaPh67_01}
P.~Lax and R.~Phillips, \emph{Scattering theory}, Pure and Applied Mathematics,
  Vol. 26, Academic Press, 1967.

\bibitem{Ma02_02}
A.~Martinez, \emph{An introduction to semiclassical and microlocal analysis},
  Universitext, Springer, 2002.

\bibitem{MaRaSj09_01}
A.~Martinez, T.~Ramond, and J.~Sj\"{o}strand, \emph{Resonances for nonanalytic
  potentials}, Anal. PDE \textbf{2} (2009), no.~1, 29--60.

\bibitem{Me88_01}
R.~Melrose, \emph{Weyl asymptotics for the phase in obstacle scattering}, Comm.
  Partial Differential Equations \textbf{13} (1988), no.~11, 1431--1439.

\bibitem{Mi04_01}
L.~Michel, \emph{Semi-classical behavior of the scattering amplitude for
  trapping perturbations at fixed energy}, Canad. J. Math. \textbf{56} (2004),
  no.~4, 794--824.

\bibitem{Na99_01}
S.~Nakamura, \emph{Spectral shift function for trapping energies in the
  semiclassical limit}, Comm. Math. Phys. \textbf{208} (1999), no.~1, 173--193.

\bibitem{NaStZw03_01}
S.~Nakamura, P.~Stefanov, and M.~Zworski, \emph{Resonance expansions of
  propagators in the presence of potential barriers}, J. Funct. Anal.
  \textbf{205} (2003), no.~1, 180--205.

\bibitem{Or90_01}
A.~Orth, \emph{Quantum mechanical resonance and limiting absorption: the many
  body problem}, Comm. Math. Phys. \textbf{126} (1990), no.~3, 559--573.

\bibitem{PeZw99_01}
V.~Petkov and M.~Zworski, \emph{Breit-{W}igner approximation and the
  distribution of resonances}, Comm. Math. Phys. \textbf{204} (1999), no.~2,
  329--351.

\bibitem{PeZw00_01}
V.~Petkov and M.~Zworski, \emph{Erratum: ``{B}reit-{W}igner approximation and the distribution
  of resonances''}, Comm. Math. Phys. \textbf{214} (2000), no.~3, 733--735.

\bibitem{PeZw01_01}
V.~Petkov and M.~Zworski, \emph{Semi-classical estimates on the scattering determinant}, Ann.
  Henri Poincar\'{e} \textbf{2} (2001), no.~4, 675--711.

\bibitem{Ra96_01}
T.~Ramond, \emph{Semiclassical study of quantum scattering on the line}, Comm.
  Math. Phys. \textbf{177} (1996), no.~1, 221--254.

\bibitem{Ro87_01}
D.~Robert, \emph{Autour de l'approximation semi-classique}, Progress in
  Mathematics, vol.~68, Birkh\"auser, 1987.

\bibitem{Ro94_01}
D.~Robert, \emph{Relative time-delay for perturbations of elliptic operators and
  semiclassical asymptotics}, J. Funct. Anal. \textbf{126} (1994), no.~1,
  36--82.

\bibitem{Ro99_01}
D.~Robert, \emph{Semiclassical asymptotics for the spectral shift function},
  Differential operators and spectral theory, Amer. Math. Soc. Transl. Ser. 2,
  vol. 189, Amer. Math. Soc., 1999, pp.~187--203.

\bibitem{RoTa87_01}
D.~Robert and H.~Tamura, \emph{Semiclassical estimates for resolvents and
  asymptotics for total scattering cross-sections}, Ann. Inst. H. Poincar\'e
  Phys. Th\'eor. \textbf{46} (1987), no.~4, 415--442.

\bibitem{RoTa88_01}
D.~Robert and H.~Tamura, \emph{Semi-classical asymptotics for local spectral densities and time
  delay problems in scattering processes}, J. Funct. Anal. \textbf{80} (1988),
  no.~1, 124--147.

\bibitem{RoTa89_01}
D.~Robert and H.~Tamura, \emph{Asymptotic behavior of scattering amplitudes in semi-classical
  and low energy limits}, Ann. Inst. Fourier \textbf{39} (1989), no.~1,
  155--192.

\bibitem{Ro10_01}
J.~Royer, \emph{Limiting absorption principle for the dissipative {H}elmholtz
  equation}, Comm. Partial Differential Equations \textbf{35} (2010), no.~8,
  1458--1489.

\bibitem{Sj87_01}
J.~Sj{\"o}strand, \emph{Semiclassical resonances generated by nondegenerate
  critical points}, Pseudodifferential operators (Oberwolfach, 1986), Lecture
  Notes in Math., vol. 1256, Springer, 1987, pp.~402--429.

\bibitem{Sj07_01}
J.~Sj{\"o}strand, \emph{Lectures on resonances}, preprint at {\tt
  http://sjostrand.perso.math.cnrs.fr/} (2007), 1--169.

\bibitem{SjZw91_01}
J.~Sj{\"o}strand and M.~Zworski, \emph{Complex scaling and the distribution of
  scattering poles}, J. Amer. Math. Soc. \textbf{4} (1991), no.~4, 729--769.

\bibitem{SoWe98_01}
A.~Soffer and M.~Weinstein, \emph{Time dependent resonance theory}, Geom.
  Funct. Anal. \textbf{8} (1998), no.~6, 1086--1128.

\bibitem{St01_01}
P.~Stefanov, \emph{Resonance expansions and {R}ayleigh waves}, Math. Res. Lett.
  \textbf{8} (2001), no.~1-2, 107--124.

\bibitem{St02_01}
P.~Stefanov, \emph{Estimates on the residue of the scattering amplitude}, Asymptot.
  Anal. \textbf{32} (2002), no.~3-4, 317--333.

\bibitem{TaZw00_01}
S.-H. Tang and M.~Zworski, \emph{Resonance expansions of scattered waves},
  Comm. Pure Appl. Math. \textbf{53} (2000), no.~10, 1305--1334.

\bibitem{Va89_01}
B.~Va{\u\i}nberg, \emph{Asymptotic methods in equations of mathematical
  physics}, Gordon \& Breach Science Publishers, 1989, Translated from the
  Russian by E. Primrose.

\bibitem{Vo02_01}
G.~Vodev, \emph{Uniform estimates of the resolvent of the {L}aplace--{B}eltrami
  operator on infinite volume {R}iemannian manifolds with cusps}, Comm. Partial
  Differential Equations \textbf{27} (2002), no.~7-8, 1437--1465.

\bibitem{Ya92_01}
D.~Yafaev, \emph{Mathematical scattering theory}, Translations of Mathematical
  Monographs, vol. 105, American Mathematical Society, 1992, General theory,
  Translated from the Russian by J. R. Schulenberger.

\bibitem{Zw12_01}
M.~Zworski, \emph{Semiclassical analysis}, Graduate Studies in Mathematics,
  vol. 138, American Mathematical Society, 2012.

\end{thebibliography}

\end{document}